\documentclass[a4paper]{amsart}

\usepackage[backrefs,lite]{amsrefs}

\usepackage{amssymb}              
\usepackage{lmodern,bm}              
\usepackage{stmaryrd}
\usepackage{mathrsfs}

\makeatletter
\DeclareFontEncoding{LS1}{}{}
\DeclareFontSubstitution{LS1}{stix}{m}{n}
\DeclareMathAlphabet{\skr}{LS1}{stixscr}{m}{n}
\makeatother

\usepackage[all]{xy}
\usepackage{tikz}

\usetikzlibrary{calc}
\usetikzlibrary{decorations.markings}
\usetikzlibrary{positioning,arrows}

\usepackage{longtable}

\CompileMatrices

\newtheorem{theorem}{Theorem}[section]

\newtheorem{lemma}[theorem]{Lemma}
\newtheorem{proposition}[theorem]{Proposition}
\newtheorem{corollary}[theorem]{Corollary}

\theoremstyle{definition}

\newtheorem{definition}[theorem]{Definition}
\newtheorem{construction}[theorem]{Construction}

\newtheorem{example}[theorem]{Example}
\newtheorem{remark}[theorem]{Remark}

\newtheorem{summary}[theorem]{Summary}

\newtheorem{algorithm}[theorem]{Algorithm}

\numberwithin{equation}{theorem}

\def\vector2#1#2{\left(\begin{array}{c} #1 \\ #2 \end{array}\right)}
\def\Cl{{\rm Cl}}

\def\LL{{\mathbb L}}

\def\KK{{\mathbb K}}
\def\TT{{\mathbb T}}
\def\ZZ{{\mathbb Z}}
\def\RR{{\mathbb R}}

\def\QQ{{\mathbb Q}}
\def\PP{{\mathbb P}}

\def\div{{\rm div}}
\def\quot{/\!\!/}
\def\mal{\! \cdot \!}

\def\conv{{\rm conv}}

\def\im{{\rm im}}

\def\bangle#1{{\langle #1 \rangle}}

\def\Aut{{\rm Aut}}

\def\GL{{\rm GL}}

\def\WDiv{{\rm WDiv}}
\def\GL{{\rm GL}}

\def\Pic{{\rm Pic}}

\def\cone{{\rm cone}}

\def\pr{{\rm pr}}
\def\id{{\rm id}}

\def\lcm{{\rm lcm}}
\def\im{{\rm im}}

\def\trop{{\rm trop}}

\DeclareMathOperator{\image}{\mathrm{im}}
\DeclareMathOperator{\Eff}{\mathrm{Eff}}
\DeclareMathOperator{\Mov}{\mathrm{Mov}}
\DeclareMathOperator{\Ample}{\mathrm{Ample}}
\DeclareMathOperator{\SAmple}{\mathrm{SAmple}}

\DeclareMathOperator{\vol}{\mathrm{vol}}

\title[Log del Pezzo surfaces with torus action]{Classifying log del Pezzo surfaces \\ with torus action}

\author[Daniel H\"attig, J\"urgen Hausen, Justus Springer]{Daniel H\"attig, J\"urgen Hausen, Justus Springer}

\address{Mathematisches Institut, Universit\"at T\"ubingen,
Auf der Morgenstelle 10, 72076 T\"ubingen, Germany}
\email{daniel.haettig@uni-tuebingen.de}

\address{Mathematisches Institut, Universit\"at T\"ubingen,
Auf der Morgenstelle 10, 72076 T\"ubingen, Germany}
\email{juergen.hausen@uni-tuebingen.de}

\address{Mathematisches Institut, Universit\"at T\"ubingen,
Auf der Morgenstelle 10, 72076 T\"ubingen, Germany}
\email{justus.springer@uni-tuebingen.de}


\begin{document}

\begin{abstract}
We consider log del Pezzo surfaces coming with a
non-trivial torus action.
Such a surface is $1/k$-log canonical if it allows
a resolution of singularities with discrepanies
all greater or equal to $1/k-1$.
We provide a concrete classification algorithm
for fixed $k$ and give explicit results for
$k=1$, $2$ and $3$.
This comprises in particular the cases
of Gorenstein index $1$, $2$ and $3$.
\end{abstract}

\maketitle

\section{Introduction}

A~\emph{log del Pezzo surface} is a del Pezzo
surface $X$ with at most log terminal singularities,
that means that all discrepancies of some resolution of
singularities $X' \to X$ are greater than $-1$.
The log del Pezzo surfaces are all
rational~\cite[Prop.~3.6]{Nak}
and they form an intensely studied class of rational
surfaces.
Without posing further conditions on the singularities,
the class of log del Pezzo surfaces is unbounded.
Alexeev~\cite{Al} observed in 1994 that for any given
$\varepsilon > 0$ the class of \emph{$\varepsilon$-log canonical}
del Pezzo surfaces, i.e. those admitting only discrepancies
at most equal to $\varepsilon-1$, is bounded.

In the present article we focus on log del Pezzo
surfaces that come with a non-trivial effective torus
action.
This class is much more accessible than the general one
but still unbounded and rich enough to serve as an interesting
example class and testing ground.
Our aim is to provide explicit and feasible classification
procedures for the cases $\varepsilon = 1/k$ with a positive
integer $k$.
We distinguish between the toric log del Pezzo surfaces,
where we have a two-dimensional acting torus
$\KK^* \times \KK^*$ and the non-toric log del Pezzo surfaces
coming with a $\KK^*$-action.
We provide explicit results for the cases $k=1,2,3$,
where the case $k=1$ is part of $k=2$ etc..
The following gives a concrete impression.

\begin{theorem}
\label{introthm1}
There are $129\,904$ families of $\frac{1}{3}$-log canonical
del Pezzo $\KK^*$-sur\-faces, $48\,032$ of them toric
and $81\,872$ non-toric.
According to the Picard number $\rho$, the numbers of families
are distributed as follows:

\begin{center}
\begin{tabular}{c|c|c|c|c|c|c|c|c|c|c|c|c}
\scriptsize
$\rho$
&
\scriptsize
$1$
&
\scriptsize
$2$
&
\scriptsize
$3$
&
\scriptsize
$4$
&
\scriptsize
$5$
&
\scriptsize
$6$
& 
\scriptsize
$7$
&
\scriptsize
$8$
&
\scriptsize
$9$
&
\scriptsize
$10$
&
\scriptsize
$11$
&
\scriptsize
$12$ 
\\
\hline
\scriptsize
toric
&
\scriptsize
$355$
&
\scriptsize
$3983$
&
\scriptsize
$13454$
&
\scriptsize
$17791$
&
\scriptsize
$9653$
&
\scriptsize
$2465$
&
\scriptsize
$292$
&
\scriptsize
$37$
&
\scriptsize
$1$
&
\scriptsize
$1$
&
---
&
--- 
\\
\hline
\scriptsize
non-toric
&
\scriptsize
$318$
&
\scriptsize
$3632$
&
\scriptsize
$15174$
&
\scriptsize
$26654$
&
\scriptsize
$22648$
&
\scriptsize
$10024$
&
\scriptsize
$2731$
&
\scriptsize
$577$
&
\scriptsize
$100$
&
\scriptsize
$11$
&
\scriptsize
$2$
&
\scriptsize
$1$
\\
\hline
\scriptsize
in total
&
\scriptsize
$673$
&
\scriptsize
$7615$
&
\scriptsize
$28628$
&
\scriptsize
$44445$
&
\scriptsize
$32301$
&
\scriptsize
$12489$
&
\scriptsize
$3023$
&
\scriptsize
$614$
&
\scriptsize
$101$
&
\scriptsize
$12$
&
\scriptsize
$2$
&
\scriptsize
$1$
\end{tabular} 
\end{center}
\end{theorem}

Every log del Pezzo $\KK^*$-surface of Gorenstein
index~$\iota$ is $1/k$-log canonical for all
$k \ge \iota$.
Thus, our classification comprises in particular
Huggenberger's classification of the Gorenstein
del Pezzo $\KK^*$-surfaces~\cite[Sec.~5.6]{Hug},
see also~\cite{IlMiTr}. Note that due
to~\cite[Thm.~2.2]{HiWa}, Gorenstein del Pezzo
surfaces are automatically log terminal.
Moreover, we find all log-del Pezzo $\KK^*$-surfaces of
Gorenstein indices $\iota = 2, 3$ within our classification
and thus complement S\"u\ss's results~\cite{Su},
obtained on the case of Picard number one in this
setting.

\begin{corollary}
\label{introcor1}
The non-zero numbers of families of log del Pezzo $\KK^*$-surfaces
of Gorenstein index $\iota = 1,2,3$ and given Picard number
$\rho$ are given as follows:
\begin{center}
\begin{tabular}{c|c|c|c|c|c|c|c|c|c|c|c|c}
\scriptsize
$\rho$
&
\scriptsize
$1$
&
\scriptsize
$2$
&
\scriptsize
$3$
&
\scriptsize
$4$
&
\scriptsize
$5$
&
\scriptsize
$6$
& 
\scriptsize
$7$
&
\scriptsize
$8$
&
\scriptsize
$9$
&
\scriptsize
$10$
&
\scriptsize
$11$
&
\scriptsize
$12$ 
\\
\hline
\scriptsize
$\iota=1$
&
\scriptsize
$18$
&
\scriptsize
$19$
&
\scriptsize
$10$
&
\scriptsize
$3$
&
---
&
---
&
---
&
---
&
---
&
---
&
---
&
--- 
\\
\hline
\scriptsize
$\iota=2$
&
\scriptsize
$17$
&
\scriptsize
$36$
&
\scriptsize
$30$
&
\scriptsize
$16$
&
\scriptsize
$7$
&
\scriptsize
$4$
&
\scriptsize
$1$
&
\scriptsize
$1$
&
---
&
---
&
---
&
---
\\
\hline
\scriptsize
$\iota=3$
&
\scriptsize
$54$
&
\scriptsize
$162$
&
\scriptsize
$185$
&
\scriptsize
$105$
&
\scriptsize
$59$
&
\scriptsize
$32$
&
\scriptsize
$16$
&
\scriptsize
$9$
&
\scriptsize
$5$
&
\scriptsize
$2$
&
\scriptsize
$1$
&
\scriptsize
$1$
\end{tabular} 
\end{center}
\end{corollary}

Further data on the classification will be presented in the
later sections and, more importantly, we provide a searchable
database~\cite{HaHaSp} containing the defining data of all
the surfaces and important invariants, such as Gorenstein
index, anticanonical self intersection and much more.
For instance, storing the minimal discrepancies of the canonical
resolutions allows us to figure out the \emph{jumping places},
that means the $1 \le \alpha \le 3$, at which the number of
$\alpha^{-1}$-log canonical surfaces gets incremented.

\goodbreak

\begin{corollary}
\label{introcor2}
The function $\QQ_{\ge 1} \to \ZZ$ sending $\alpha$
to the number $\nu(\alpha)$ of families of $\alpha^{-1}$-log
canonical del Pezzo surfaces with torus action has precisely
$32$ jumping places in the intervall $[1,3]$, as listed the
table below:

\begin{center}  
\begin{tabular}{c|c|c|c|c|c|c|c|c}
\scriptsize $\alpha$
  &
    \scriptsize $1$
  &
    \scriptsize $3/2$
  &
    \scriptsize $5/3$
  &
    \scriptsize $7/4$
  &
    \scriptsize $9/5$
  &
    \scriptsize $11/6$
  &
    \scriptsize $13/7$
  &
    \scriptsize $15/8$
\\
\hline
\scriptsize $\nu(\alpha)$ & \scriptsize $50$ & \scriptsize $117$ & \scriptsize $223$ & \scriptsize $290$ & \scriptsize $320$ & \scriptsize $333$ & \scriptsize $337$ & \scriptsize $338$
\\
\hline  
\hline                                                                        
\scriptsize $\alpha$ & \scriptsize $2$ & \scriptsize $11/5$ & \scriptsize $7/3$ & \scriptsize $17/7$ & \scriptsize $5/2$ & \scriptsize $23/9$ & \scriptsize $13/5$ & \scriptsize $8/3$,
\\
\hline  
\scriptsize $\nu(\alpha)$ & \scriptsize $1742$ & \scriptsize $1896$ & \scriptsize $3777$ & \scriptsize $3812$ & \scriptsize $10889$ & \scriptsize $10902$ & \scriptsize $13910$ &
\scriptsize $14185$
\\
\hline  
\hline
  \scriptsize $\alpha$ & \scriptsize $19/7$ & \scriptsize $11/4$ & \scriptsize $25/9$ & \scriptsize $14/5$ & \scriptsize $31/11$ & \scriptsize $17/6$ & \scriptsize $37/13$ & \scriptsize $43/15$
  \\
  \hline
  \scriptsize $\nu(\alpha)$ & \scriptsize $15173$ & \scriptsize $23095$ & \scriptsize $23516$ & \scriptsize $23518$ & \scriptsize $23679$ & \scriptsize $26357$ & \scriptsize $26407$ & \scriptsize $26416$
\\
\hline
\hline
\scriptsize $\alpha$ & \scriptsize $23/8$ & \scriptsize $49/17$ & \scriptsize $29/10$ & \scriptsize $35/12$ & \scriptsize $41/14$ & \scriptsize $47/16$ & \scriptsize $53/18$ & \scriptsize $3$
\\
\hline
\scriptsize $\nu(\alpha)$ & \scriptsize $27392$ & \scriptsize $27393$ & \scriptsize $27743$ & \scriptsize $27852$ & \scriptsize $27880$ & \scriptsize $27887$ & \scriptsize $27889$ & \scriptsize $129904$                                                                                                                                                                               
\end{tabular} 
\end{center}
\end{corollary}

The basic idea behind our classification is to build
up algorithmically the class of all $1/k$-log canonical
del Pezzo $\KK^*$-surfaces from a suitable smaller
and more accessible starting class.
We now outline the procedures in the toric and and
in the non-toric case.

In the toric case, we benefit from the well-known
correspondence between toric del Pezzo surfaces
and two-dimensional lattice polygons primitive
vertices having the origin as an interior point.
Our starting class consists of $\PP_1 \times \PP_1$
together with the weighted projective planes $\PP_{1,1,\alpha}$,
where $\alpha = 1, \ldots, 2k$.
The $1/k$-log canonicity of these surfaces is reflected
by the fact that the corresponding lattice polygons
\[
\conv(\pm e_1, \, \pm e_2),
\qquad
\conv(-\alpha e_1-e_2, \, e_1, \, e_2), 
\quad
\alpha = 1, \ldots, 2k
\]
are \emph{almost $k$-hollow}, i.e. have only the origin
as a $k$-fold interior lattice point.
We show that any almost $k$-hollow lattice polygon
after applying a suitable unimodular
transformation contains one from the above starting
class and moreover fits into a disk of radius
\[
R \ = \ k^2 \sqrt{4k^2+4k+2}.
\]
This allows us to build up almost $k$-hollow lattice
polygons from the starting class by successively
adding new vertices taken from the disk of radius $R$.
A suitable refining of the searching space for
the new vertices speeds up the procedure and, after
removing redundant entries, we obtain a list of
representatives for almost $k$-hollow lattice polygons
up to unimodular equivalence; see Algorithm~\ref{algo:khollow}.

The case of non-toric log del Pezzo $\KK^*$-surfaces
is the major effort.
Here, the approach via the Cox ring initiated
in~\cite{HaSu,HaHe}
provides us with a concrete working environment
allowing in particular explicit implementation.
Similar as with the lattice polygon for the
toric del Pezzo surface, we can also in the
non-toric case encode  $1/k$-log
canonicity in a combinatorial way, this time via
the \emph{anticanonical complex} introduced
in~\cite{BeHaHuNi}, see also~\cite{HiWr,HaMaWr}.
For producing a suitable starting class in
the classification process, we work
with the following simple, raw minimality concept:
a normal complete surface $X$ is
\emph{combinatorially minimal}
if it does not allow any normal birational contraction
of a curve; see~\cite{Ha} and also~\cite{KeLe} for
a recent application.
In order to see that combinatorial minimality is
suitable for our purposes, we first want to see
that every $1/k$-log canonical del Pezzo
$\KK^*$-surface $X$ can be contracted to a
minimal one by a sequence
$$X = X_1 \to \ldots \to X_q = X_{\min}$$
of explicitly understandable contractions
$X_i \to X_{i+1}$ of rational $1/k$-log canonical
del Pezzo $\KK^*$-surfaces.
This is done in Propositions~\ref{prop:kstarcontractdecomp}
to~\ref{prop:dplusdminuscontract}.
It turns out that, even when starting with a non-toric
$\KK^*$-surface, we may end up with a toric surface
$X_{\min}$. In this situation, the toric classification
performed before provides in particular all possible $X_{\min}$.
The non-toric combinatorially minimal $1/k$-log
canonical del Pezzo $\KK^*$-surfaces are classified
in Section~\ref{sec:combmin}.
For any $k$, Theorem~\ref{thm:surfclass1} provides
bounds for the defining data of the non-toric
combinatorially minimal $1/k$-log canonical
del Pezzo $\KK^*$-surfaces.
Moreover, in Theorem~\ref{thm:combminclass}, we
perform the explicit classification of the combinatorially
minimal case for $k=1,2,3$.
Finally, in Section~\ref{sec:classall}, we produce a suitable starting
class by merging the toric and the non-toric combinatorially
minimal ones and produce bounds for the building process;
see Proposition~\ref{prop:toricrep} and
Theorem~\ref{thm:kstar-effective-bounds}.
Algorithm~\ref{algo:classall} then is able to produce
the full classification of all non-toric $1/k$-log
canonical del Pezzo $\KK^*$-surfaces.
Theorem~\ref{thm:kstardelpezzo} presents the key data for $k=1,2,3$.

Besides establishing the classification procedures,
our aim is to provide a consistent presentation of the necessary
theory on rational surfaces with torus action, taking into
account also introductory aspects. Parts of the basic
material can be found as well spread over the literature,
see for instance~\cite[Sec.~5.4]{ArDeHaLa} and~\cite{Hug}.
For convenience, we develop the theory in coherent notation
and also provide the full proofs for some of the central facts.

\tableofcontents


\section{A quick reminder on toric geometry}
\label{section:toric-varieties}

We briefly gather general background from toric geometry
that will be used in the subsequent sections. As detailed
introductory texts, we refer to \cite{CoLiSc,Dan,Ful}.
Toric geometry, initiated by Demazure's work \cite{Dem}
on the Cremona group, connects algebraic geometry with
combinatorics and has rapidly become a rich and intensely
studied interplay of these disciplines.

The objects on the side of algebraic geometry
are the \emph{standard $n$-torus}, i.e. the $n$-fold
direct product $\mathbb{T}^{n} = \mathbb{K}^* \times
\ldots \times \mathbb{K}^*$, and \emph{toric varieties},
i.e. open embeddings $\mathbb{T}^n \subseteq Z$ into
normal varieties $Z$ such that the multiplication on
$\mathbb{T}^{n} \times \mathbb{T}^{n} \to \mathbb{T}^{n}$
extends to a morphical action $\mathbb{T}^{n} \times Z
\to Z$. A \emph{toric morphism} between toric varieties
$\mathbb{T}^{m} \subseteq Y$ and $\mathbb{T}^{n}
\subseteq Z$ is a morphism $Y \to Z$ that restricts to a
homomorphism $\mathbb{T}^{m} \to \mathbb{T}^{n}$ of tori.
Usually we just write $Z$, $Y$, etc. for toric varieties
and $\varphi \colon Y \to Z$, etc. for toric morphisms.

On the combinatorial side we use the following
terminology. A \emph{lattice fan} in $\mathbb{Z}^{n}$ is
a finite collection $\Sigma$ of pointed, convex,
polyhedral cones in $\mathbb{Q}^{n}$ such that for any
$\sigma \in \Sigma$ every face $\tau \preccurlyeq \sigma$
belongs to $\Sigma$ and any two
$\sigma, \sigma' \in \Sigma$ intersect in a common face.
A \emph{map of lattice fans} from a fan $\Sigma$ in
$\mathbb{Z}^{n}$ to a fan $\Delta$ in $\mathbb{Z}^{m}$
is a homomorphism
$F \colon \mathbb{Z}^{n} \to \mathbb{Z}^{m}$ such that
for every $\sigma \in \Sigma$ there is a
$\tau \in \Delta$ with $F(\sigma) \subseteq \tau$. We
also write $F$ for the linear map
$\mathbb{Q}^{n} \to \mathbb{Q}^{m}$ extending the
homomorphism $\mathbb{Z}^{n} \to \mathbb{Z}^{m}$.

We present the basic construction of toric
geometry, associating a toric variety with an
arbitrary lattice fan. This construction is
functorial and the fundamental theorem of toric
geometry tells us that it even sets up a covariant
equivalence between the category of lattice fans
and the category of toric varieties.

\begin{construction}
\label{constr:fan2tv}
Let $\Sigma$ be a lattice fan in $\mathbb{Z}^{n}$.
For a cone $\sigma \subseteq \mathbb{Q}^{n}$ of
$\Sigma$ denote the dual cone by
$\sigma^\vee \subseteq \mathbb{Q}^{n}$ and consider
the spectrum $Z_\sigma$  of the monoid algebra
$\mathbb{K}[M_\sigma]$ of the additive monoid
$M_\sigma := \sigma^\vee \cap \mathbb{Z}^{n}$:
\[
Z_\sigma 
\ = \ 
\mathrm{Spec} \, \mathbb{K}[M_\sigma],
\qquad\qquad
\mathbb{K}[M_\sigma]
\ = \ 
\bigoplus_{u \in M_\sigma} \mathbb{K} \chi^u.
\]
The acting torus $\mathbb{T}^{n} = \mathrm{Spec} \, \mathbb{K}[\mathbb{Z}^{n}]$
embeds via $\mathbb{K}[M_\sigma] \subseteq \mathbb{K}[\mathbb{Z}^{n}]$ 
canonically into $Z_\sigma$, turning it into
an affine toric variety.
Similarly, we have open embeddings
$Z_\tau \subseteq Z_\sigma$
whenever $\tau \preceq \sigma$.
This allows to glue together all the $Z_\sigma$
to a variety $Z$:
\[
Z 
\ = \ 
\bigcup_{\sigma \in \Sigma} Z_\sigma,
\qquad
\text{where}
\
Z_{\sigma} \cap Z_{\sigma'}
\ = \
Z_{\sigma \cap \sigma'}
\ \subseteq \ 
Z.
\]
The gluing respects the toric structures
$\mathbb{T}^{n} \subseteq  Z_\sigma$ such that
we obtain a toric variety $\mathbb{T}^{n} \subseteq Z$,
\emph{the toric variety associated with the fan} $\Sigma$
in $\mathbb{Z}^{n}$. Any
$\chi^u \in \mathbb{K}[\mathbb{Z}^{n}]$ yields a
rational function on $Z$ satisfying
\[
\chi^u(1,\ldots,1)
\ = \
1,
\qquad\qquad
\chi^u(t \cdot z)
\ = \
t_1^{u_1} \cdots t_n^{u_n} \chi^u(z),
\]
whenever defined at $z \in Z$. In particular, the
restrictions $\chi^u \colon \mathbb{T}^{n} \to \mathbb{K}^*$
are precisely the characters of $\mathbb{T}^{n} \subseteq Z$.
Given a map $F$ of lattice fans  $\Sigma$ in
$\mathbb{Z}^{n}$ and $\Delta$ in $\mathbb{Z}^{m}$, we obtain
algebra homomorphisms
\[
\mathbb{K}[M_\tau] \ \to \ \mathbb{K}[M_\sigma],
\qquad
\chi^u \mapsto \chi^{u \circ F},
\]
whenever $\sigma \in \Sigma$ and $\tau \in \Delta$ satisfy
$F(\sigma) \subseteq \tau$. Passing to the spectra, this
defines morphisms $Z_\sigma \to Y_\tau$ which in turn glue
together to a morphism $Z \to Y$ of the toric varieties
associated with $ \Sigma$ and $\Delta$.
\end{construction}

The task of toric geometry is to link geometric
properties on the one side to combinatorial ones on the
other. We first indicate how to detect the orbit
decomposition and the invariant divisors of a toric
variety from its defining fan.

\begin{summary}
\label{sum:orbits-invarprimediv}
Let $\Sigma$ be a lattice fan in $\mathbb{Z}^{n}$ and $Z$
the associated toric variety. Every lattice vector
$v \in \mathbb{Z}^{n}$ defines a
\emph{one-parameter subgroup}
$$
\lambda_v \colon \mathbb{K}^* \ \to \ \mathbb{T}^{n} \ \subseteq Z,
\qquad \qquad 
t \ \mapsto \ (t^{v_1}, \ldots, t^{v_n}).
$$
If $v$ lies in the relative interior
$\sigma^\circ$ of a cone $\sigma \in \Sigma$,
then $\lambda_v$ extends to a morphism
$\bar \lambda_v  \colon \mathbb{K} \to Z$.
The associated \emph{limit point} is
$$
z_\sigma
\ := \
\lim_{t \to 0} \lambda_v(t)
\ := \
\bar \lambda_v (0)
\ \in \
Z_\sigma
\ \subseteq \
Z.
$$
Here, $z_\sigma$ does not depend on the
particular choice of $v \in \sigma^\circ$.
Note that the zero cone $\{0\} \in \Sigma$ yields
the unit element $z_0 \in \mathbb{T}^{n}$.
The limit points set up a bijection
$$
\Sigma \ \to \ \{\mathbb{T}^{n}\text{-orbits of } Z\},
\qquad
\sigma \ \mapsto \ \mathbb{T}^{n} \cdot z_\sigma.
$$
The dimension of the orbit $\mathbb{T}^{n} \cdot z_\sigma$
is $n- \dim(\sigma)$. In particular, the \emph{rays}, i.e.
the one-dimensional cones $\varrho_1,\ldots,\varrho_r$ of
$\Sigma$, define invariant prime divisors
\[
D_{i}
\ := \
\overline{\mathbb{T}^{n} \cdot z_{\varrho_i}}
\ \subseteq \
Z.
\]
We have $Z = \mathbb{T}^{n} \cup D_1 \cup \ldots \cup D_r$ and
$D_i \cap D_j \ne \emptyset$ if and only if
$\varrho_i, \varrho_j \subseteq \sigma$ for some
cone $\sigma \in \Sigma$. Moreover, there is an isomorphism
\[
\mathbb{Z}^{r} \ \to \ \mathrm{WDiv}(Z)^{\mathbb{T}^{n}},
\qquad
a \ \mapsto \ a_1D_1 + \ldots + a_rD_r.
\]
In other words, the prime divisors $D_{1},\dots,D_{r}$ freely
generate the group $\mathrm{WDiv}(Z)^{\mathbb{T}^{n}}$ of
\emph{invariant Weil divisors}. Finally, we obtain an
invariant anticanonical divisor
\[
-\mathcal{K}_{Z} \ = \ D_{1}+\dots+D_{r} \ \in \ \mathrm{WDiv}(Z)^{\mathbb{T}^{n}}.
\]
\end{summary}

Our next topic is the \emph{divisor class group}.
This is the group of Weil divisors modulo the subgroup
of principal divisors. We obtain an explicit description
in terms of the defining fan, which is also the key to
other invariants. We briefly recall the terminology. For
a point $z$ of a normal variety $Z$, the
\emph{local class group} $\mathrm{Cl}(Z,z)$ is the group
of Weil divisors modulo the subgroup of divisors that are
principal near~$z$. A \emph{Cartier divisor} is a Weil
divisor that is locally principal. The \emph{Picard group}
is the group of Cartier divisors modulo the subgroup of
principal divisors. Moreover, in the rational vector space
\[
\mathrm{Cl}_\QQ(Z)
\ = \
\QQ \otimes_\ZZ \mathrm{Cl}(Z)
\]
associated with the divisor class group of a variety $Z$,
there are the following important subsets. The
\emph{effective cone} $\Eff(Z)$ is generated by the
classes of effective divisors, i.e. contains all
non-negative $\QQ$-linear combinations of these classes.
The \emph{moving cone} $\Mov(Z)$ is generated by the
classes of movable divisors, i.e. those with base
locus of codimension one.
The \emph{semiample cone} $\SAmple(Z)$ is generated by
the classes of semiample divisors and the
\emph{ample cone} $\Ample(Z)$ consists of all strictly
positive $\QQ$-linear combinations of classes of ample divisors.

\begin{summary}
\label{sum:toriccldata}
Consider a lattice fan $\Sigma$ in $\mathbb{Z}^{n}$ and
its rays  $\varrho_1, \ldots, \varrho_r$ with the
corresponding primitive lattice vectors
$v_i \in \varrho_i$. The
\emph{generator matrix of $\Sigma$} is
\[
P \ := \ [v_1,\ldots,v_r].
\]
Suppose that $P$ is of rank $n$. We describe the
divisor class group of the toric variety~$Z$ associated
with $\Sigma$ in terms of $P$. First recall the
identification
\[
\mathbb{Z}^{r} \ \to \ \mathrm{WDiv}(Z)^{\mathbb{T}^{n}},
\qquad
a \ \mapsto \ D(a) := a_1D_1 + \ldots + a_rD_r.
\]
The divisor of a character function 
$\chi^u \in \Gamma(\mathbb{T}^{n},\mathcal{O}) \subseteq \mathbb{K}(Z)$,
where $u \in \mathbb{Z}^{n}$, is given via the transpose
$P^t$ of the generator matrix:
\[
D(P^t(u) )
\ = \
\bangle{u,v_1} D_1 + \ldots + \bangle{u,v_r} D_r
\ = \
\div(\chi^u).
\]
The divisor class group $\mathrm{Cl}(Z)$ of $Z$ turns
out to be the group of invariant Weil divisors modulo
the group of invariant principal divisors and thus we
obtain
\[
\mathrm{Cl}(Z)
\ = \
K
\ := \
\mathbb{Z}^{r} / \im(P^t).
\]
Using the projection $Q \colon \mathbb{Z}^{r} \to K$,
we have $[D_i] = Q(e_i)$ for the class
of the invariant prime divisor $D_i \subseteq Z$.
The local class group of $z_\sigma \in Z$ is
given by
\[
\mathrm{Cl}(Z,z_\sigma)
\ = \
K / K_\sigma,
\qquad\qquad
K_\sigma \ := \ \bangle{Q(e_i); \ P(e_i) \not\in \sigma}.
\]
Since $Z$ is the union of the $\mathbb{T}^{n}$-orbits
through the points $z_\sigma$, where
$\sigma \in \Sigma$, this determines all local class
groups. The Picard group of $Z$ is given by
\[
\Pic(Z)
\ = \
\bigcap_{\sigma \in \Sigma} K_\sigma
\ \subseteq \
\mathrm{Cl}(Z).
\]
In the rational divisor class group
$\mathrm{Cl}_\QQ(Z) = K_\QQ$, we identify the cones
of effective and movable divisor classes as
\[
\Eff(Z) \ = \ Q(\gamma),
\qquad\qquad
\Mov(Z)
\ = \
\bigcap_{{\gamma_0 \preccurlyeq \gamma}\atop{\text{facet}}} Q(\gamma_0),
\]
where $\gamma = \QQ^r_{\ge 0}$ denotes the positive
orthant in $\QQ^r$. The cones of semiample and
ample divisor classes in $\mathrm{Cl}_\QQ(Z)$ are
given by 
\[
\SAmple(Z)
\ = \
\bigcap_{\sigma \in \Sigma} \cone(Q(e_i); \ P(e_i) \not\in \sigma),
\qquad 
\Ample(Z) \ = \ \SAmple(Z)^\circ.
\]
\end{summary}

We take a look at singularities and smoothness.
Recall that a point $z \in Z$ of a
normal variety $Z$ is called \emph{$\QQ$-factorial}
if every Weil divisor admits a non-zero multiple
that is principal near $z$. A normal variety is
called \emph{$\QQ$-factorial} if each of its points
is $\QQ$-factorial.

\begin{summary}
\label{sum:tv-sing}
Let $Z$ be the toric variety arising from a lattice
fan $\Sigma$ in $\mathbb{Z}^{n}$. Consider the
limit points $z_\sigma \in Z$, where
$\sigma \in \Sigma$.
\begin{enumerate}
\item
The point $z_\sigma \in Z$ is $\QQ$-factorial if
and only if $\sigma$ is \emph{simplicial}, i.e.
generated by part of a vector space basis of
$\mathbb{Q}^{n}$.
\item
The point $z_\sigma \in Z$ is smooth if and only
if $\sigma$ is \emph{regular}, i.e. generated by
part of a lattice basis of $\mathbb{Z}^{n}$.
\item
The point $z_\sigma \in Z$ is $\QQ$-factorial
(smooth) if and only if all points of
$\mathbb{T}^{n} \cdot z_\sigma$ are
$\QQ$-factorial (smooth).  
\item
The variety $Z$ is $\QQ$-factorial (smooth)
if and only if all cones of the lattice fan
$\Sigma$ are simplicial (regular).
\end{enumerate}
\end{summary}

A lattice fan $\Sigma$ in $\mathbb{Z}^{n}$
is \emph{complete} if its \emph{support}, i.e.
the union over all its cones, equals $\mathbb{Q}^{n}$.
Moreover, $\Sigma$ is \emph{polytopal} if it
is spanned by an $n$-dimensional convex polytope
$\mathcal{A} \subseteq \mathbb{Q}^{n}$ containing the
origin in its interior; here \emph{spanned} means
that $\Sigma$ consists precisely of the cones generated
by the proper faces of $\mathcal{A}$.

\begin{summary}
Let the toric variety $Z$ arise from the lattice
fan $\Sigma$ in $\mathbb{Z}^{n}$. Then~$Z$ is
complete if and only if $\Sigma$ is complete and $Z$
is projective if and only if $\Sigma$ is polytopal.
If the latter holds, then, for any
$a \in \mathbb{Z}^{r}$ representing an ample divisor
$D(a) = a_{1}D_{1}+\dots+a_{r}D_{r}$, the lattice
fan $\Sigma$ is spanned by the dual of the polytope
\[
(P^t)^{-1}(B_a - a) \ \subseteq \ \mathbb{Q}^{n},
\qquad
B_a \ := \ Q^{-1}(Q(a)) \cap \gamma.
\]
Here, as before, $P$ is the generator matrix of
$\Sigma$ and $Q \colon \mathbb{Z}^{r} \to K =
\mathbb{Z}^{r}/\im(P^t)$ the projection and
$\gamma = \QQ^r_{\ge 0}$ the positive orthant. In
dimension one the only complete toric variety is
the projective line. Every complete toric surface
is projective. The first non-projective complete
toric varieties occur in dimension three.
\end{summary}

A \emph{Fano variety} is a normal
complete variety admitting an ample anticanonical
divisor. Observe that by this definition, Fano
varieties are projective and some positive
multiple of the anticanonical divisor of a Fano
variety is Cartier.

\begin{summary}
\label{sum:torvar-fano}
Let $\Sigma$ be a complete lattice fan in
$\mathbb{Z}^{n}$ and
$\mathcal{A} \subseteq \mathbb{Q}^{n}$ the convex
hull over the primitive generators $v_1,\ldots,v_r$ 
of $\Sigma$. Then the following statements are
equivalent.
\begin{enumerate}
\item
The toric variety $Z$ associated with $\Sigma$ is a
Fano variety.
\item
The vector $Q(e_{1})+\dots+Q(e_{r})$ lies in the
ample cone of $Z$.
\item
For every $\sigma \in \Sigma$ there is a
$u \in \QQ^r$ with $\bangle{u,v_i}=1$ whenever
$v_i \in \sigma$.
\item
The fan $\Sigma$ is spanned by the polytope
$\mathcal{A}$.
\end{enumerate}
\end{summary}

We will make use of
\emph{Cox's quotient presentation}, which generalizes
the construction of the projective space $\PP_n$ as
the quotient of  $\mathbb{K}^{n+1} \setminus \{0\}$
by $\mathbb{K}^*$ acting via scalar multiplication.
See \cite{Co}, also \cite[Sec. 5]{CoLiSc} and
\cite[Sec. 2.1.3]{ArDeHaLa} for more details.

\begin{construction}
\label{constr:torcox}
Consider a lattice fan $\Sigma$ in $\mathbb{Z}^{n}$
with generator matrix $P$ of rank~$n$ and let $Z$
be the associated toric variety. The \emph{Cox ring}
of $Z$ is given by
\[
\mathcal{R}(Z)
\ = \
\bigoplus_{D \in \mathrm{Cl}(Z)} \Gamma(Z,\mathcal{O}_D(Z))
\ \cong \
\bigoplus_{w \in K} \mathbb{K}[T_1, \ldots,T_r]_w
\ = \
\mathbb{K}[T_1, \ldots,T_r].
\]
The $K$-grading of the polynomial ring
$\mathbb{K}[T_1, \ldots,T_r]$ is given by
$\deg(T_i) := [D_i] = Q(e_i)$. Consider the orthant
$\gamma = \QQ_{\ge 0}^r$, its fan of faces
$\bar \Sigma$ and 
\[
\hat \Sigma
\ := \ 
\{\tau \preccurlyeq \gamma; \ 
P(\tau) \subseteq \sigma 
\text{ for some } 
\sigma \in \Sigma\}.
\]
Then $\hat \Sigma$ is a subfan of the fan $\bar \Sigma$
and $P$ sends cones from $\hat \Sigma$ into
cones of $\Sigma$.
For the associated toric varieties this leads to
the picture
\[
\xymatrix{
{\TT^r}
\ar@{}[r]|\subseteq
\ar[d]_{/H}^p  
&
{\hat Z}
\ar@{}[r]|\subseteq
\ar[d]^{\quot H}_p
&
{\bar Z}
\ar@{}[r]|{:=}
&
{\mathbb{K}^r}.
\\
{\mathbb{T}^{n}}
\ar@{}[r]|\subseteq
&
Z
&
&
}
\]
Here, $\hat Z \subseteq \bar Z$ is an open 
$\TT^r$-invariant subvariety and $p$ extends the
homomorphism of tori having the rows of
$P = (p_{ij})$ as its exponent vectors:
\[
\TT^r \ \to \ \mathbb{T}^{n}, 
\qquad
t \ \mapsto \ (t^{P_{1*}}, \ldots, t^{P_{n*}}),
\qquad
t^{P_{i*}} \ := \ t_1^{p_{i1}} \cdots t_r^{p_{ir}}.
\]
The morphism $p \colon \hat Z \to Z$ is a \emph{good quotient}
for the action of the quasitorus
$H = \ker(p) \subseteq \TT^r$ on $\hat Z$. If
$\Sigma$ is simplicial, then each fiber of $p$ is
an $H$-orbit.
\end{construction}

\begin{remark}
\label{rem:coxcoord}
Let $Z$ be a toric variety with quotient presentation
$p \colon \hat Z \to Z$ as in Construction
\ref{constr:torcox}. Then every $p$-fiber contains a
unique closed $H$-orbit. The presentation in
\emph{Cox coordinates} of a point $x \in Z$ is
\[
x \ = \ [z_1,\ldots, z_r],
\qquad
\text{where}
\
z = (z_1,\ldots, z_r) \in p^{-1}(x)
\text{ with } H \cdot z \subseteq \hat Z
\text{ closed}.
\]
Thus, $[z]$ and $[z']$ represent 
the same point $x \in Z$ if and only if 
$z$ and $z'$ lie in the same closed $H$-orbit 
of $\hat Z$.
For instance, the points $z_\sigma \in Z$, 
where $\sigma \in \Sigma$, 
are given in Cox coordinates as
\[
z_\sigma 
\ = \ 
[\varepsilon_1,\ldots, \varepsilon_r],
\qquad
\varepsilon_i 
\ = \ 
\begin{cases}
0, & P(e_i) \in \sigma,
\\
1, & P(e_i) \not\in \sigma.
\\
\end{cases}
\]

\end{remark}

We conclude the section with an example
from the surface case in order to see how its
(well known) geometric features are obtained
from its defining combinatorial data.

\begin{example}

Consider the weighted projective plane $\PP_{1,2,3}$.
As a toric variety it arises from the complete
lattice fan $\Sigma$ in $\ZZ^2$ with generator
matrix
\[
P
\ = \
[v_1,v_2,v_3]
\ = \
\left[
\begin{array}{rrr}
1 & 1 & -1
\\
2 & -1 & 0
\end{array}  
\right].
\]
In order to see that the toric variety $Z$
associated with $\Sigma$ indeed equals
$\PP_{1,2,3}$ we look at Cox's quotient
construction.
\begin{center}
\begin{tikzpicture}[scale=0.6]
\sffamily
\coordinate(oo) at (0,0);
\coordinate(e0) at (-1,0);
\coordinate(e1) at (1,-.5);
\coordinate(e2) at (0.7,0.3);
\coordinate(c0) at (-2,0);
\coordinate(c1) at (2,-1);
\coordinate(c2) at (1.4,0.6);
\coordinate(ooo) at (0,0+2.5);
\coordinate(e0o) at (-1,0+3.5);
\coordinate(e1o) at (1,-.5+3.5);
\coordinate(e2o) at (0.7,0.3+3.5);
\coordinate(c0o) at (-2,0+4.5);
\coordinate(c1o) at (2,-1+4.5);
\coordinate(c2o) at (1.4,0.6+4.5);
\coordinate(av1) at (0,2.1);
\coordinate(av2) at (0,0.8);
\coordinate(bv1) at (-8.5,2.5);
\coordinate(bv2) at (-8.5,0.8);

\node[left] at (-8,3.4) {$\mathbb{K}^3 \setminus \{0\} \ = \ \hat Z$};
\node[left] at (-8.1,0) {$\PP_{1,2,3} \ = \ Z$};
\node[right] at (-8.4,1.7) {$p$};
\node[left] at (-8.5,1.7) {$/H$};
\node[] at (2.8,4) {$\hat \Sigma$};
\node[] at (2.8,0) {$\Sigma$};
\node[right] at (0.2,1.5) {$P$};
\path[fill, color=gray!10] (oo) -- (c2) -- (c0) -- (oo);
\path[fill, color=gray!30] (oo) -- (c1) -- (c2) -- (oo);
\path[fill, color=gray!50] (oo) -- (c0) -- (c1) -- (oo);
\path[fill, color=gray!10] (ooo) -- (c2o) -- (c0o) -- (ooo);
\path[fill, color=gray!30] (ooo) -- (c1o) -- (c2o) -- (ooo);
%
\path[fill, color=black] (e0) circle (.5ex);
\path[fill, color=black] (e1) circle (.5ex);
\path[fill, color=black] (e2) circle (.5ex);
\draw[thick, color=black] (oo) to (c0);
\draw[thick, color=black] (oo) to (c1);
\draw[thick, color=black] (oo) to (c2);
\path[fill, color=black] (e2o) circle (.5ex);
\draw[thick, color=black] (ooo) to (c2o);
\path[fill, opacity=.9, color=gray!50] (ooo) -- (c0o) -- (c1o) -- (ooo);
\path[fill, color=black] (e0o) circle (.5ex);
\path[fill, color=black] (e1o) circle (.5ex);
\draw[thick, color=black] (ooo) to (c0o);
\draw[thick, color=black] (ooo) to (c1o);
%
\draw[->, thick, color=black] (av1) to (av2);
\draw[->, thick, color=black] (bv1) to (bv2);
\end{tikzpicture}
\end{center}
Note that the homomorphism of tori defined by the
generator matrix $P$ and its kernel $H$ are
explicitly given by 
\[
p \colon \TT^3 \ \to \ \TT^2,
\quad
(t_1,t_2,t_3) \ \mapsto \ \left(\frac{t_1t_2}{t_3},\frac{t_1^2}{t_2}\right),
\qquad
H \ = \ \{(t,t^2,t^3); \ t \in \mathbb{K}^*\}.
\]
Thus, $H = \mathbb{K}^*$ acts with the weights
$1$, $2$ and $3$ in $\mathbb{K}^3$ and
$Z = \PP_{1,2,3}$ is the quotient of
$\hat Z = \mathbb{K}^3 \setminus \{0\}$ by
this action. Let us look at the geometry of
$\PP_{1,2,3}$. We have
\[
\mathrm{Cl}(Z) \ = \ \ZZ^3/\im(P^t) \ = \ \ZZ
\]
with the projection $Q \colon \ZZ^3\to \ZZ$
sending $(a_1,a_2,a_3)$ to $(a_1,2a_2,3a_3)$.
This allows to recover the Picard group as 
\[
\Pic(Z)
\ = \ 
\ZZ \cap 2\ZZ \cap 3 \ZZ
\ = \
6\ZZ
\ \subseteq \
\ZZ
\ = \
\mathrm{Cl}(Z).
\]
As $\Sigma$ has three rays, there are three
invariant prime divisors $D_1$, $D_2$ and $D_3$
on $Z$. In particular, the anticanonical
class is given by
\[
[-\mathcal{K}_Z]
\ = \
[D_1]+[D_2]+[D_3]
\ = \
1 + 2 + 3
\ = \
6
\ \in \
\ZZ
\ = \
\mathrm{Cl}(Z).
\]
Thus, $-\mathcal{K}_Z$ is an ample Cartier divisor
generating $\Pic(Z)$. Note that $\Sigma$ is spanned
by $\mathcal{A} = \conv(v_1,v_2,v_3)$. There are
two singularities, given in Cox-coordinates by
\[
[0,1,0], \qquad\qquad [0,0,1].
\]

\end{example}


\section{Geometry of toric surfaces}
\label{section:toric-surfaces}

In this section, we take a closer look at two-dimensional
toric geometry, i.e. the case of toric surfaces.
In dimension two, the combinatorial framework of lattice
fans is elementary and thus the toric surfaces form a
particularly grateful example class;
see~\cite[Chap.~10]{CoLiSc} for a detailed treatment.
Let us see how the general picture from the preceding
section looks in the surface case.

\begin{summary}
Consider the projective toric surface $Z$ arising from
a complete lattice fan $\Sigma$ in $\ZZ^2$, the generator
matrix $P = [v_1,\ldots, v_r]$ of $\Sigma$, the rays
$\varrho_i = \cone(v_i) \in \Sigma$ and their limit
points $z_i \in Z$. Then we have
\[
Z \ = \  \TT^2  \cup D_1 \cup \ldots \cup D_r,
\qquad
D_i \ := \ \overline{\TT^2 \cdot z_i}.
\]
Each of the orbit closures $D_i$ is a smooth rational
curve and $D_1 \cup \ldots \cup D_r$ is a cycle in the
sense that $D_i \cap D_j$ is non-empty if and only if the
corresponding rays~$\varrho_i$ and~$\varrho_j$ are
adjacent. If two rays $\varrho_i$ and $\varrho_j$ are
adjacent, we have
\[
D_i \cap D_j \ = \ \{z_{ij}\},
\qquad
z_{ij}  \ := \ z_{\sigma_{ij}},
\qquad 
\sigma_{ij} \ := \ \cone(v_i,v_j) \ \in \ \Sigma.
\]
A one-parameter subgroup
$\lambda_v \colon \mathbb{K}^* \to \TT^2$,
$t \mapsto (t^{v_1},t^{v_2})$ approaches $z_i$ or
$z_{ij}$ for $t \to 0$ if and only if $v = (v_1,v_2)$
lies in the relative interior of $\varrho_i$ or
$\sigma_{ij}$, respectively. Gathering derivatives of 
$\lambda_v$ with common limit into cones we recover
$\Sigma$ from $Z$:

\begin{center}
\begin{tikzpicture}[scale=0.6]
\sffamily
\coordinate(x0) at (0,0);
\coordinate(x1) at (0,2.1);
\coordinate(y1) at (2,0.7);
\coordinate(y2) at (-2,0.7);
\coordinate(z1) at (1.5,-1.7);
\coordinate(z2) at (-1.5,-1.7);
\coordinate(x1y2) at ($(x1)!0.5!(y2)$);
\coordinate(r0) at ([xshift=9cm]$(x0)$);
\coordinate(r1) at ([xshift=9cm]$(x1)!0.5!(y1)$);
\coordinate(r2) at ([xshift=9cm]$(x1)!0.5!(y2)$);
\coordinate(r3) at ([xshift=9cm]$(y2)!0.5!(z2)$);
\coordinate(r4) at ([xshift=9cm]$(z2)!0.5!(z1)$);
\coordinate(r5) at ([xshift=9cm]$(z1)!0.5!(y1)$);
\coordinate(rr1) at ($(r0)!1.5!(r1)$);
\coordinate(rr2) at ($(r0)!1.5!(r2)$);
\coordinate(rr3) at ($(r0)!1.3!(r3)$);
\coordinate(rr4) at ($(r0)!1.3!(r4)$);
\coordinate(rr5) at ($(r0)!1.3!(r5)$);
\path[fill, color=gray!30] (rr1) -- (rr2) -- (rr3) -- (rr4) -- (rr5) -- (rr1);
\begin{scope}[line width=2pt]
\draw[bend angle=20, bend right, color=black] (x1) to (y1);
\draw[bend angle=20, bend right, color=black] (y2) to (x1);
\draw[bend angle=20, bend right, color=black] (y1) to (z1);
\draw[bend angle=20, bend right, color=black] (z2) to (y2);
\draw[bend angle=20, bend right,  color=black] (z1) to (z2);
\end{scope}
\path[fill, color=black] (x0) circle (1mm)
node[below, font=\scriptsize]{$z_0$};
\path[fill, color=black] (x1) circle (1mm)
node[above, font=\scriptsize]{$z_{23}$};
\path[fill, color=black] (y1) circle (1mm)
node[right, font=\scriptsize]{$z_{12}$};
\path[fill, color=black] (y2) circle (1mm)
node[left, font=\scriptsize]{$z_{34}$};
\path[fill, color=black] (z1) circle (1mm)
node[below, font=\scriptsize]{$z_{51}$};
\path[fill, color=black] (z2) circle (1mm)
node[below, font=\scriptsize]{$z_{45}$};
\begin{scope}[very thick, decoration={markings,
    mark=at position 0.5 with {\arrow{>}}}
    ] 
\draw[postaction={decorate}, color=black] (x0) to (y1);
\draw[postaction={decorate}, bend angle=40, bend right, color=black] (x0) to (y1);
\draw[postaction={decorate}, bend angle=40, bend left, color=black] (x0) to (y1);
\draw[postaction={decorate}, shorten >=2mm, color=black] (x0) to ($(x0)!1.05!(x1y2)$);
\draw[postaction={decorate}, bend angle=10, bend left, color=black] (3,0) to (6,0);
\path[] (3,.8) node[right, font=\scriptsize] {$\lambda_v \mapsto \frac{\partial \lambda_v}{\partial t}(1)$};
\end{scope}
\draw[thick, color=black] (r0) to (rr1);
\draw[thick, color=black] (r0) to (rr2);
\draw[thick, color=black] (r0) to (rr3);
\draw[thick, color=black] (r0) to (rr4);
\draw[thick, color=black] (r0) to (rr5);
\node[font=\scriptsize] (sig12) at (10.4,.4) {$\sigma_{12}$};
\node[font=\scriptsize] (sig12) at (9,1.6) {$\sigma_{23}$};
\node[font=\scriptsize] (sig34) at (7.6,.4) {$\sigma_{34}$};
\node[font=\scriptsize] (sig12) at (8.3,-1) {$\sigma_{45}$};
\node[font=\scriptsize] (sig12) at (9.9,-1) {$\sigma_{51}$};
\end{tikzpicture}
\end{center}

\noindent
The orbits $\TT^2 \cdot z_0$ and $\TT^2 \cdot z_i$
only contain smooth points of $Z$. The $z_{ij}$ are
precisely the toric fixed points. Every toric fixed point
$z_{ij}$ is $\QQ$-factorial. Moreover, $z_{ij}$ is smooth
if and only if $\det(v_i,v_j) = \pm 1$.
\end{summary}

Being $\QQ$-factorial, any complete toric
surface comes with a well-defined intersection product. 
Here is how to compute the intersection numbers explicitly
in terms of the defining fan.

\begin{summary}
\label{sum:torsurfintnos}
Consider the toric surface $Z$ arising from a complete
lattice fan $\Sigma$ in $\ZZ^2$ with generator matrix
\[
P
\ = \
[v_1,\ldots,v_r].
\]
For any two distinct generators $v_i, v_j$ in positive
orientation, the intersection number of the associated
divisors $D_i, D_j$ is given as
\[
D_i \cdot D_j
\ = \
\begin{cases}
\det(v_i,v_j)^{-1}, & \text{if } \cone(v_i,v_j) \in \Sigma,
\\
0, & \text{else}.
\end{cases}
\]
Moreover, we can compute the self intersection number
of a divisor $D_j$. Taking adjacent generators
$v_i$, $v_j$, $v_k$ in positive orientation, we have
\[
D_j^2
\ = \
-\frac{\det(v_i,v_k)}{\det(v_i,v_j)\det(v_j,v_k)}.
\]
\end{summary}

\begin{remark}
\label{rem:toricaccselfie}
Let $\Sigma$ be a complete lattice fan in $\ZZ^2$ with
generator matrix $P$ with columns $v_1,\ldots,v_r$ listed
according to positive orientation.
Assume $v_i = (l_i,d_i)$ with $l_i \ne 0$ for $i=1,\dots,r$
and set $v_{r+1} := v_1$.
Then the self intersection number of any canonical
divisor $\mathcal{K}_{Z}$ on the toric surface $Z$
associated with $\Sigma$ is given as
\[
\mathcal{K}_Z^2
\ = \
\sum_{i = 1}^r
\left(
2 - \frac{l_i}{l_{i+1}} - \frac{l_{i+1}}{l_i}
\right)
\frac{1}{\det(v_i,v_{i+1})}.
\]
\end{remark}

We take a closer look at the singularities
of toric surfaces. As observed, these are necessarily
toric fixed points. Thus, for the local study, we have to
consider affine toric surfaces defined by
two-dimensional lattice cones.

\begin{summary}
Consider the affine toric surface $Z_\sigma$
associated with a two-dimensional lattice cone $\sigma$
in $\ZZ^2$. After applying a suitable unimodular
transformation, the generator matrix is of the shape
\[
P
\ = \
\left[
\begin{array}{cc}
1 & a 
\\
0 & b
\end{array}    
\right],
\qquad
0 \le a < b.
\]
The cone $\sigma$ is therefore generated by the
columns of $P$. The point $z_\sigma \in Z_\sigma$ is
singular if and only if $b>1$ holds. Cox's quotient
presentation from \ref{constr:torcox} yields a
morphism $p \colon \mathbb{K}^2 \to \mathbb{K}^2/H =
Z_\sigma$, where
\[
H
\ = \
\ker(p)
\ = \ 
\{(t^{-a},t); \ t \in \Gamma_b\},
\qquad
\Gamma_b
\ := \
\{t \in \mathbb{K}^*; \ t^b = 1\}.
\]
Thus, $Z_\sigma$ is the quotient of $\mathbb{K}^2$ by 
a diagonal action of a cyclic group of order $b$. In
particular, for $b \ge 2$, we see that the point
$z_\sigma = p(0,0)$ is a two-dimensional cyclic
quotient singularity.
\end{summary}

We discuss the resolution of toric surface
singularities. Recall that the \emph{Hilbert basis}
$\mathcal{H}_\sigma$ of a pointed convex polyhedral
cone $\sigma \subseteq \mathbb{Q}^{n}$ is the (finite)
set of indecomposable elements of the additive monoid
$\sigma \cap \mathbb{Z}^{n}$. Here, a non-zero element
$v \in \sigma \cap \mathbb{Z}^{n}$ is
\emph{indecomposable} if $v = v'+ v''$ with
$v',v'' \in \sigma \cap \mathbb{Z}^{n}$ is only
possible for $v = v'$ or $v = v''$.

\begin{summary}
\label{sum:toricsurfinters}
Let the affine toric surface $Z_\sigma$ arise
from a two-dimensional lattice cone $\sigma$ in
$\ZZ^2$. Subdividing $\sigma$ along the members
$v_1,\dots,v_r$ of the Hilbert basis~$\mathcal{H}_\sigma$
gives a lattice fan $\Sigma$
in $\ZZ^2$ with support $\sigma$ and generator
matrix $P = [v_1,\dots,v_r]$.
\begin{center}
\begin{tikzpicture}[scale=0.6]
\sffamily
\coordinate(ool) at (0,0);
\coordinate(v1l) at (1,0);
\coordinate(sv1l) at (3,0);
\coordinate(v2l) at (1,1);
\coordinate(sv2l) at (3,3);
\coordinate(v3l) at (2,3);
\coordinate(sv3l) at (3,4.5);
\coordinate(oor) at (0+8,0);
\coordinate(v1r) at (1+8,0);
\coordinate(sv1r) at (3+8,0);
\coordinate(v3r) at (2+8,3);
\coordinate(sv3r) at (3+8,4.5);
\path[fill, color=gray!30] (ool) -- (sv1l) -- (sv3l) -- (ool);
\path[fill, color=gray!30] (oor) -- (sv1r) -- (sv3r) -- (oor);
\path[fill, color=black] (v1l) circle (1mm);
\path[fill, color=black] (v2l) circle (1mm);
\path[fill, color=black] (v3l) circle (1mm);
\path[fill, color=black] (v1r) circle (1mm);
\path[fill, color=black] (v3r) circle (1mm);
\draw[thick, color=black] (ool) to (sv1l);
\draw[thick, color=black] (ool) to (sv2l);
\draw[thick, color=black] (ool) to (sv3l);
\draw[thick, color=black] (oor) to (sv1r);
\draw[thick, color=black] (oor) to (sv3r);
\draw[very thick, color=black, -stealth] (4.5,1.5) to (6.5,1.5);
\end{tikzpicture}
\end{center}
The toric surface $Z$ associated with $\Sigma$ is
smooth and the canonical toric morphism
$\pi \colon Z \to Z_\sigma$ is the minimal
resolution of singularities. The exceptional
curves of $\pi$ are precisely the $D_i$ given
by the Hilbert basis members
$v_i \in \sigma^\circ$.

\end{summary}

We say that a normal variety $X$ with
canonical divisor $\mathcal{K}_X$ is
\emph{$\QQ$-Gorenstein} if some non-zero integral
multiple of $\mathcal{K}_X$ is Cartier. If this is
fulfilled, then the \emph{Gorenstein index} of $X$ is
the smallest non-zero integer $\iota_X$ such that
$\iota_X \mathcal{K}_X$ is Cartier. A variety is
said to be \emph{Gorenstein} if it is of
Gorenstein index $1$.

\begin{summary}
\label{sum:toric-surf-gorind}
Consider a two-dimensional lattice cone $\sigma$
in $\ZZ^2$ with primitive generators $v_1$ and $v_2$.
Then there is a primitive
$u_\sigma \in \ZZ^2$ and an
$\iota_\sigma \in \ZZ_{> 0}$ such that 
\[
\bangle{u_\sigma,v_1} \ = \ \iota_\sigma,
\qquad\qquad
\bangle{u_\sigma,v_2} \ = \ \iota_\sigma.
\]
Moreover $u_\sigma$ and $\iota_\sigma$ are
uniquely determined by this property. For the
associated affine toric surface $Z_\sigma$ and
its invariant canonical divisor, we have 
\[
- \iota_\sigma  \mathcal{K}_{Z_\sigma}
\ = \
\iota_\sigma D_1 + \iota_\sigma D_2
\ = \
\bangle{u_\sigma,v_1}D_1 + \bangle{u_\sigma,v_2}D_2
\ = \
\div(\chi^{u_\sigma}).
\]
Thus, $\iota_\sigma$ is minimal with
$\iota_\sigma \mathcal{K}_Z$ being Cartier and
hence equals the Gorenstein index of $Z_\sigma$.
For a toric surface $Z$ arising from a complete
lattice fan $\Sigma$ in $\ZZ^2$, we obtain
\[
\iota_Z
\ = \
\lcm(\iota_\sigma; \ \sigma \in \Sigma, \ \dim(\sigma) = 2). 
\]
\end{summary}

\begin{remark}
\label{rem:toric-surf-gorind}
Consider two  primitive vectors $v_1 = (a,c)$
and $v_2 = (b,d)$ in $\ZZ^2$ generating a
two-dimensional cone $\sigma \subseteq \QQ^2$.
The linear form $u_\sigma$ and the number
$\iota_\sigma$ from Summary
\ref{sum:toric-surf-gorind} are given as 
\[
u_\sigma
\ = \
\iota_\sigma
\left(
\frac{c-d}{ad-bc},  \, \frac{b-a}{ad-bc}
\right),
\qquad
\qquad
\iota_\sigma
\ = \
\frac{\vert ad-bc \vert}{\gcd(c-d, \, b-a)}.
\]
In particular, this allows to compute the
Gorenstein index $\iota_\sigma$ of the affine
toric surface $Z_\sigma$ in terms of the
generator matrix $P = [v_1,v_2]$. Note that
$\iota_\sigma$ equals the order of the class
of $\mathcal{K}_{Z_\sigma}$ in the local
class group $\mathrm{Cl}(Z_\sigma,z_\sigma)$.
\end{remark}

\begin{proposition}
\label{prop:affineiota}
In $\ZZ^2$, consider the vector $e := (1,0)$,
for $a \in \ZZ_{\ge 1}$ the vectors
$v_a := (1,a)$ and for
$\iota, b \in \ZZ_{\ge 2}$ the vectors  
\[
v_{\iota,b,\kappa}
\ := \
\left(b, \, \iota \frac{b-1}{\kappa} \right),
\,
\kappa = 1, \ldots, \iota -1,
\
\kappa \mid b-1,
\
\gcd(b,\iota {\scriptstyle \frac{b-1}{\kappa}}) = \gcd(\kappa,\iota) = 1.
\]
Set $\sigma_a := \cone(e,v_a) \subseteq \QQ^2$
and $\sigma_{\iota,b,\kappa} :=
\cone(e,v_{\iota,b,\kappa}) \subseteq \QQ^2$.
Then the following statements hold.
\begin{enumerate}
\item
Up to isomorphy the Gorenstein affine toric surfaces
with fixed point are precisely the $Z_{\sigma_a}$.
\item
Fix $\iota \in \ZZ_{\ge 2}$. Up to isomorphy the
affine toric surfaces with fixed point being of
Gorenstein index $\iota$ are precisely the
$Z_{\sigma_{\iota,b,\kappa}}$.
\end{enumerate} 
\end{proposition}

\begin{proof}
By Summary \ref{sum:toric-surf-gorind}, a toric
surface $Z_\sigma$ is of Gorenstein index $\iota$
if there is a primitive $u \in \ZZ^2$ evaluating
to $\iota$ on the primitive generators of
$\sigma$.
The linear forms
\[
u_a
\ := \
(1,0),
\qquad\qquad
u_{\iota,b,\kappa}
\ := \
(\iota,-\kappa)
\]
do so for $Z_{\sigma_a}$ and
$Z_{\sigma_{\iota,b,\kappa}}$. Conversely, given
an affine toric surface $Z_\sigma$ with fixed
point of Gorenstein index $\iota$, we may assume
that
\[
\sigma \ = \ \cone(v_1,v_2),
\qquad
v_1 = (1,0),
\qquad
v_2 = (a,b),
\qquad
0 \le a < b.
\]
Then we directly see that the existence of a
primitive $u \in \ZZ^2$ evaluating to $\iota$
on $v_1$ and $v_2$ forces $v_2 = v_a$ or
$v_2 = v_{\iota,b,\kappa}$.
\end{proof}

\begin{example}
We examine the minimal resolution of
affine toric surfaces of small Gorenstein
index. For $\iota=1$, we have to look at
\[
\sigma_a \ = \ \cone((1,0), \, (1,a)),
\quad a = 1,2,3, \ldots 
\qquad
\mathcal{H}_{\sigma_a}
\ = \
\{(1,j); \ j = 1, \ldots, a\}.
\]
Moreover for $\iota = 2$, by Proposition
\ref{prop:affineiota} we have to consider the
cones $\sigma_{2,b,\kappa}$, where
$b \ge 2$ must be odd and $\kappa = 1$.
Thus, we end up with
\[
\sigma_{2,b,1} \ = \ \cone((1,0), \, (b,2b-2)),
\quad b = 3, 5, 7, 9, \ldots 
\]
Besides the primitive generators, the Hilbert
basis of $\sigma_{2,b,1}$ contains all
interior lattice points of the line with
slope $2$ through $(1,1)$, i.e. 
\[
\mathcal{H}_{\sigma_{2,b,1}}
\ = \
\{(1,0), \, (b,2b-2)\}
\cup 
\left\{(1+j,1+2j); \ j = 0, \ldots, \frac{b-3}{2}\right\}.
\]
Using Summary \ref{sum:toricsurfinters} we can
compute \emph{resolution graphs}, the vertices
of which represent the exceptional curves,
labelled by their self intersection number.
Two vertices are connected by an edge if and
only if the corresponding curves intersect.
\begin{center}
\begin{tikzpicture}[scale=0.6]
\sffamily
\path[color=black] (-1,0) circle (1mm);
\node at (-7,0) {$Z_{\sigma_a}$:};
\draw (-3.5,0) circle (10pt);
\draw (-2,0) circle (10pt);
\node[] (a) at (-3.5,0) {$\scriptscriptstyle -2$};
\node[] (a1) at (-2,0) {$\scriptscriptstyle -2$};
\draw (3.5,0) circle (10pt);
\draw (2,0) circle (10pt);
\node[] (b) at (3.5,0) {$\scriptscriptstyle -2$};
\node[] (b1) at (2,0) {$\scriptscriptstyle -2$};
\draw[] (a) edge (a1);
\draw[] (b1) edge (b);
\draw[dotted] (-1.5,0) to (1.5,0);
\node at (7,0) {$a = 2,3,4, \ldots ,$};

\node at (-7,0-1.5) {$Z_{\sigma_{2,b,1}}$:};
\draw (0,0-1.5) circle (10pt);
\node[] (a) at (0,0-1.5) {$\scriptscriptstyle -4$};
\node at (7,0-1.5) {$b = 3, $};

\node at (-7,0-3) {$Z_{\sigma_{2,b,1}}$:};
\draw (-0.75,0-3) circle (10pt);
\draw (0.75,0-3) circle (10pt);
\node[] (a) at (-0.75,0-3) {$\scriptscriptstyle -3$};
\node[] (b) at (0.75,0-3) {$\scriptscriptstyle -3$};
\draw[] (a) edge (b);
\node at (7,-3) {$b = 5,$};

\node at (-7,0-4.5) {$Z_{\sigma_{2,b,1}}$:};
\draw (-3.5,0-4.5) circle (10pt);
\draw (-2,0-4.5) circle (10pt);
\node[] (a) at (-3.5,0-4.5) {$\scriptscriptstyle -3$};
\node[] (a1) at (-2,0-4.5) {$\scriptscriptstyle -2$};
\draw (3.5,0-4.5) circle (10pt);
\draw (2,0-4.5) circle (10pt);
\node[] (b) at (3.5,0-4.5) {$\scriptscriptstyle -3$};
\node[] (b1) at (2,0-4.5) {$\scriptscriptstyle -2$};
\draw[] (a) edge (a1);
\draw[] (b1) edge (b);
\draw[dotted] (-1.5,0-4.5) to (1.5,0-4.5);
\node at (7,0-4.5) {$b = 7,9,11, \ldots .$};
\end{tikzpicture}
\end{center}

One calls the singularities
of $Z_{\sigma_a}$ of \emph{type $A_n$}
with $n=a-1$ and those of
$Z_{\sigma_{2,b,1}}$ of
\emph{type $K_n$} with $n =(b-1)/2$.
Note that $n$ is the number of
exceptional curves.
\end{example}

We dicuss toric del Pezzo surfaces. Recall that a
\emph{del Pezzo surface} is a two-dimensional Fano
variety, i.e. a normal projective surface admitting
an ample anticanonical divisor.
By a \emph{(lattice) polygon} we mean a convex
polytope in $\QQ^2$ (having only integral vertices).
An \emph{LDP-polygon} is a polygon in $\QQ^2$
containing the origin as an interior
point and having only primitive vectors from
$\ZZ^2$ as vertices. We call two polygons 
\emph{unimodularly equivalent} if they can be
transformed into each other by a
\emph{unimodular matrix}, i.e. an integral
$2 \times 2$ matrix with determinant $\pm 1$.

\begin{summary}
Any LDP-polygon $\mathcal{A} \subseteq \QQ^2$
spans a complete lattice fan $\Sigma_{\mathcal{A}}$
in $\ZZ^2$. The toric surface $Z_\mathcal{A}$
associated with $\Sigma_\mathcal{A}$ is a del Pezzo
surface. The assignment
$\mathcal{A} \mapsto Z_{\mathcal{A}}$ yields a
bijection between the unimodular equivalence
classes of LDP-polygons and the isomorphy classes
of toric del Pezzo surfaces.
\end{summary}

\begin{example}
Up to isomorphy, there are only five smooth
toric del Pezzo surfaces. They are given by the
following LDP-polygons.

\def\plinepline{
\begin{tikzpicture}[scale=0.4]
\sffamily
\coordinate(oo) at (0,0);
\coordinate(v1) at (1,0);
\coordinate(v2) at (0,1);
\coordinate(v3) at (-1,0);
\coordinate(v4) at (0,-1);
\coordinate(r1) at (2,0);
\coordinate(r2) at (0,2);
\coordinate(r3) at (-2,0);
\coordinate(r4) at (0,-2);
\path[fill, color=gray!30] (oo) -- (r1) -- (r2) -- (r3) -- (r4) -- (r1) -- (oo);  
\path[fill, color=gray!50] (oo) -- (v1) -- (v2) -- (v3) -- (v4) -- (v1) -- (oo);  
\draw[color=black] (oo) to (r1);
\draw[color=black] (oo) to (r2);
\draw[color=black] (oo) to (r3);
\draw[color=black] (oo) to (r4);
\draw[thick, color=black] (v1) to (v2);
\draw[thick, color=black] (v2) to (v3);
\draw[thick, color=black] (v3) to (v4);
\draw[thick, color=black] (v4) to (v1);
\path[fill, color=black] (v1) circle (1mm);
\path[fill, color=black] (v2) circle (1mm);
\path[fill, color=black] (v3) circle (1mm);
\path[fill, color=black] (v4) circle (1mm);
\end{tikzpicture}
}

\def\projplane{
\begin{tikzpicture}[scale=0.4]
\sffamily
\coordinate(oo) at (0,0);
\coordinate(v1) at (1,0);
\coordinate(v2) at (0,1);
\coordinate(v3) at (-1,-1);
\coordinate(r1) at (2,0);
\coordinate(r2) at (0,2);
\coordinate(r3) at (-2,-2);
\path[fill, color=gray!30] (oo) -- (r1) -- (r2) -- (r3) -- (r1) -- (oo);  
\path[fill, color=gray!50] (oo) -- (v1) -- (v2) -- (v3) -- (v1) -- (oo);  
\draw[color=black] (oo) to (r1);
\draw[color=black] (oo) to (r2);
\draw[color=black] (oo) to (r3);
\draw[thick, color=black] (v1) to (v2);
\draw[thick, color=black] (v2) to (v3);
\draw[thick, color=black] (v3) to (v1);
\path[fill, color=black] (v1) circle (1mm);
\path[fill, color=black] (v2) circle (1mm);
\path[fill, color=black] (v3) circle (1mm);
\end{tikzpicture}
}

\def\bloneprojplane{
\begin{tikzpicture}[scale=0.4]
\sffamily
\coordinate(oo) at (0,0);
\coordinate(v1) at (1,0);
\coordinate(v2) at (0,1);
\coordinate(v3) at (-1,-1);
\coordinate(v4) at (0,-1);
\coordinate(r1) at (2,0);
\coordinate(r2) at (0,2);
\coordinate(r3) at (-2,-2);
\coordinate(r4) at (0,-2);
\path[fill, color=gray!30] (oo) -- (r1) -- (r2) -- (r3) -- (r4) -- (r1) -- (oo);  
\path[fill, color=gray!50] (oo) -- (v1) -- (v2) -- (v3) -- (v4) -- (v1) -- (oo);  
\draw[color=black] (oo) to (r1);
\draw[color=black] (oo) to (r2);
\draw[color=black] (oo) to (r3);
\draw[color=black] (oo) to (r4);
\draw[thick, color=black] (v1) to (v2);
\draw[thick, color=black] (v2) to (v3);
\draw[thick, color=black] (v3) to (v4);
\draw[thick, color=black] (v4) to (v1);
\path[fill, color=black] (v1) circle (1mm);
\path[fill, color=black] (v2) circle (1mm);
\path[fill, color=black] (v3) circle (1mm);
\path[fill, color=black] (v4) circle (1mm);
\end{tikzpicture}
}

\def\bltwoprojplane{
\begin{tikzpicture}[scale=0.4]
\sffamily
\coordinate(oo) at (0,0);
\coordinate(v1) at (1,0);
\coordinate(v2) at (0,1);
\coordinate(v3) at (-1,-1);
\coordinate(v4) at (0,-1);
\coordinate(v5) at (1,1);
\coordinate(r1) at (2,0);
\coordinate(r2) at (0,2);
\coordinate(r3) at (-2,-2);
\coordinate(r4) at (0,-2);
\coordinate(r5) at (2,2);
\path[fill, color=gray!30] (oo) -- (r1) -- (r5) -- (r2) -- (r3) -- (r4) -- (r1) -- (oo);  
\path[fill, color=gray!50] (oo) -- (v1) -- (v5) -- (v2) -- (v3) -- (v4) -- (v1) -- (oo);  
\draw[color=black] (oo) to (r1);
\draw[color=black] (oo) to (r2);
\draw[color=black] (oo) to (r3);
\draw[color=black] (oo) to (r4);
\draw[color=black] (oo) to (r5);
\draw[thick, color=black] (v1) to (v5);
\draw[thick, color=black] (v5) to (v2);
\draw[thick, color=black] (v2) to (v3);
\draw[thick, color=black] (v3) to (v4);
\draw[thick, color=black] (v4) to (v1);
\path[fill, color=black] (v1) circle (1mm);
\path[fill, color=black] (v2) circle (1mm);
\path[fill, color=black] (v3) circle (1mm);
\path[fill, color=black] (v4) circle (1mm);
\path[fill, color=black] (v5) circle (1mm);
\end{tikzpicture}
}

\def\blthreeprojplane{
\begin{tikzpicture}[scale=0.4]
\sffamily
\coordinate(oo) at (0,0);
\coordinate(v1) at (1,0);
\coordinate(v2) at (0,1);
\coordinate(v3) at (-1,-1);
\coordinate(v4) at (0,-1);
\coordinate(v5) at (1,1);
\coordinate(v6) at (-1,0);
\coordinate(r1) at (2,0);
\coordinate(r2) at (0,2);
\coordinate(r3) at (-2,-2);
\coordinate(r4) at (0,-2);
\coordinate(r5) at (2,2);
\coordinate(r6) at (-2,0);
\path[fill, color=gray!30] (oo) -- (r1) -- (r5) -- (r2) -- (r6) -- (r3) -- (r4) -- (r1) -- (oo);  
\path[fill, color=gray!50] (oo) -- (v1) -- (v5) -- (v2) -- (v6) -- (v3) -- (v4) -- (v1) -- (oo);  
\draw[color=black] (oo) to (r1);
\draw[color=black] (oo) to (r2);
\draw[color=black] (oo) to (r3);
\draw[color=black] (oo) to (r4);
\draw[color=black] (oo) to (r5);
\draw[color=black] (oo) to (r6);
\draw[thick, color=black] (v1) to (v5);
\draw[thick, color=black] (v5) to (v2);
\draw[thick, color=black] (v2) to (v6);
\draw[thick, color=black] (v6) to (v3);
\draw[thick, color=black] (v3) to (v4);
\draw[thick, color=black] (v4) to (v1);
\path[fill, color=black] (v1) circle (1mm);
\path[fill, color=black] (v2) circle (1mm);
\path[fill, color=black] (v3) circle (1mm);
\path[fill, color=black] (v4) circle (1mm);
\path[fill, color=black] (v5) circle (1mm);
\path[fill, color=black] (v6) circle (1mm);
\end{tikzpicture}
}
 
\begin{center}
\quad
\plinepline
\quad  
\projplane
\quad  
\bloneprojplane
\quad  
\bltwoprojplane
\qquad  
\blthreeprojplane
\qquad
\qquad
\end{center}

\noindent
The surfaces are $\PP_1 \times \PP_1$, the
projective plane $\PP_2$ and the blowing-up of $\PP_2$
in up to three points in general position.
\end{example}

The situation gets much more lively if
we look at singular del Pezzo surfaces.
We will consider singularities coming from the minimal
model program. Let us briefly recall their
definition. Given any $\QQ$-Gorenstein variety $X$
with canonical divisor~$\mathcal{K}_X$, consider
a resolution of singularities
$\pi \colon X' \to X$. Then there are canonical
divisors $\mathcal{K}_X$ on $X$ and
$\mathcal{K}_X'$ on $X'$ such that the
\emph{ramification formula}
\[
\mathcal{K}_X'
\ = \
\pi^* \mathcal{K}_X + \sum a(E) E
\]
holds. Here, $E$ runs through the exceptional
prime divisors and the $a(E) \in \QQ$ are the
\emph{discrepancies} of $\pi \colon X' \to X$.
For $0 \le \varepsilon \le 1$, the
singularities of $X$ are called
\emph{$\varepsilon$-log terminal
($\varepsilon$-log canonical)} if
$a(E) > \varepsilon-1$ for each $E$
($a(E) \ge \varepsilon-1$ for each~$E$). For
$\varepsilon=0$, one speaks of log terminal
(log canonical) singularities and for
$\varepsilon=1$ of terminal (canonical)
ones.

\begin{summary}
Consider the toric del Pezzo surface $Z_\mathcal{A}$
defined by an LDP-polygon $\mathcal{A}$ and a toric
resolution of singularities
\[
\pi \colon Z \ \to \ Z_{\mathcal A}
\]
given by a map of lattice fans $\Sigma$ and
$\Sigma_{\mathcal{A}}$. The discrepancy of an
exceptional divisor $E_\varrho \subseteq Z$
corresponding to a ray $\varrho \in \Sigma$ is
given by
\[
a(E_\varrho) 
\ = \ 
\frac{\Vert v_\varrho \Vert}{\Vert v_\varrho'  \Vert} - 1,
\]
where $v_\varrho \in \varrho$ is the primitive
lattice vector and $v_\varrho' \in \varrho$
denotes the intersection point of $\varrho$ and 
the boundary $\partial \mathcal{A}$ of
$\mathcal{A}$.  
\end{summary}

This shows in particular that toric
del Pezzo surfaces are always log terminal, hence
the ``L'' in LDP-polygon.
Moreover, a toric del Pezzo surface $Z$ of
Gorenstein index $\iota_Z$ is $1/k$-log canonical
for all $k \ge \iota_Z$.
Finally, we obtain the following characterization of
$1/k$-log canonicity for a toric del
Pezzo surface via its polygon.

\begin{definition}
Given $k \in \ZZ_{\ge 1}$, we say that a polygon
$\mathcal{A} \subseteq \QQ^2$ is \emph{almost $k$-hollow}
if $\mathcal{A}^\circ \cap k \ZZ^2 = \{(0,0)\}$ holds.  
\end{definition}

\begin{proposition}
\label{prop:khollow}
Let $\mathcal{A}$ be an LDP-polygon and consider the
associated toric del Pezzo surface $Z_{\mathcal{A}}$.
Then, for any $k \in \ZZ_{\ge 1}$, the following
statements are equivalent.
\begin{enumerate}
\item
The polygon $\mathcal{A}$ is almost $k$-hollow.
\item
The surface $Z_{\mathcal{A}}$ has only
$1/k$-log canonical singularities.
\end{enumerate}
\end{proposition}


\section{Classifying toric del Pezzo surfaces}

The aim of this section is to classify
$1/k$-log canonical toric del Pezzo surfaces.
According to Proposition \ref{prop:khollow},
the task is to classify almost $k$-hollow
LDP-polygons.
With Algorithm \ref{algo:khollow}, we present
a sufficently efficient classification procedure
for given $k$ and Theorem \ref{thm:polygonclass}
gathers explicit key data of the polygon classification
for $k=1,2,3$.
Corollary \ref{cor:toricdelpezzo} presents the
corresponding classification statements on del
Pezzo surfaces.

The general idea behind the Classification
Algorithm \ref{algo:khollow} is to build up almost
$k$-hollow LDP-polygons from a finite number
of minimal ones by successively adding further
vertices.
By suitable boundedness statements we ensure
that this becomes a finite procedure and allows
us to reach indeed all almost $k$-hollow
LDP-polygons up to unimodular equivalence.

\begin{definition}
Given a polygon $\mathcal{A}$ and a vector
$v \in \QQ^2$, we obtain new polygons
$\mathcal{A}^+_v$ by \emph{expanding}
$\mathcal{A}$ at $v$ and $\mathcal{A}^-_v$ by
\emph{collapsing} $\mathcal{A}$ at $v$:
\[
\mathcal{A}^+_v \ := \ \conv(\mathcal{A} \cup \{v\}),
\qquad
\qquad
\mathcal{A}^-_v \ := \ \conv(\ZZ^2 \cap \mathcal{A} \setminus \{v\}).
\]
A polygon $\mathcal{A}$ having the origin $(0,0)$ as
an interior point is \emph{minimal} if
$(0,0) \notin (\mathcal{A}^{-}_{v})^{\circ}$ holds for
every vertex $v$ of $\mathcal{A}$.
\end{definition}

\begin{example}
Consider the canonical basis vectors $e_{1} = (1,0)$
and $e_{2} = (0,1)$ in~$\mathbb{Z}^{2}$. Given
$k \in \mathbb{Z}_{\ge 1}$, we have minimal almost
$k$-hollow lattice polygons
\[
\Delta_{\alpha} :=  \conv(e_1,\, e_2,\, -\alpha e_{1}-e_{2}),
\
\alpha = 1, \ldots, 2k,
\qquad
\Box := \conv(\pm e_1, \, \pm e_2).
\]
\end{example}

\begin{proposition}
\label{prop:min-k-hollow}
Let $k \in \ZZ_{\ge 1}$. Then, up to unimodular
equivalence, $\Delta_{1}, \ldots, \Delta_{2k}$
and $\Box$ are the only minimal almost
$k$-hollow lattice polygons.
\end{proposition}

\begin{lemma}
\label{lem:min2vertices}
Let $\mathcal{A}$ be a minimal polygon. Then
$\mathcal{A}$ has at most four vertices.
\end{lemma}

\begin{proof}
We number the vertices $v_1,\ldots, v_r$ of $\mathcal{A}$ 
counter-clockwise. Assume $r \ge 5$. Then we have
\[
\mathcal{A}^\circ
\ = \
\conv(v_1,v_3, \ldots , v_r)^\circ
\cup
\conv(v_1,\ldots, v_{r-1})^\circ
\ \subseteq \
(\mathcal{A}^-_{v_2})^\circ
\cup
(\mathcal{A}^-_{v_r})^\circ.
\]
The origin lies in $\mathcal{A}^\circ$ and thus in
$(\mathcal{A}^-_{v_2})^\circ$ or in
$(\mathcal{A}^-_{v_r})^\circ$. This is a 
contradiction to the minimality of $\mathcal{A}$.
\end{proof}

\begin{proof}[Proof of Proposition \ref{prop:min-k-hollow}]
Let $\mathcal{A}$ be a minimal almost $k$-hollow lattice
polygon. By Lemma \ref{lem:min2vertices}, we know that
$\mathcal{A}$ has either four or three vertices.

We first treat the case, that $\mathcal{A}$
has four vertices. Say $v_1,v_2,v_3,v_4$, ordered
counter-clockwise. We claim that the two diagonals of
the quadrangle $\mathcal{A}$ intersect in the origin.
Indeed, minimality of $\mathcal{A}$ yields
\[
(0,0)
\ \not\in \
\conv(v_1,v_2,v_3)^\circ
\ \subseteq \
(\mathcal{A}^-_{v_4})^\circ,
\quad 
(0,0)
\ \not\in \
\conv(v_1,v_3,v_4)^\circ
\ \subseteq \
(\mathcal{A}^-_{v_2})^\circ.
\]
Since $(0,0) \in \mathcal{A}^\circ$, we have
$(0,0) \in \conv(v_1,v_3)$. Analogously,
$(0,0) \in \conv(v_2,v_4)$. Using minimality of
$\mathcal{A}$ again, we see that each vertex
$v_i \in \ZZ^2$ is primitive. Thus, after
applying a suitable unimodular transformation
of $\ZZ^2$, we can assume
\[
v_1 = (1,0),
\qquad
v_2 = (a,b), \quad a,b \in \ZZ, \ 0 \le a < b, \ \gcd(a,b) = 1.
\]
By primitivity, we also have $v_3 = -v_1$ and
$v_4 = -v_2$. Finally, we claim $v_2 = (0,1)$.
Otherwise $(1,1) \in \mathcal{A}$, which
yields a contradiction to the minimality
of $\mathcal{A}$, namely
\[
(0,0)
\ \in \
\conv(e_1,\, e_1+e_2, \, -e_1, -e_1-e_2)^\circ
\ \subseteq \
(\mathcal{A}^-_{v_4})^\circ .
\]

Now assume that $\mathcal{A}$ has three
vertices, say $v_1,v_2,v_3$, numbered
counter-clockwise. As in the previous case,
each of these vertices has to be primitive by
minimality.
We first assume that
$\mathcal{A}$ has an interior lattice point
$v \ne (0,0)$. Then we can write
\[
\mathcal{A}
\ = \
\conv(v_1,v_2,v)
\cup
\conv(v_2,v_3,v)
\cup
\conv(v_3,v_1,v).
\]
By minimality, $(0,0)$ lies on a line segment,
say on $\conv(v,v_1)$. Then we can assume
$v=-v_1$. Applying a suitable unimodular
transformation gives
\[
v_1 = (1,0),
\qquad
v_2 = (a,b), \quad a,b \in \ZZ, \ 0 \le a < b, \ \gcd(a,b) = 1.
\]
We show that $\mathcal{A}$ equals $\Delta_{\alpha}$
for some $1 \le \alpha \le 2k$. We claim $a=0$.
Otherwise, $(1,1) \in \mathcal{A}$ and, in
contradiction to the minimality of $\mathcal{A}$
we obtain
\[
(0,0)
\ \in \
\conv(v_{1}, \, e_{1}+e_{2}, \, -v_{1}, v_{3})^{\circ}
\ \subseteq \
(\mathcal{A}^-_{v_2})^{\circ}.
\]
Hence $a=0$ and $v_2 = e_2$. Moreover,
$v_{3} = (c,d)$ with integers $c,d < 0$. Since
$-v_1$ is an interior point of $\mathcal{A}$,
the vertex $v_3$ lies above the line through
$v_2$ and $-v_1$. In particular, $c \le -3$.
By minimality of $\mathcal{A}$, no point
$(e,-1)$ with $e \in \ZZ_{\le 0}$ lies in
$\mathcal{A}^\circ$. This yields $d = -1$.

We are left with the situation that  $\mathcal{A}$
has the three vertices $v_1,v_2,v_3$ and the
origin is its only interior lattice point.
As before, we adjust by means of a suitable
unimodular transformation to
\[
v_1 = (1,0),
\qquad
v_2 = (a,b), \quad a,b \in \ZZ, \ 0 \le a < b, \ \gcd(a,b) = 1.
\]
Minimality implies $a \le 1$. If $a=0$, the
only possibility is $v_2=e_2$ and $v_3$
being one of $(-1,-1)$, $(-2,-1)$ or
$(-1,-2)$. If $a=1$, we have $v_2 = (2,1)$
and $v_3 = (-1,-1)$. For each of these
constellations, $\mathcal{A}$ is
equivalent to a $\Delta_\alpha$ for
$\alpha=1,2$.
\end{proof}

We are ready to provide effective
bounds for almost $k$-hollow lattice polygons.
Below, $D_r \subseteq \RR^2$ denotes the disk of
radius $r$ centered around the origin $(0,0)$.

\begin{proposition}
\label{prop:almost-k-hollow}
Let $\mathcal{A}$ be an almost $k$-hollow lattice polygon.
Then there is a unimodular transformation
$\varphi \colon \RR^2 \to \RR^2$ such that one of the
following holds.
\begin{enumerate}
\item
We have $\Box \subseteq \varphi(\mathcal{A}) \subseteq D_r$ 
for $r=k^2 / \sqrt{2}$.
\item
We have $\Delta_\alpha \subseteq \varphi(\mathcal{A}) \subseteq D_r$
for $1 \le \alpha \le 2k$ and
\[
r=2k^2\sqrt{\alpha^2+2\alpha+2}.
\]
\end{enumerate}
\end{proposition}

\begin{lemma}
\label{lem:minkowski}
Let $\mathcal{A}$ be an almost $k$-hollow polygon and
$r \in \RR$ with $D_r \subseteq \mathcal{A}$ and let
$v \in \RR^2$. If the extension $\mathcal{A}_v^+$ is
almost $k$-hollow, then $\Vert v \Vert \le k^2 /2r$.
\end{lemma}

\begin{proof}
First consider the case $k=1$. Then the origin is the
only interior lattice point of $\mathcal{A}$. Being
contained in $\mathcal{A}_v^+$, the set
$\conv(D_r \cup \{v\})$ has the origin as its only
interior lattice point. Thus, also
$\mathcal{B} := \conv(D_r \cup \{\pm v\})$ has the
origin as its only interior lattice point. Since
$\mathcal{B}$ is a centrally symmetric convex set,
Minkowski's Theorem yields $\vol(\mathcal{B}) \le 4$.
Moreover, we directly see
$2r \Vert v \Vert \le \vol(\mathcal{B})$. This proves
the assertion for $k=1$. For the case of a general
$k$, apply the previous consideration to
$k^{-1} \mathcal{A}$ and $k^{-1} v$.
\end{proof}

\begin{proof}[Proof of Proposition \ref{prop:almost-k-hollow}]
First observe that successively collapsing
$\mathcal{A}$ at vertices,  we arrive at a minimal
almost $k$-hollow lattice polygon $\mathcal{A}'$.
Proposition \ref{prop:min-k-hollow} provides
a unimodular transformation 
$\varphi \colon \RR^2 \to \RR^2$ such that
$\varphi(\mathcal{A}')$ is one of the polygons
mentioned in (i) and (ii). It remains to show that
$\varphi(\mathcal{A})$ lies in the corresponding
disks. For this, we have to bound the length of
any given vertex $v \in \varphi(\mathcal{A})$
accordingly. Collapsing step by step suitable
vertices turns $\varphi(\mathcal{A})$ into
$\varphi(\mathcal{A}')_v^+$. Thus the assertion
follows from Lemma \ref{lem:minkowski} and the
fact that we find a disk $D_r$ of radius
$r = 1/\sqrt{2}$ in $\Box$ and of radius
$r = 1/\sqrt{\alpha^2+2\alpha+2}$ in
$\Delta_\alpha$.
\end{proof}

\begin{corollary}
\label{cor:globalbound}
Up to unimodular transformation every almost $k$-hollow
lattice polygon is obtained by stepwise extending
almost $k$-hollow lattice polygons inside~$D_R$
starting with $\Box$ and $\Delta_\alpha$ for
$R := R(k) := k^2 \sqrt{4k^2+4k+2}$.
\end{corollary}

In particular, this allows to construct all
almost $k$-hollow lattice polygons up to unimodular
equivalence. A naive way is to extend $\Box$ and
$\Delta_\alpha$ by lattice points from the box
$\conv(\pm Re_1+\pm Re_2)$ with
$R=k^2 \sqrt{4k^2+4k+2}$ and check in each step for
almost $k$-hollowness. The following principle allows
more target-oriented searching.

\begin{remark}
\label{rem:shadow}
Consider a box $\mathcal{B} := \conv(\pm Re_{1}, \pm Re_{2})$
in $\RR^2$ and an almost $k$-hollow lattice polygon
$\mathcal{A} \subseteq \mathcal{B}$. The \emph{shadow} of
$w \in k\ZZ^2$ with respect to $\mathcal{A}$ is
\[
S(w,\mathcal{A})
\ := \
\cone(w-u; \ u \in \mathcal{A})^\circ + w
\ \subseteq \
\RR^2.
\]
This is an open affine cone in $\RR^2$. The lattice
vectors $v \in \mathcal{B} \cap \ZZ^2$ such that
$\mathcal{A}^+_v$ is almost $k$-hollow are all located
in the star-shaped set
\[
\Xi(\mathcal{A})
\ := \
\bigcap_{0 \ne w \in k\ZZ^2} 
\mathcal{B} \setminus S(w, \mathcal{A})
\ \subseteq \
\mathcal{B}.
\]
Each $S(w,\mathcal{A})$ is described by two
inequalities. This leads to an explicit description
of the searching space $\Xi(\mathcal{A})$ and thus allows
computational membership tests.
\end{remark}

\begin{example}
In the case $k=2$ and $\mathcal{A} = \Box$,
the searching space $\Xi(\mathcal{A})$ for
possible $v \in \ZZ^2$ with $\mathcal{A}^+_v$
almost $2$-hollow is the white area in the figure
below.
\[
\includegraphics[scale=0.15]{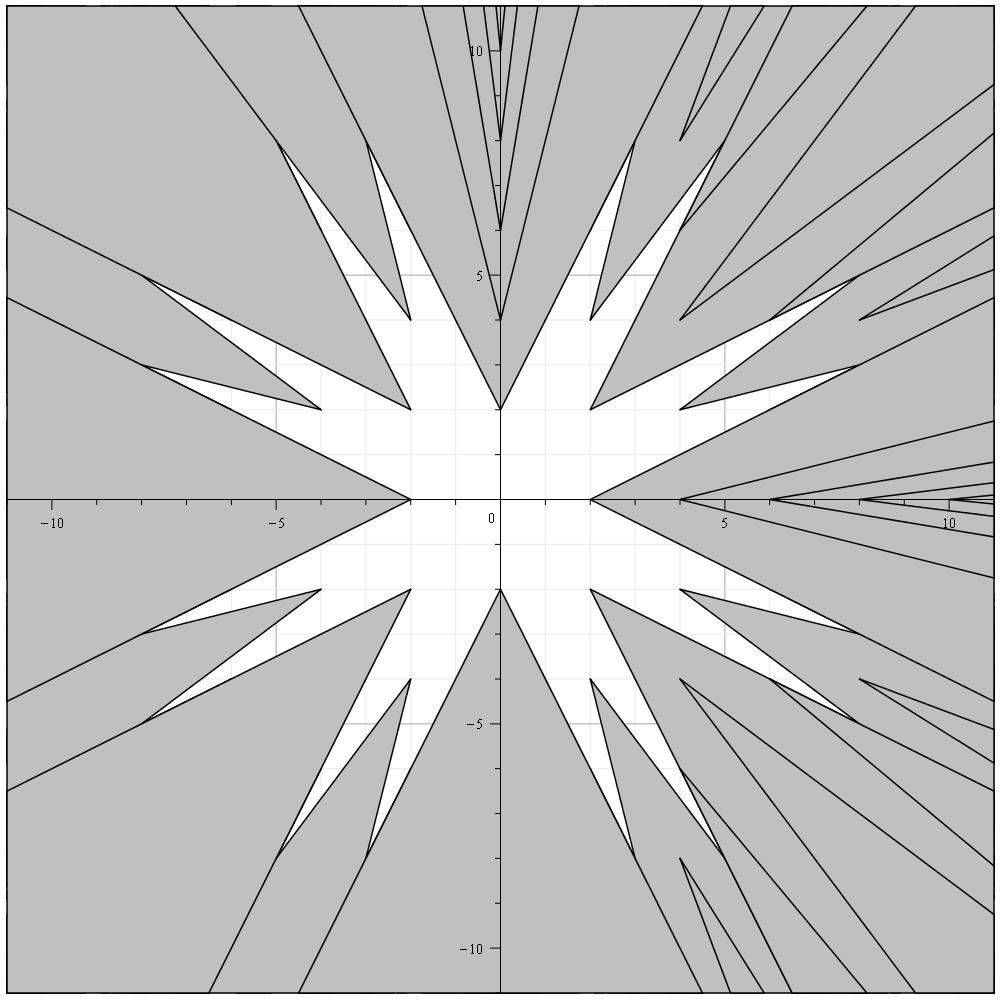}
\]
\end{example}

With the bound provided by Corollary \ref{cor:globalbound}
and the searching spaces from Remark~\ref{rem:shadow},
we obtain the following procedure for computing the almost
$k$-hollow polygons.

\begin{algorithm}
\label{algo:khollow}
Given $k \in \ZZ_{>0}$, set $R:=k^2 \sqrt{4k^2+4k+2}$.
The following procedure terminates and delivers
a list $L$ of representatives for the almost $k$-hollow
polygons up to unimodular equivalence:
\begin{itemize}
\item
Start with $i := 0$ and the list
$L_0 = [\Box, \Delta_{1}, \ldots, \Delta_{2k}]$
of minimal almost $k$-hollow lattice polygons.  
\item
As long as the list $L_i$ is non-empty, repeat the
following:
Increment $i$ to $i+1$ and compute the list $L_{i}$ of
all almost $k$-hollow polygons $\mathcal{A}^+_v$ with $\mathcal{A}$
from $L_{i-1}$ and
$v \in \ZZ^2 \cap (\Xi(\mathcal{A}) \setminus \mathcal{A})$.  
\item 
Take the concatenation $L'$ of all lists $L_i$ and reduce
it to a list $L$ containing each
entry from $L'$ precisely once up to unimodular equivalence.
Return the reduced list $L$.
\end{itemize} 
\end{algorithm}

Using an ad hoc implementation of this procedure,
we classified the almost $k$-hollow lattice polygons
for $k = 1,2,3$. The following theorem gathers some
key data of the classification; the complete classification
list can be extracted from~\cite{HaHaSp}.

\goodbreak

\begin{theorem}
\label{thm:polygonclass}
We obtain the following statements on almost
$k$-hollow LDP-polygons.
\begin{itemize}
\item[$k=1$:]
There are up to unimodular equivalence exactly
$16$ almost $1$-hollow LDP-polygons. These are
the well known reflexive polygons. The maximum
number of vertices is $6$ and the maximum
volume is $\frac{9}{2}$.

\begin{small}
\begin{center}
\begin{tabular}{l|c|c|c|c}
vertices & $3$  & $4$ & $5$ & $6$
\\
\hline
number & $5$ & $7$ & $3$ & $1$
\\
\hline
max.vol. & $\frac{9}{2}$ & 4 & $\frac{7}{2}$ & 3
\end{tabular}

\end{center}
\end{small}

\item[$k=2$:]
There are up to unimodular equivalence exactly
$505$ almost $2$-hollow LDP-polygons. The
maximum number of vertices is $8$, realized by
exactly one polygon, and the maximum volume
is $17$.

\medskip

\begin{small}
\begin{center}
\begin{tabular}{l|c|c|c|c|c|c}
vertices & $3$  & $4$ & $5$ & $6$ & $7$ & $8$
\\
\hline
number & $42$ & $181$ & $202$ & $74$ & $5$ & $1$ 
\\
\hline
max.vol. &  $16$ & $16$ & $17$ & $\frac{33}{2}$ & $14$ & $14$
\end{tabular}  

\medskip

\end{center}
\end{small}

\item[$k=3$:]
There are up to unimodular equivalence exactly
$48032$ almost $3$-hollow LDP-polygons. The
maximum number of vertices is $12$, realized
by exactly one polygon, and the maximum volume
is $47$.

\medskip

\begin{tiny}
\begin{center}
\begin{tabular}{l|c|c|c|c|c|c|c|c|c|c}
vertices & $3$  & $4$ & $5$ & $6$ & $7$ & $8$ & $9$ & $10$ & $11$ & $12$
\\
\hline
number & $355$ & $3983$ & $13454$ & $17791$ & $9653$ & $2456$ & $292$ & $37$ & $1$ & $1$ 
\\
\hline
max.vol. &  $44$ & $\frac{91}{2}$ & $47$ & $43$ & $39$ & $35$ & $30$ & $29$ & $\frac{47}{2}$ & $24$ 
\end{tabular}
\end{center}
\end{tiny}
\end{itemize}

\end{theorem}

Here is the translation of Theorem
\ref{thm:polygonclass} into the setting of toric
log del Pezzo surfaces.
Note that the case $\varepsilon=1$ delivers
precisely the 16 Gorenstein toric log del
Pezzo surfaces.
Moreover, the cases $\varepsilon = 1/2$ and
$\varepsilon = 1/3$
comprise in particular the toric log del Pezzo
surfaces of Gorenstein index 2 and 3 which
have been found in~\cite{KaKrNi}.

\begin{corollary}
\label{cor:toricdelpezzo}
We obtain the following statements on toric
$\varepsilon$-log canonical del Pezzo surfaces.
\begin{itemize}
\item[$\varepsilon=1$:]
Up to isomorphy there are exactly $16$ toric
canonical del Pezzo surfaces. These are the well
known toric Gorenstein del Pezzo surfaces. The
maximum Picard number is $4$, realized by exactly
one surface.
\item[$\varepsilon=\frac{1}{2}$:]
Up to isomorphy there are exactly $505$ toric
$\frac{1}{2}$-log canonical del Pezzo surfaces.
The maximum Picard number is $6$, realized by
exactly one surface.
\item[$\varepsilon=\frac{1}{3}$:]
Up to isomorphy there are exactly $48032$ toric
$\frac{1}{3}$-log canonical del Pezzo surfaces.
The maximum Picard number is $10$, realized by
exactly one surface.
\end{itemize}
\end{corollary}


\section{Rational $\KK^*$-surfaces}
\label{sec:kstar-surfaces}

A \emph{$\KK^*$-surface} is an irreducible,
normal surface $X$ endowed with an effective
morphical action $\KK^* \times X \to X$ of
the multiplicative group $\KK^*$. In this
section, we recall the basic background and
present our working environment for
rational $\KK^*$-surfaces: the Cox ring
based approach from \cite{HaSu,HaHe,HaWr};
see also\cite[Sec.~5.4]{ArDeHaLa}.
We begin with a brief reminder on the raw
geometric picture of projective
$\KK^*$-surfaces, the major part of which
has been drawn in the work of Orlik and
Wagreich~\cite{OrWa3}.

\begin{summary}
\label{sum:orwa}
Let $X$ be a projective $\KK^*$-surface.
We discuss the basic geometric properties
of the $\KK^*$-action. The possible
isotropy groups $\KK^*_x$, where
$x \in X$, are~$\KK^*$ itself and the
subgroups of order $l \in \ZZ_{\ge 1}$
consisting of the $l$-th roots of unity.
Thus, the non-trivial orbits are locally
closed curves of the form
\[
\KK^* \cdot x
\ \cong \
\KK^* / \KK^*_x
\ \cong \
\KK^*.
\]
There are three types of fixed points.
A fixed point $x \in X$ is
\emph{elliptic (hyperbolic, parabolic)} 
if it lies in the closure of 
infinitely many (precisely two,
precisely one) non-trivial
$\KK^*$-orbit(s). Elliptic and
hyperbolic fixed points are isolated,
whereas the parabolic fixed points
form a curve in $X$. The
\emph{limit points} $x_0$ and $x_\infty$
of an orbit $\KK^* \cdot x$ are obtained 
by extending the orbit map
$t \to t \cdot x$ to a morphism
$\varphi_x \colon \PP_1 \to X$ and
setting
\[
x_0 \ := \  \lim_{t \to 0} t \cdot x \ := \ \varphi_x(0),
\qquad\qquad
x_\infty \ := \ \lim_{t \to \infty} t \cdot x \ := \ \varphi_x(\infty).
\]
These limit points $x_0$, $x_\infty$
are fixed points and together with
$\KK^* \cdot x$ they form the closure
of the orbit $\KK^* \cdot x$. Every
projective $\KK^*$-surface $X$ has a
\emph{source} and a \emph{sink}, i.e.
irreducible components
$F^+, F^- \subseteq X$ of the fixed
point set admitting open
$\KK^*$-invariant neighborhoods
$U^+, U^- \subseteq X$ such that 
\[
x_0 \ \in \ F^+
\text{ for all } 
x \in U^+,
\qquad\qquad
x_\infty \in \ F^-
\text{ for all } 
x \in U^-.
\]
The source and sink each consist of
either a single elliptic fixed point 
or it is a smooth irreducible curve
of parabolic fixed points. Apart
from the source and the sink, we
find at most hyperbolic fixed points.
The raw geometric picture of a
projective $\KK^*$-surface $X$ is as 
follows.
\begin{center}
\begin{tikzpicture}[scale=0.6]
\sffamily
\coordinate(source) at (0,2);
\coordinate(sink) at (0,-2);
\coordinate(x0a) at (-2,1);
\coordinate(x0aa) at (-1.75,1.9);
\coordinate(x0z) at (-2,-1);
\coordinate(x0zz) at (-1.75,-2);
\coordinate(x1a) at (-1,1);
\coordinate(x1z) at (-1,-1);
\coordinate(xr-1a) at (1,1);
\coordinate(xr-1z) at (1,-1);
\coordinate(xra) at (2,1);
\coordinate(xraa) at (1.75,1.9);
\coordinate(xrz) at (2,-1);
\coordinate(xrzz) at (1.75,-2);
\coordinate(A0) at (-2.65,0);
\coordinate(Ar) at (2.65,0);
\coordinate(P0) at (-2,-4.5);
\coordinate(P1) at (2,-4.5);

\path[fill, color=black] (source) circle (.6ex) 
node[above, font=\scriptsize]{$F^+$};
\path[fill, color=black] (sink) circle (.6ex) 
node[below, font=\scriptsize]{$F^-$};
\path[fill, color=black] (x0a) circle (.5ex) node[left]{};
\path[fill, color=black] (x0z) circle (.5ex) node[left]{};
\path[fill, color=black] (x1a) circle (.5ex) node[left]{};
\path[fill, color=black] (x1z) circle (.5ex) node[left]{};
\path[fill, color=black] (xr-1a) circle (.5ex) node[left]{};
\path[fill, color=black] (xr-1z) circle (.5ex) node[left]{};
\path[fill, color=black] (xra) circle (.5ex) node[right]{};
\path[fill, color=black] (xrz) circle (.5ex) node[right]{};
\draw[thick, bend right=30] (source) to (x0a) node[]{};
\draw[thick, bend right=30] (source) to (x1a) node[]{};
\node[align=center, font=\scriptsize] (D01) at (x0aa)  {$D_{01}$};
\draw[thick, bend right=30] (x0z) to (sink) node[]{};
\draw[thick, bend right=30] (x1z) to (sink) node[]{};
\node[align=center, font=\scriptsize] (D0n0) at (x0zz)  {$D_{0n_0}$};
\draw[thick, bend left=30] (source) to (xr-1a) node[]{};
\draw[thick, bend left=30] (source) to (xra) node[]{};
\node[align=center, font=\scriptsize] (Dr1) at (xraa)  {$D_{r1}$};
\draw[thick, bend left=30] (xr-1z) to (sink) node[]{};
\draw[thick, bend left=30] (xrz) to (sink) node[]{};
\node[align=center, font=\scriptsize] (Drnr) at (xrzz)  {$D_{rn_r}$};
\draw[thick, dashed, bend right=20] (x0a) to (x0z) node[]{};
\draw[thick, dashed, bend right=20] (x1a) to (x1z) node[]{};
\draw[thick, dashed, bend left=20] (xr-1a) to (xr-1z) node[]{};
\draw[thick, dashed, bend left=20] (xra) to (xrz) node[]{};
\node[align=center, font=\scriptsize] (Anull) at (A0)  {$\mathscr{A}_0$};
\node[align=center, font=\scriptsize] (Aerr) at (Ar)  {$\mathscr{A}_r$};
\draw[->, thick, dashed] (0,-3) to (0,-4);
\node[align=center, font=\scriptsize] at (.3,-3.5)  {$\pi$};

\draw[thick] (P0) to (P1) node[]{};
\path[fill, color=black] (-2,-4.5) circle (.5ex);
\node[align=center, font=\scriptsize] at (-2.2,-5)  {$\pi(\mathscr{A}_0)$};
\path[fill, color=black] (-1,-4.5) circle (.5ex) node[left]{};
\path[fill, color=black] (1,-4.5) circle (.5ex) node[left]{};
\path[fill, color=black] (2,-4.5) circle (.5ex) node[left]{};
\node[align=center, font=\scriptsize] at (2.2,-5)  {$\pi(\mathscr{A}_r)$};
\end{tikzpicture}
\end{center}
The general orbit $\KK^* \cdot x \subseteq X$
has trivial isotropy group and connects the
source and the sink in the sense that its
closure contains one fixed point from
$F^+$ and one from  $F^-$. Besides the
general orbits, there are special
non-trivial orbits. Their closures are
rational curves $D_{ij} \subseteq X$
forming the \emph{arms}
\[
\mathscr{A}_i
\ := \
D_{i1} \cup \dots \cup D_{in_i}
\ \subseteq \ X,
\qquad i = 0 , \dots, r.
\]
The intersections $F^+ \cap D_{i1}$ and
$D_{in_i} \cap F^-$ each consist of a
fixed point and any two subsequent
$D_{ij}$, $D_{ij+1}$ intersect in a
hyperbolic fixed point. Finally, the
field of invariant rational functions
$\LL \subseteq \KK(X)$ yields a
projective curve $C$ and a surjective
rational map
\[
\pi \colon X \dasharrow C.
\]
It is defined everywhere except at
possible elliptic fixed points. The
$\KK^*$-surface $X$ is rational if and
only if $C = \PP_1$ holds. We will
also call $\pi \colon X \dasharrow C$
the (rational) quotient of $X$. The
critical fibers of $\pi$ are up to
elliptic fixed points precisely the
arms containing two or more
non-trivial orbits or an orbit with
non-trivial finite isotropy group.
\end{summary}

Our working environment for
a sufficiently explicit treatment is
the Cox ring based approach to
rational normal varieties with
torus action of complexity one from
\cite{HaSu,HaHe,HaWr}; see also
\cite{HaHiWr} for torus actions of
higher complexity. The following is a
play back of
\cite[Constr. 1.1, 1.4 and 1.6]{HaWr}
adapted to the surface case. The
procedure starts with a pair $(A,P)$
of matrices as defining data and
puts out a \emph{semiprojective}
rational $\KK^*$-surface $X(A,P)$
having only constant invariant
functions and coming embedded into a
toric variety $Z$. Here,
semiprojective means projective over
an affine variety. This comprises
for instance the projective and the
affine varieties.

\begin{construction}
\label{constr:kstarsurf}
Fix integers $r, n_0, \dots, n_r \ge 1$
and $0 \le m \le 2$. We construct a
rational semiprojective $\KK^*$-surface
$X$ coming from a $2 \times (r+1)$
matrix $A$ over $\KK$ and an integral
$(r+1) \times (n+m)$ matrix $P$, both
given as a list of columns
\[
A 
\ = \ 
[a_0, \dots, a_r],
\qquad
P
\  = \ 
[v_{01}, \dots, v_{0n_0},
\dots,
v_{r1}, \dots, v_{rn_r},
v_1,\dots, v_m].
\]
Here $r+1<n+m$ for
$n=n_{0}+\dots+n_{r}$. The columns of
$A$ are pairwise linearly independent
and those of $P$ generate $\QQ^{r+1}$
as a vector space. Moreover, the
columns~$v_{ij}$ of $P$ are pairwise
distinct and of the form
\[
v_{0j} = (-l_{0j},\dots,-l_{0j},d_{0j}),
\quad
v_{ij} = (0,\dots,0,l_{ij},0\dots,0,d_{ij}),
\quad i = 1,\dots,r,
\]
where $l_{ij}$ sits at the $i$-th place
for $i = 1, \dots, r$ and we always have
$\gcd(l_{ij},d_{ij})=1$. The columns
$v_k$ are pairwise distinct as well
and of the form
\[
v^+ \ := \ (0,\dots,0,1),
\qquad\qquad 
v^- \ := \ (0,\dots,0,-1).
\]
The idea is to let $X$ come embedded
into a toric variety $Z$. The
defining fan $\Sigma$ of~$Z$
can be written down comfortably
provided that $P$ is \emph{slope-ordered},
i.e.
\[
m_{i1} > \dots > m_{in_i}, 
\quad
m_{ij} := \frac{d_{ij}}{l_{ij}},
\quad
i = 0, \dots, r,
\
j = 1, \dots, n_i.
\]
This can always be achieved by suitable
numbering. In this setting, $\Sigma$
has the columns of $P$ as its rays and
there are the maximal cones
\[
\tau_{ij}
\ := \
\cone(v_{ij}, v_{ij+1})
\ \in \ \Sigma,
\qquad
i = 0, \dots, r, \ j = 1, \dots, n_i-1.
\]
Depending on the existence of $v^+$
and $v^-$ and the values of 
$\skr{m}^+ := m_{01}+\dots+m_{r1}$
and
$\skr{m}^- := m_{0n_{0}}+\dots+m_{rn_{r}}$,
we complement the collection of maximal
cones by

\smallskip

\hspace*{1cm}
$\tau_i^+ := \cone(v^+,v_{i1})$
for $i = 0, \dots, r$
if $P$ has a column $v^+$,

\smallskip

\hspace*{1cm}
$\tau_i^- := \cone(v^-,v_{in_i})$
for $i = 0, \dots, r$
if $P$ has a column $v^-$, 

\smallskip

\hspace*{1cm}
$\sigma^+ :=  \cone(v_{01}, \dots, v_{r1})$
if $\skr{m}^+ > 0$ and there is no $v^+$, 

\smallskip
\hspace*{1cm}
$\sigma^- :=  \cone(v_{0n_0}, \dots, v_{rn_r})$
if $\skr{m}^- < 0$ and there is no $v^-$.

\smallskip

\noindent
The toric variety $Z$ associated with
$\Sigma$ will be the ambient variety
for $X$. Consider $\KK^{n+m}$ with the
coordinate functions $T_{ij}$, $S_k$
and for $\iota = 0, \dots, r-2$ the
trinomials 
\[
g_\iota
\ := \
\det
\left[
\begin{array}{ccc}
T_{\iota}^{l_{\iota}} & T_{\iota+1}^{l_{\iota+1}} & T_{\iota+2}^{l_{\iota+2}}
\\
a_{\iota} & a_{\iota+1} & a_{\iota+2}
\end{array}
\right]
\ \in \
\KK[T_{ij},S_k],
\]
where $T_i^{l_i} := T_{i1}^{l_{i1}} \cdots T_{in_i}^{l_{in_i}}$.
Then, with
$\bar X := V(g_0, \dots, g_{r-2}) \subseteq \KK^{n+m} =: \bar Z$,
we have a commutative diagram 
\[
\xymatrix@R=1.5em{
{\bar{X}}
\ar@{}[r]|\subseteq
\ar@{}[d]|{\rotatebox[origin=c]{90}{$\scriptstyle\subseteq$}}
&
{\bar{Z}}
\ar@{}[d]|{\rotatebox[origin=c]{90}{$\scriptstyle\subseteq$}}
\\
{\hat{X}}
\ar[r]
\ar[d]_{/ H}^p
&
{\hat{Z}}
\ar[d]^{/ H}_p
\\
X
\ar[r]
&
Z.}
\]
Here $p \colon \hat Z \to Z$ is
Cox's quotient presentation from
\ref{constr:torcox}, we set
$\hat X := \bar X \cap \hat Z$ and
obtain a normal closed surface
\[
X \ := \ X(A,P) \ := \ p(\hat X) \ \subseteq \ Z.
\]
The $\KK^*$-action on the surface $X$
is obtained as $t \cdot x = \lambda(t) \cdot x$,
where $\lambda$ is the one-parameter subgroup
of the acting torus $\TT^{r+1}$ of $Z$ given by
\[
\lambda \colon \KK^* \ \to \ \TT^{r+1},
\qquad
t \ \mapsto \ = (1, \dots, 1, t).
\]
Altogether, we end up with a rational
semiprojective $\KK^*$-surface $X = X(A,P)$
with
$\Gamma(X,\mathcal{O})^{\KK^*} = \KK$
coming embedded into a toric variety
$Z$. Furthermore,
\begin{itemize}
\item
$X$ is projective if and only if the
columns of $P$ generate $\QQ^{r+1}$ as
a cone,
\item
$X$ is affine if and only if
$n_0 = \dots = n_r = 1$ and $P$ has
neither $v^+$ nor $v^-$.
\end{itemize}
\end{construction}

\begin{theorem}
See \cites{HaSu,HaWr}. Every rational,
semiprojective $\KK^*$-surface having
only constant invertible global
functions is equivariantly isomorphic
to a $\KK^*$-surface arising from
Construction \ref{constr:kstarsurf}.
\end{theorem}

\begin{remark}
\label{rem:toric-XAP}
For $r=1$, Construction
\ref{constr:kstarsurf} precisely
gives the semiprojective toric surfaces
with a $\KK^*$-action. To be more
specific, given $r=1$, the following
holds for $X = X(A,P)$.
\begin{enumerate}
\item
The fan $\Sigma$ in $\ZZ^2$ has convex
support and its one-dimensional cones
are the rays given by the columns of
$P$.
\item
The surface $X$ coincides with the
toric surface $Z$ and $\KK^*$ acts via
the one-parameter subgroup
$\KK^* \to \TT^2$, $t \mapsto (1,t)$.
\end{enumerate}
\end{remark}

Let us see how to extract the
general raw geometric picture from Summary~\ref{sum:orwa}
in the case of a rational projective
$\KK^*$-surface $X = X(A,P)$ from its
defining matrices~$A$ and $P$.

\begin{summary}
\label{rem:matrices-to-geometry}
Consider a projective $\KK^*$-surface
$X = X(A,P)$. Let
$D^{ij}_Z \subseteq Z$ denote the toric
prime divisor corresponding to the ray
generated by the column~$v_{ij}$ of
$P$. Then we have $r+1$ arms in $X$:
\[
\mathscr{A}_i
\ = \ 
D^{i1}_X \cup \dots \cup D^{in_i}_X,
\qquad
D^{ij}_X \ = \ D^{ij}_Z \cap X.
\]
The toric orbit of $Z$ corresponding
to a cone $\tau_{ij}$ cuts out the
hyperbolic fixed point in
$D_X^{ij} \cap D_X^{ij+1}$. According
to the possible constellations of
$v^+,v^-,\sigma^+$ and $\sigma^-$
source and sink look as follows.
\begin{itemize}
\item[(e-e)]
The fan $\Sigma$ has the cones
$\sigma^+$ and $\sigma^-$. Then the
associated toric orbits are elliptic
fixed points $x^+,x^- \in X$ forming
source and sink.
\item[(e-p)]
The fan $\Sigma$ has $\sigma^+$ as a
cone and $P$ has $v^-$ as a column. The
toric orbit given by $\sigma^+$ is an
elliptic fixed point $x^+ \in X$
forming the source. The toric divisor
$D_Z^-$ given by the ray through
$v^-$ cuts out a curve $D_X^-$ of
parabolic fixed points forming the
sink.
\item[(p-e)]
The matrix $P$ has $v^+$ as a column
and $\Sigma$ has $\sigma^-$ as a cone.
The toric divisor $D_Z^+$ given by the
ray through $v^+$ cuts out a curve
$D_X^+$ of parabolic fixed points
forming the source. The toric orbit
given by $\sigma^-$ is an elliptic
fixed point $x^- \in X$ forming the
sink.
\item[(p-p)]
The matrix $P$ has $v^+$ and $v^-$ as
columns. The toric divisors $D_Z^+$,
$D_Z^-$ given by the rays through
$v^+$, $v^-$ cut out curves $D_X^+$,
$D_X^-$ of parabolic fixed points
forming source and sink.
\end{itemize}
The entry $l_{ij}$ of $P$ equals the
order of the isotropy group $\KK^*_x$
of a general $x \in D_X^{ij}$ and the
associated $d_{ij}$ yields the weight
of the tangent representation of
$\KK^*_x$. Moreover, we retrieve the
quotient map
$\pi \colon X \dasharrow \PP_1$ from
the commutative diagram
\[
\xymatrix{
X
\ar@{}[r]|\subseteq
\ar@{-->}[d]_{/ \KK^*}^\pi
&
Z
\ar@{-->}[d]^{/ \KK^*}_\pi
\\
{\PP_1}
\ar@{}[r]|\subseteq
&
{\PP_r},
}
\]
where the r.h.s map arises from the projection
$\ZZ^{r+1} \to \ZZ^r$ onto the first~$r$ coordinates.
The homogeneous coordinates of $[a_{i1},a_{i2}]$
of $\pi(\mathscr{A}_i) \in \PP_1$ of
are given by the $i$-th column of the
matrix $A$.
Moreover, $\PP_1 \subseteq \PP_r$
is cut out by the linear forms
\[
h_\iota
\ := \
\det
\left[
\begin{array}{ccc}
U_{\iota} & U_{\iota+1} & U_{\iota+2}
\\
a_{\iota} & a_{\iota+1} & a_{\iota+2}
\end{array}
\right]
\ \in \
\KK[U_0,\dots,U_r],
\qquad
\iota = 0, \dots, r-2.
\]
Note that the rational quotient map
$\pi \colon X \dasharrow \PP_1$ is
defined everywhere apart from possible elliptic
fixed points~$x^{\pm}$.
Moreover, $\pi$ is categorical in the sense that
every $\KK^*$-invariant morphism
$X \setminus \{x^{\pm}\} \to Y$ factors uniquely
through $\pi$.
\end{summary}

\begin{remark}[Cox coordinates]
\label{rem:coxcoord2}
Every point $x \in X$ is of the form $x=p(z)$
with a point
\[
z
\ = \
(z_{ij},z^\pm)
\ \in \
\hat X
\ = \
\bar X \cap \hat Z
\ \subseteq \ \bar Z
\ = \
\KK^{n+m}.
\]
Here, $z = (z_{ij},z^\pm)$ is unique up
to multiplication by the quasitorus
$H$. We call $x = [z_{ij},z^\pm]$ a
presentation in \emph{Cox coordinates}.
For any $x = [z_{ij},z^\pm] \in X$, we
have
\begin{align*}
x \text{ general } \
&\Longleftrightarrow \ \text{ all }
z_{ij}, z^\pm \ne 0,
\\
x \in D^+ \
&\Longleftrightarrow \ \text{ only } z^+ = 0,
\\
x \in D^- \
&\Longleftrightarrow \ \text{ only } z^- = 0,
\\
x \in D_{ij} \
&\Longleftrightarrow \ \text{ only } z_{ij} = 0,
\\
x \in D_{ij} \cap D_{ij+1} \
&\Longleftrightarrow \ \text{ only } z_{ij} = z_{ij+1} = 0,
\\
x = x^+ \
&\Longleftrightarrow \ \text{ only } z_{01} = \dots = z_{r1} = 0,
\\
x = x^- \
&\Longleftrightarrow \ \text{ only } z_{0n_0} = \dots = z_{rn_r} = 0.
\end{align*}
By ``general'' we mean it is not contained
in any arm of $X$. This collects all
possibilities of vanishing and
non-vanishing of Cox coordinates of the
points of $X$.
\end{remark}

\begin{example}
\label{ex:runningexampleA1}
Given formatting data $r=2$,
$n_{0}=n_{1}=2,n_{2}=1$ and $m=1$
consider the following defining
matrices $A$ and $P$:
\[
A
 = 
\left[
\begin{array}{rrr}
1 & 0 & -1
\\
0 & 1 & -1
\end{array}    
\right],
\qquad
P
 = 
[v_{01},v_{02},v_{11},v_{21},v^+]
 = 
\left[
\begin{array}{ccccc}
-3 & -5 & 2 & 0 & 0 
\\
-3 & -5 & 0 & 2 & 0 
\\
-4 & -8 & 1 & 1 & 1 
\end{array}    
\right].
\]
For $i = 0,1,2$, the projection
$\pr \colon \ZZ^3 \to \ZZ^2$ onto the
first two coordinates sends the columns
$v_{ij}$ of the matrix $P$ into the
rays $\varrho_i \subseteq \QQ^2$
spanned by the vectors
\[
(-1,-1),
\qquad
(1,0),
\qquad
(0,1).
\]
These are the primitive generators of 
the fan $\Delta$ of the projective
plane $\PP_2$. For the fan $\Sigma$ we
obtain the picture  
\begin{center}
\begin{tikzpicture}[scale=0.6]
\sffamily
\coordinate(oo) at (0,-5);
\coordinate(e0) at (-1,-5);
\coordinate(e1) at (.5,-.3-5);
\coordinate(e2) at (1,0.5-5);
\coordinate(c0) at (-2,-5);
\coordinate(c1) at (1,-.6-5);
\coordinate(c2) at (2,1-5);
\coordinate(ooo) at (0,0);
\coordinate(v01) at (-2,-0.5);
\coordinate(v02) at (-2.5,-1);
\coordinate(v11) at (.7,2);
\coordinate(v12) at (1,-0.7);
\coordinate(v21) at (2.5,3);
\coordinate(vplus) at (0,1);
\coordinate(cplus) at (0,2);
\coordinate(c01) at (-2,-0);
\coordinate(c02) at (-2,-1.5);
\coordinate(c11) at (1,-2);
\coordinate(c21) at (2,-.75);
\node[] at (2.8,1) {$\Sigma$};
\node[] at (2.8,-5) {$\Delta$};
\node[] at (-0.7,-3.2) {$\pr$};
\path[fill, color=gray!15] (oo) -- (c2) -- (c0) -- (oo);
\path[fill, color=gray!20] (oo) -- (c1) -- (c2) -- (oo);
\path[fill, color=gray!10] (oo) -- (c0) -- (c1) -- (oo);
\path[fill, color=gray!40] (ooo) -- (c11) -- (c21) -- (ooo);
\path[fill, color=gray!30] (ooo) -- (cplus) -- (c21) -- (ooo);
\path[fill, color=gray!10] (ooo) -- (cplus) -- (c01) -- (c02) -- (ooo);
\path[fill, color=gray!30] (ooo) -- (c02) -- (c11) -- (ooo);
\draw[thick, color=black] (oo) to (c0);
\draw[thick, color=black] (oo) to (c1);
\draw[thick, color=black] (oo) to (c2);
\draw[thick, color=black] (ooo) to (c01);
\draw[thick, color=black] (ooo) to (c02);
\draw[thick, color=black] (ooo) to (c11);
\draw[thick, color=black] (ooo) to (c21);
\draw[thick, color=black] (ooo) to (cplus);
\draw[thick, dotted, color=black] (cplus) to (c01);
\draw[thick, dotted, color=black] (cplus) to (c21);
\draw[thick, dotted, color=black] (c01) to (c02);
\draw[thick, dotted, color=black] (c02) to (c11);
\draw[thick, dotted, color=black] (c11) to (c21);
\path[fill, opacity=.9, color=gray!20] (ooo) -- (cplus) -- (c11) -- (ooo);
\draw[thick, dotted, color=black] (cplus) to (c11);
\draw[dashed, color=black] (c02) to (c0);
\draw[dashed, color=black] (c11) to (c1);
\draw[dashed, color=black] (c21) to (c2);
\draw[->, thick, color=black] (0,-2.5) to (0,-4);
\end{tikzpicture}
\end{center}
\noindent
In terms of Cox coordinates,
the surface $X=X(A,P)$ sitting in the
toric variety~$Z$ defined by $\Sigma$
is given as
\[
\bar X
\ = \
V(T_{01}^3T_{02}^5 + T_{11}^2 + T_{21}^2)
\ \subseteq \
\bar Z
\ = \
\KK^5.
\]
The rational toric morphism
$\pi \colon Z \dasharrow \PP_2$ given
by $\pr \colon \ZZ^3 \to \ZZ^2$ is
defined everywhere except at the point
$x^- \in X \subseteq Z$. Restricting to
$X$ gives the map
\[
\pi \colon X
\ \dasharrow \
\PP_1
\ = \
V(S_0+S_1+S_2)
\ \subseteq \
\PP_2,
\]
which is the rational quotient. The
source $D^+ \subseteq X$ maps onto
$\PP_1$ and the critical values of
$\pi$ are cut out by $S_i = 0$ for
$i=0,1,2$.
\end{example}



\section{The isomorphy problem}

We consider the question whether or not
two pairs of defining matrices $(A,P)$
and $(A',P')$ give rise to isomorphic
$\KK^*$-surfaces.
In Proposition~\ref{prop:isochar}, we provide
a characterization in terms of what we call
admissible operations on defining matrices.
In Proposition~\ref{normalform}, we
present an efficiently computable normal form
for defining matrices, which is
used for the systematic data storage
in~\cite{HaHaSp}.
Moreover, the normal form directly shows that
for fixed $P$, the non-toric $X(A,P)$ come
in $(r-2)$-dimensional families;
see Remark~\ref{rem:families}.

A key feature of Construction~\ref{constr:kstarsurf}
is that it delivers for free for any $\KK^*$-surface
$X = X(A,P)$ the divisor class group
$\Cl(X)$ and the Cox ring
\[
\mathcal{R}(X)
\ = \
\bigoplus_{\Cl(X)} \Gamma(X,\mathcal{O}(D)).
\]
We refer to \cite[Sec.~1.1.4]{ArDeHaLa} for the
precise definition of the Cox ring and to
\cite[Sec.~3.4.3]{ArDeHaLa} for details of the
following.

\begin{summary}
\label{sum:classgroup-coxring}
Consider a $\KK^*$-surface $X = X(A,P) \subseteq Z$
as provided by
Construction~\ref{constr:kstarsurf}. Recall that
the columns $v_{ij}$ and $v^\pm$ of $P$ yield
prime divisors
\[
D_X^{ij}
\ = \
X \cap D_Z^{ij}
\ \subseteq \
X,
\qquad\qquad
D_X^\pm
\ = \
X \cap D_Z^\pm
\ \subseteq \
X.
\]
Moreover, every character function
$\chi^u \in \KK(Z)$ restricts to a non-zero rational
function on~$X$ with divisor given by
\[
\div(\chi^u)
\ = \
\sum_{i,j} \bangle{u,v_{ij}} D_X^{ij} + \sum_{+,-} \bangle{u,v^\pm} D_X^\pm.
\]
The terms of the second sum are understood to
equal zero if there is no $v^+$ or $v^-$.
The subvariety $X \subseteq Z$ inherits the
divisor class group from its ambient variety:
\[
\Cl(X) \ = \ \Cl(Z) \ = \ K \ := \ \ZZ^{n+m}/\image(P^t).
\]
Let $Q \colon \ZZ^{n+m} \to K$ be the projection
and let $e_{ij}$, $e^\pm$ be the canonical basis
vectors of $\ZZ^{n+m}$, indexed in accordance
with the variables $T_{ij}$ and $S^\pm$. Then
\[
\deg(T_{ij}) := w_{ij} := Q(e_{ij}) = [D_X^{ij}],
\qquad 
\deg(S^\pm) := w^\pm := Q(e^\pm) = [D_X^\pm]
\]
defines a $K$-grading on the polynomial ring
$\KK[T_{ij},S^\pm]$ such that the trinomials
$g_0,\dots,g_{r-2}$ are homogeneous. We have
isomorphisms
\[
\mathcal{R}(X) 
\ \cong \
R(A,P)
\ := \ 
\KK[T_{ij},S_k]/\bangle{g_0,\dots,g_{r-2}}
\ \cong \
\mathcal{R}(Z)/\bangle{g_0,\dots,g_{r-2}}
\]
of $K$-graded $\KK$-algebras. In particular,
the Cox ring of $X$ is a factor algebra of the
Cox ring of its ambient toric variety $Z$.
\end{summary}


We enter the isomorphy problem.
The subsequent two definitions
provide us with the necessary concepts 
to characterize of isomorphy of two given
$\KK^*$-surfaces in terms of their defining
matrices.

\begin{definition}
\label{adjustedP}
Consider a defining matrix $P$ as in
Construction~\ref{constr:kstarsurf}.
\begin{enumerate}
\item
If $l_{i1}n_{i} > 1$ for $i=0,\ldots,r$, then
we call $P$ \emph{irredundant}, else \emph{redundant}.
\item
We call a column $v_{i1}$ of $P$ \emph{erasable}
if $i>0$, $n_i=1$, $l_{i1}=1$ and $d_{i1}=0$.
\end{enumerate}
\end{definition}

The tuple $v_{i1}, \dots, v_{in_i}$ of columns
of a defining matrix $P$ describes the $i$-th arm of
any $X = X(A,P)$. Hence, we also refer to
$v_{i1}, \dots, v_{in_i}$
as the \emph{$i$-th arm} of the defining matrix
$P$.

\begin{definition}
\label{admop}
Consider pair $(A,P)$ of defining
matrices as in Construction~\ref{constr:kstarsurf}.
The \emph{admissible operations} on $(A,P)$
are the following:
\begin{enumerate}
\item
\label{blub}
\emph{erase} an erasable column
$v_{i1}$ by removing the $i$-th row and
the $i1$-th column from $P$ and the
$i$-th column of $A$,
\item
multiply the last row of $P$ by $-1$,
\item 
swap the columns $v_+$ and $v_-$,
\item
swap two columns $v_{ij_1}$ and
$v_{ij_2}$ inside the $i$-th arm of
$P$,
\item 
swap the $i$-th and $j$-th column of
$A$, the $i$-th and $j$-th arm of $P$
and rearrange the shape of $P$ by
elementary operations on the first $r$
rows,
\item 
add a multiple of one of the upper $r$
rows to the last row of $P$,
\item
transform $A$ into $BAD$ with
$B \in \GL_2(\KK)$ and diagonal
$D \in \GL_{r+1}(\KK)$,
\end{enumerate}
We say that two pairs $(A,P)$ and
$(A',P')$ of defining matrices are
\emph{equivalent} if we can transform
one of them into the other via
admissible operations. 
\end{definition}

\begin{remark}
\label{rem:admop2irredundant}
Every defining matrix $P$ can be turned
via admissible operations of types (i),
(v) and (vi) into an irredundant one.
\end{remark}

\begin{remark}
\label{rem:irredundaprops}
Consider a $\KK^*$-surface $X = X(A,P)$ with an
irredundant $P$. Then we have the following.
\begin{enumerate}
\item
The surface $X$ is isomorphic to a toric surface
if and only if $r = 1$ holds.
\item 
The arms of $X$ coincide with the critical fibers
of the map $\pi \colon X \dasharrow \PP_1$.
\end{enumerate}
\end{remark}

Recall that a \emph{morphism}
of $\KK^*$-surfaces $X$ and $X'$ is a
pair $(\varphi,\psi)$ with a morphism
$\varphi \colon X \to X'$ of varieties
and a homomorphism
$\psi \colon \KK^* \to \KK^*$ of
algebraic groups such that we always
have $\varphi(t \cdot x) = \psi(t)
\cdot \varphi(x)$. So, in this setting,
an equivariant morphism is a morphism
$(\varphi,\psi)$ with $\psi$ being the
identity.

\begin{example}
Given a rational projective
$\KK^*$-surface $X$, consider the
automorphism $\jmath(t)= t^{-1}$ of
$\KK^*$. Then $(\id_X,\jmath)$ is a
non-equivariant isomorphism $X \to X$
of $\KK^*$-surfaces, swapping the
source and the sink. For $X = X(A,P)$
the isomorphism $(\id_X,\jmath)$ is
given by multiplying the last row of
$P$ by $-1$.
\end{example}

\begin{proposition}
\label{prop:isochar}
Consider defining pairs $(A,P)$ and
$(A',P')$ of non-toric projective
$\KK^*$-surfaces $X$ and $X'$. Then the
following statements are equivalent.
\begin{enumerate}
\item
$(A,P)$ and $(A',P')$ are equivalent.
\item
$R(A,P)$ and $R(A',P')$ are isomorphic
as graded algebras.
\item
$X$ and $X'$ are isomorphic as
$\KK^*$-surfaces.
\item
$X$ and $X'$ are isomorphic as
surfaces.
\end{enumerate}
\end{proposition}

\begin{proof}
According to Summary~\ref{sum:classgroup-coxring},
the rings $R(A,P)$ and $R(A',P')$ are the Cox rings
of $X$ and $X'$ respecively.
Consequently,~(iv) implies~(ii). If~(ii) holds,
then $X$ and $X'$ have isomorphic
Cox rings.
As we are in the surface case, $X$ and $X'$ are
isomorphic to each other; use for
instance~\cite[Rem.~3.3.4.2 and Thm.~4.3.3.5]{ArDeHaLa}.
Thus,~(ii) implies~(iv).

Clearly~(iii) implies~(iv). For the converse,
it suffices to show that any two $\KK^*$-actions
on $X$ are conjugate by an automorphism of $X$.
By~\cite[Cor.~2.4]{ArHaHeLi}, the automorphism
group $\Aut(X)$ is linear algebraic and acts
morphically.
In particular, any $\KK^*$-action on $X$ defines
a one-dimensional torus in $\Aut(X)$.
As $X$ is non-toric, any two one-dimensional
tori in $\Aut(X)$ are maximal an hence conjugate
to each other.
We showed that~(iv) implies~(iii).

In order to see that~(i) implies~(ii), we have
to check that any admissible operation turning
a pair of defining matrices $(A,P)$ into
a pair $(A',P')$ induces an isomorphism
from $R(A,P)$ onto $R(A',P')$.
First observe that for operations of the
types~(ii) and~(vi) we even have
$R(A,P) = R(A',P')$.
Moreover, in case of an operation of 
type~(iii), (iv) or~(v), the ring $R(A',P')$
arises from $R(A,P)$ by renumbering the variables
and their degrees accordingly.
For operations of type (i), it suffices to consider
the case $i=r$. Then the trinomial~$g_{r-2}$ is of
the form
\[
g_{r-2}
\ = \
\alpha T_{r-2}^{l_{r-2}}+ \beta T_{r-1}^{l_{r-1}} +\gamma T_{r1},
\qquad
\alpha,\beta,\gamma\in\KK^*.
\]
Thus, sending $T_{r1}$ to
$-\gamma^{-1}(\alpha T_{r-2}^{l_{r-2}}+\beta T_{r-1}^{l_{r-1}})$ 
and leaving the other variables untouched defines 
a graded algebra isomorphism $R(A,P)\to R(A',P')$.
Finally, consider an operation of type~(vii). 
Multiplying $A$ with $B \in \GL_2(\KK)$ from the 
left means scaling the trinomials~$g_{\iota}$
by the factor $\det(B)$ and hence doesn't change
$R(A,P)$.
Moreover, multiplying $A$ with a diagonal matrix
from the right amounts to scaling the coefficients
in the trinomials $g_{\iota}$ by non-zero factors.
Appropriately rescaling the $T_{ij}$
yields an isomorphism from $R(A,P)$ onto $R(A',P')$.

We are left with verifying ``(iii)$\Rightarrow$(i)''.
We may assume that $P$ and $P'$ are irredundant.
Let $\varphi\colon X\to X'$  be an isomorphism of
$\KK^*$-surfaces.
Then $\varphi$ respects the isotropy groups
and, using for instance Remark~\ref{rem:irredundaprops},
we see that  $\varphi$ respects also the arms of the
$\KK^*$-surfaces.
We conclude $r=r'$, $n=n'$ and $m=m'$.
Summary~\ref{rem:matrices-to-geometry} tells us
that, after applying suitable operations
of types~(iv) and~(v),
$P$ and $P'$ coincide in the first~$r$ rows
and $\varphi(D_X^{ij})=D_{X'}^{ij}$ as well as
$\varphi(D_X^{\pm}) = D_{X'}^{\pm}$ hold.
This leads to a commutative diagram
\[
\xymatrix@R=1.5em{
0 \ar[r] & 
\ZZ^{r+1} \ar[r]^{P^t} \ar[d]_{S} & 
\ZZ^{n+m} \ar[r] \ar@{}[d]|{\rotatebox[origin=c]{90}{$=$}} & 
\mathrm{Cl}(X) \ar[r] \ar[d]_{\varphi_*} & 
0 \\
0 \ar[r] &
\ZZ^{r+1} \ar[r]_{(P')^t} &
\ZZ^{n+m} \ar[r] &
\mathrm{Cl}(X') \ar[r] & 
0
}
\]
with exact rows, where $\varphi_*$ denotes the
induce push forward isomorphism on the divisor
class groups.
The l.h.s. downwards map must be an isomorphism,
hence it is given by a unimodular matrix $S$,
satisfying $(P')^tS=P^t$, or equivalently, $S^tP'=P$. 
By the structure of $P$ and $P'$, the latter
forces $S^t$ to be of the form
\[
S^t
\ = \
\left[
\begin{array}{cc}
E_r & 0
\\
\ast & \pm 1
\end{array}
\right].
\]
This implies that $P$ and $P'$ differ only by admissible 
operations of types (ii) and~(vi).
It remains to show that $A$ and $A'$ differ only by an
admissible operation of type~(vii).
Stressing again Summary~\ref{rem:matrices-to-geometry},
we obtain a commutative diagram involving the rational
quotients
\[
\xymatrix@R=2em{
X \ar[r]^{\varphi}  \ar@{-->}[d]_{/\KK^*}^{\pi}
&
X' \ar@{-->}[d]^{/\KK^*}_{\pi'}
\\
\PP_1 \ar[r]_{B}
&
\PP_1,
}
\]
where the induced map $B$ is an isomorphism,
hence given by a matrix $B \in \GL_2(\KK)$
acting on the homogeneous coodinates.
The quotient maps $\pi$ send
the $i$-th arm of~$X$ to the point 
$a_i \in \PP_1$ with homogeneous coordinates
given by $i$-th column of $A$.
The same holds for $\pi'$, $X'$, $a_i'$
and $A'$.
Thus $B(a_i)$ and $a'_i$ differ only by a
non-zero scalar in their homogeneous coordinates. 
This implies $A'=BAD$ with a
diagonal matrix $D \in \GL_{r+1}(\KK)$.
\end{proof}

\begin{remark}
In Proposition~\ref{prop:isochar},
the assumption of $X$ and $X'$ being
non-toric is essential.
For instance, on the
projective plane $\PP_2$ the
$\KK^*$-actions $[z_0,tz_1,t^2z_2]$ and
$[z_0,z_1,tz_2]$ give rise to
non-isomorphic $\KK^*$-surfaces.
\end{remark}

We come to the normal form. Recall from
Construction~\ref{constr:kstarsurf} that
for every defining matrix $P$, we have
the \emph{slopes}
\[
m_{ij} := \frac{d_{ij}}{l_{ij}},
\quad
i = 0,\ldots, r, \ j = 1, \ldots, n_i.
\]

\begin{definition}
\label{mplusminus}
With any defining matrix $P$ as in 
Construction~\ref{constr:kstarsurf},
we associate the following integers
\[
\skr{b}^+_i \ := \ \max(\lfloor m_{i1}\rfloor,\ldots,\lfloor m_{in_i}\rfloor),
\qquad\qquad
\skr{b}^+ \ := \ \skr{b}_0^+ + \ldots + \skr{b}_r^+,
\]
\[
\skr{b}^-_i \ := \ \min(\lceil m_{i1}\rceil,\ldots,\lceil m_{in_i}\rceil),
\qquad\qquad
\skr{b}^- \ := \ -(\skr{b}_0^- + \ldots + \skr{b}_r^-).
\]
\end{definition}

Recall that a defining matrix $P$ is \emph{slope-ordered}
if $m_{i1} > \ldots > m_{in_i}$ holds for 
$i=0, \ldots, r$.
If $P$ is  slope-ordered, then the numbers just introduced
satisfy $\skr{b}^+_i=\lfloor m_{i1}\rfloor$
and $\skr{b}^-_i=\lceil m_{in_i}\rceil$.

\begin{definition}
\label{betaplusminus}
Consider a defining matrix $P$ as in Construction~\ref{constr:kstarsurf}. 
For any $i = 0 ,\ldots, r$ and $j = 1, \ldots, n_i$, we have the rational
numbers
\[
\beta_{ij}^+ \ := \ m_{ij}- \skr{b}^+_{i},
\qquad\qquad
\beta_{ij}^- \ := \ \skr{b}^-_{i} - m_{ij}.
\]
Let $\beta^{\pm}_i \in \QQ^{n_i}$ have the entries
$\beta^{\pm}_{i1},\ldots,\beta^{\pm}_{in_i}$ odered
descendingly and $\beta^{\pm}$ be the
lexicographical ordering of
$(\beta_0^{\pm}, \ldots, \beta_r^{\pm})$. 
Then $P$ is \emph{oriented} 
if it satisfies one of
\begin{enumerate}
\item[(O1)]
$P$ is of type (p-e),
\item[(O2)]
$P$ is of type (e-e) or (p-p) and $\skr{b}^+ > \skr{b^-}$,
\item[(O3)]
$P$ is of type (e-e) or (p-p) and $\skr{b}^+ = \skr{b^-}$ and 
$\beta^+ >_{\rm{lex}} \beta^-$,
\item[(O4)]
$P$ is of type (e-e) or (p-p) and $\skr{b}^+ = \skr{b^-}$ and 
$\beta^+ = \beta^-$.
\end{enumerate}
\end{definition}

Note that the properties (O1) to (O4) exclude each
other. 
If a defining matrix~$P$ is slope-ordered, then
$\beta^+_{i1}>\ldots>\beta^+_{in_i}$ and
$\beta^-_{i1}<\ldots<\beta^-_{in_i}$ hold.
In particular, the~$\beta_i^{\pm}$ are ordered descendingly
in that case, that means that
$\beta_i^{+}=(\beta_{i1}^{+},\ldots,\beta_{in_i}^{+})$
and $\beta_i^{-}=(\beta_{in_i}^{-},\ldots,\beta_{i1}^{-})$.

\begin{remark}
\label{mplusminusswap}
Let $P$ be a defining matrix as in Construction~\ref{constr:kstarsurf}.
Then $\skr{b}^{\pm}$ and~$\beta^{\pm}$ behave as follows with
respect to admissible operations:
\begin{enumerate}
\item
operations of type~(iii), (iv), (v), (vi) or~(vii)
don't affect  $\skr{b}^{\pm}$ and $\beta^{\pm}$,
\item
the operation of type~(ii) swaps $\skr{b}^+$, $\skr{b}^-$
and as well $\beta^+$, $\beta^-$.
\end{enumerate}
\end{remark}

\begin{remark}
\label{mplusminusswap2}
Applying the admissible operation of type~(ii),
we can turn any non-oriented matrix $P$  into
an oriented one.
\end{remark}

\begin{remark}
\label{mplusminusswap3}
Let $(A,P)$ and $(\tilde A,\tilde P)$ be equivalent
and assume that $P$ as well as $\tilde P$ are
irredundant.
Then we can transform $(A,P)$ into
$(\tilde A,\tilde P)$ via admissible operations
of types~(ii) to (vii).
According to the cases of using an even or an odd
number of operations of type~(ii) in this process,
we have
\[
(\skr{b}^+,\beta^+)
\ = \
(\tilde{\skr{b}}^+,\tilde{\beta}^+),
\qquad\qquad
(\skr{b}^+,\beta^+)
\ = \
(\tilde{\skr{b}}^-,\tilde{\beta}^-).
\]
\end{remark}

\begin{definition}
\label{normalformdef}
We say that a pair of defining matrices $(A,P)$ is in 
\emph{normal form}, if satisfies the following conditions:
\begin{enumerate}
\item
$P$ is irredundant,
\item
$P$ is oriented,
\item
if $m \ge 1$, then $d_1=1$,
\item
$P$ is slope-ordered, 
\item
$(\beta_0^+, \ldots, \beta_r^+)$ is lexicographically
ordered,
\item 
$P$ is \emph{adapted to the source}, i.e.
$0 \le d_{i1} < l_{i1}$ for $i=1, \ldots,r$,
\item
with suitable pairwise distinct
$\lambda_1=1, \lambda_2, \ldots \lambda_{r-1} \in \KK*$
we have
\[
A
\ = \ 
\left[
\begin{array}{cccccc}
1 & 0 & -\lambda_1 & -\lambda_2 & \ldots & -\lambda_{r-1}
\\
0 & 1 & -1 & -1 & \ldots & -1
\end{array}
\right].
\]
\end{enumerate}
\end{definition}

The main statements on the normal form of a pair
of defining matrices ensure existence and
uniqueness.

\begin{proposition}
\label{normalform}
Every pair of defining matrices is equivalent to
a pair of defining matrices in normal form.
Moreover, if two pairs of defining matrices in
normal form are equivalent to each other, then they
coincide.
\end{proposition}

\begin{proof}
We prove the first statement by presenting a concrete
procedure for turning a given pair of defining matrices
$(A,P)$ into normal form via admissible operations.
First, Remark~\ref{rem:admop2irredundant} makes~$P$
irredundant and Remark~\ref{mplusminusswap}~(ii) orients~$P$.
This establishes~\ref{normalformdef}.~(i) and~(ii).
Remark~\ref{mplusminusswap}~(i) tells us that
operations of types~(iii),~(iv) and~(v) 
ensure~\ref{normalformdef}.~(iii), (iv) and~(v)
without affecting $P$ to be oriented.
Subtracting the $\skr{b}^+_i$-fold of the
$i$-th row from the last one for $i=1,\ldots,r$ 
we achieve $0 \le d_{i1} < l_{i1}$.
In other word, using operations of type~(vi) lead 
to~\ref{normalformdef}.~(vi).
Again, the previous achievements are untouched,
as~$\skr{b}^{+}$ and~$\beta^{+}_{ij}$ stay.
Finally a suitable operation of type~(vii)
brings $A$ into the desired form. Linear
independence of the columns of $A$ guarantees
that $0,-1,\lambda_2, \ldots, \lambda_r$
are pairwise distinct.

We prove the second statement.
Let $(A,P)$ and $(\tilde A,\tilde P)$ be both in
normal form and equivalent to each other.
Since $P$ and $\tilde P$ are irredundant, we have
$r=\tilde r$, $n=\tilde n$ and $m=\tilde m$.
In the next step, we use that $P$ and $\tilde P$
are oriented to show
\[
(\skr{b}^+,\beta^+)
\ = \
(\tilde{\skr{b}}^+,\tilde{\beta}^+).
\]
If $P$ satisfies (O4) from Definition~\ref{betaplusminus},
then Remark~\ref{mplusminusswap3} directly gives
the claim.
Let $P$ satisfy one of~(O1),~(O2) or~(O3).
Then Remark~\ref{mplusminusswap3} shows that 
the transformation of $(A,P)$ into
$(\tilde A,\tilde P)$ via admissible operations
can only involve an even number of operations of
type~(ii) and thus gives the claim also
in these cases.
Now, $\beta^+=\tilde{\beta}^+$ together with
\ref{normalformdef}~(iv) and~(v) yield
\[
m_{ij}-\skr{b}^+_i
\ = \
\beta^+_{ij}
\ = \
\tilde{\beta}^+_{ij}
\ =\
\tilde m_{ij}-\tilde{\skr{b}}^+_i,
\qquad
i=0,\ldots,r,
\
j=1,\ldots, n_i.
\]
Condition~\ref{normalformdef}~(vi) implies
$\skr{b}^+_i=\tilde{\skr{b}}^+_i=0$
for $i=1,\ldots,r$.
Due to $\skr{b}^+=\tilde{\skr{b}}^+$,
we also have $\skr{b}^+_0=\tilde{\skr{b}}^+_0$.
Hence $m_{ij}=\tilde m_{ij}$ for $i=0,\ldots, r$ 
and $j=1,\ldots, n_i$.
Since $P$ and~$\tilde P$ have primitive columns,
we arrive at $l_{ij}=\tilde l_{ij}$ and
$d_{ij}=\tilde d_{ij}$.
Condition~\ref{normalformdef}~(iii) ensures that in
the case (p-p), also the last two columns of $P$ and
$\tilde P$ coincide. Thus we obtain $P=\tilde P$.
From~\ref{normalformdef}~(vii) and $\tilde A=BAD$
with $B\in\GL_2(\KK)$ and a diagonal 
$D\in\GL_{r+1}(\KK)$, we obtain $A=\tilde A$.
\end{proof}

\begin{remark}
The procedure for computing the normal form
described in the proof of the first statement
of Proposition~\ref{normalform} can be
directly implemented and gives an efficient
way to determine the normal form in practice.
\end{remark}

\begin{remark}
\label{rem:families}
Let $(A,P)$ be a pair of defining matrices
in normal form.
Then the ideal of relations of the ring
$R(A,P)$ is as well generated by
\[
\tilde g_{\iota}
\ := \
\lambda_{\iota} T_{0}^{l_{0}}
+
T_{1}^{l_{1}}
+
T_{\iota+1}^{l_{\iota+1}},
\qquad
\iota = 1, \ldots, r-1.
\]
This gives yet another trinomial presentation of
$R(A,P)$, showing that for fixed $P$,
the surfaces $X(A,P)$
come in an $(r-2)$-dimensional family.
\end{remark}

\begin{construction}
Consider defining data $(A,P)$ of a projective $\KK^*$-surface
as in Construction~\ref{constr:kstarsurf}.
\begin{enumerate}
\item
A \emph{redundant extension} of $(A,P)$ is a pair
$(A',P')$, where $A' = [A,a_{r+1}]$ and~$P'$ arises from
$P$ by inserting a zero row at $r+1$, a zero column at
$n+1$ then turning the entry at $(r+1,n+1)$ to $1$.
We will also call any defining data equivalent to
$(A',P')$ a redundant extension of $(A,P)$.
\item
A \emph{proper extension} of $(A,P)$ is a pair
$(A',P')$ of defining matrices, where $A = A'$ and
$P'$ arises from $P$ via inserting either a column
$v_{in_i+1}$ into the $i$-th arm or inserting a column
of type $v^\pm$ at the end of $P$.
\end{enumerate}
\end{construction}

\begin{remark}
Consider a projective $\KK^*$-surface
$X$ given by defining data $(A,P)$.
\begin{enumerate}
\item
Every redundant extension $(A',P')$ of $(A,P)$
defines a $\KK^*$-surface $X'$ isomorphic to $X$.
\item
For every proper extension $(A,P')$ of $(A,P)$,
the associated $\KK^*$-surface $X'$ comes with
a birational morphism $X' \to X$ contracting
a curve.
\end{enumerate}
\end{remark}


\section{Geometry of rational $\KK^*$-surfaces}
\label{section:geom-kstar-surfaces}

We show how to read off basic geometric
properties of rational $\KK^*$-surfaces
from their defining data. For a given
$X = X(A,P)$ we explicitly determine
divisor class group, Cox ring, Picard
group, cones of effective, movable,
semiample and ample divisor classes,
canonical divisor and
its intersection theory.

All the facts we discuss in this section rely on the
explicit knowledge of the Cox ring provided
in Summary~\ref{sum:classgroup-coxring}.
It allows us to apply
the whole machinery around Cox rings from~\cite{ArDeHaLa}.
Our first closer look focuses on
data in the divisor class group; compare
Summary~\ref{sum:toriccldata} for corresponding
statements in the toric case and see
\cite[Prop. 3.3.2.3 and Prop. 3.4.4.1]{ArDeHaLa}
for the proofs.

\begin{summary}
\label{sum:cl-data}
Let $X = X(A,P)$ be a $\KK^*$-surface and
$X \subseteq Z$ the associated toric embedding.
We take a look at various data in the divisor
class groups
\[
\Cl(X)
\ = \
K
\ = \
\Cl(Z).
\]
For $\sigma \in \Sigma$ define a subgroup of $K$
by $K_\sigma := \bangle{w_{ij}, w^\pm; \
v_{ij},v^\pm \not\in \sigma}$. Then the
\emph{local class group} of
$x \in X \cap \TT^n \cdot z_\sigma$ is given by
\[
\Cl(X,x) \ = \ K/K_\sigma \ = \ \Cl(Z,z_\sigma).
\]
As a consequence, $X$ also shares its
\emph{Picard group} with the ambient toric
variety $Z$. More precisely, in $\Cl(X) = \Cl(Z)$,
we obtain
\[
\Pic(X)
\ = \ 
\bigcap_{\sigma \in \Sigma} K_\sigma
\ = \
\Pic(Z).
\]
The \emph{monoid of effective divisor classes}
is generated by the classes $w_{ij}$ of
$D_X^{ij}$ and $w^\pm$ of $D_X^\pm$. Thus, in
$K_\QQ = \QQ \otimes_\ZZ K$, the effective cone
of $X$ is
\[
\Eff(X)
\ = \
\cone(w_{ij}, w^\pm)
\ = \
\Eff(Z).
\]
For $\sigma \in \Sigma$ set 
$\sigma^* := \cone(w_{ij}, w^\pm; \
v_{ij},v^\pm \not\in \sigma)$. Then, for the
\emph{cones of movable and semiample divisor
classes} of $X$, we have
\[
\Mov(Z)
 =  
\bigcap_{{\gamma_0 \preccurlyeq \gamma}\atop{\text{facet}}} Q(\gamma_0)
 = 
\Mov(X)
 = 
\SAmple(X)
 = 
\bigcap_{\sigma \in \Sigma} \sigma^*
 = 
\SAmple(Z).
\]
Furthermore, the \emph{ample cones} are given by
\[
\Ample(X) \ = \ \SAmple(X)^\circ \ = \ \SAmple(Z)^\circ \ = \ \Ample(Z).
\]
For the common degree $\mu \in \Cl(X)$ of the
defining polynomials $g_0,\dots, g_{r-2}$ of
$X \subseteq Z$, we have
\[
\mu
\ = \
\sum_{j = 0}^{n_i} l_{ij} w_{ij}
\ \in \
\bigcap_{i=0}^r \cone(w_{i1}, \dots, w_{in_i})
\ \subseteq \
\SAmple(X).
\]
Finally, for each $i=0,\dots,r$ we obtain
an anticanonical divisor of $X$ by the
following adjunction formula.
\[
- \mathcal{K}_X^i
\ = \
\sum D_X^{ij} + \sum D_X^k
-
(r-1) \sum_{j = 1}^{n_i} l_{ij} D_X^{ij}.
\]
In particular, $X$ is a del Pezzo surface if and
only if the corresponding divisor class
$-w_X=-w_Z-(r-1)\mu$ lies in the ample cone,
which can be checked explicitly.
\end{summary}

\begin{proposition}
\label{prop:kstarsurfisQfact}
Every $\KK^*$-surface $X=X(A,P)$ is
$\QQ$-factorial and has Picard number
$\rho(X) = n+m-r-1$.
\end{proposition}

\begin{proof}
We have to show that every Weil divisor has
a non-zero Cartier multiple.
For this consider the associated toric embedding
$X \subseteq Z$. The defining fan of~$Z$ is
simplicial and hence $Z$ is $\QQ$-factorial by
Summary \ref{sum:tv-sing}. Thus $\Pic(Z)$ is of
finite index in $\Cl(Z)$. Summary
\ref{sum:cl-data} shows that $\Pic(X)$ is of
finite index in $\Cl(X)$. Consequently~$X$ is
$\QQ$-factorial.
The formula of the Picard number then follows
from $\rho(X) = \mathrm{rk}(\Cl(X)) = n+m - (r+1)$.
\end{proof}

Let us emphasize that the descriptions
of the Picard group and the cones of (semi-)ample
divisor classes essentially depend on the fact
that we work with the toric embedding
$X \subseteq Z$ provided by Construction
\ref{constr:kstarsurf}. In contrast, the
descriptions of the cone of movable divisor
classes and the anticanonical divisors are more
robust and allow going over to certain
completions of $Z$ as presented below.

\begin{summary}
\label{sum:completions}
Consider a projective $\KK^*$-surface
$X = X(A,P)$ and the toric embedding
$X \subseteq Z$ as provided by Construction
\ref{constr:kstarsurf}. In general, $Z$ is not
complete and we have several choices of possible
toric completions $Z \subseteq Z'$. For instance,
every divisor class $w \in \Mov(Z)^\circ$ yields
a fan
\[
\Sigma(w)
\ := \
\{P(\gamma_0^*); \ \gamma_0 \preccurlyeq \gamma, \ w \in Q(\gamma_0)^\circ\}.
\]
Any such fan $\Sigma(w)$ is polytopal, has the
same generator matrix as $\Sigma$ and contains
all cones of $\Sigma$. The associated open
embeddings $Z \subseteq Z(w)$ are precisely the
projective toric completions such that $Z(w)$ has
the same Cox ring as $Z$. For a precise picture,
one associates with $w \in \Eff(Z) = Q(\gamma)$
the polyhedral cone
\[
\lambda(w)
\ := \
\bigcap_{\genfrac{}{}{0pt}{}{\gamma_0 \preccurlyeq \gamma,}{w \in Q(\gamma_0)}}
Q(\gamma_0).
\]
These $\lambda(w)$ form a fan supported on
$Q(\gamma) = \Eff(Z)$, the so-called
\emph{secondary fan}. For
$w,w' \in \Mov(Z)^\circ$ we have
$\lambda(w) \preccurlyeq \lambda(w')$
if and only if $\Sigma(w')$ refines $\Sigma(w)$.
In particular, all $w' \in \lambda(w)^\circ$
share the same $\Sigma(w)$. A fan $\Sigma(w)$
is simplicial if and only if $\lambda(w)$ is
full-dimensional.
\end{summary}

We come to the intersection theory of
rational $\KK^*$-surfaces $X = X(A,P)$. As just
noted, $X$ is $\QQ$-factorial and thus has indeed
a well-defined intersection product. The aim is
to express intersection numbers in terms of the
defining matrix $P$. As a preparation, we assign
the following numbers to $P$, which in fact turn
out to be ubiquitous in all of the subsequent
considerations.

\begin{definition}
\label{def:mandzeta}
With any slope-ordered defining matrix $P$ in the
sense of Construction \ref{constr:kstarsurf}, we
associate the numbers
\[
l^+ := l_{01} \cdots l_{r1},
\quad
\skr{m}^+ := m_{01} + \dots + m_{r1},
\quad
\skr{l}^+ := \frac{1}{l_{01}} + \dots + \frac{1}{l_{r1}} -r+1,
$$
$$
l^- := l_{0n_0} \cdots l_{rn_r},
\quad
\skr{m}^- := m_{0n_0} + \dots + m_{rn_r},
\quad
\skr{l}^- := \frac{1}{l_{0n_0}} + \dots + \frac{1}{l_{rn_r}} -r+1.
\]
\end{definition}

\begin{remark}
\label{rem:obviouszetaprops}
Let $P$ be a defining matrix as in Construction
\ref{constr:kstarsurf}. Then we always have
$\skr{l}^+ \le 2$ and $\skr{l}^- \le 2$. Moreover,
\begin{align*}
&\skr{m}^+ \ = \ \frac{1}{l^+} \det(\sigma^+),
\qquad
\det(\sigma^+) \ := \ (-1)^r\det(v_{01}, \dots, v_{r1}),
\\
&\skr{m}^- \ = \ \frac{1}{l^-} \det(\sigma^-),
\qquad
\det(\sigma^-) \ := \ (-1)^r\det(v_{0n_0}, \dots, v_{rn_r}).
\end{align*}
If $X = X(A,P)$ has an elliptic fixed point
$x^+ \in X$, then $\skr{m}^+ > 0$ and if there is
$x^- \in X$, then $\skr{m}^- < 0$.
\end{remark}

We first explain the basic principles
of intersecting the relevant invariant curves
on $X = X(A,P)$. 

\begin{summary}
Consider a $\KK^*$-surface $X=X(A,P)$ with
slope-ordered $P$. The fact that $X$ comes as a
complete intersection in a $\QQ$-factorial
projective ambient toric variety $Z$ allows to
perform intersections of curves via intersecting
suitable ambient toric divisors. For
$i=1,\dots,r$, consider the divisors 
\[
D_Z^i
\ := \
\sum_{j=1}^{n_i} l_{ij} D_Z^{ij}
\ \in \
\WDiv(Z).
\]
The divisor class of $D_Z^i$ equals the common
degree the defining relations of $X \subseteq Z$.
Thus, by general intersection theory, the
intersection number of any two of the $D_X^{ij}$
and $D_X^\pm$ equals the intersection of the
corresponding toric prime divisors with $r-1$ of
the $D_Z^i$. For instance, for $D_X^{01}$ and
$D_X^{11}$ we can write
\[
D_X^{01} \cdot D_X^{11}
\ = \
D_Z^{01} \cdot D_Z^{11} \cdot D_Z^{2} \cdots D_Z^{r}.
\]
The toric intersection number expands into
intersection numbers of pairwise distinct toric
prime divisors. These, if non-zero, are given by
toric intersection theory as one over the
determinant of the involved primitive generators.
If $D_X^{01}$ and $D_X^{11}$ intersect precisely
in $x^+ \in X$, we obtain
\[
D_X^{01} \cdot D_X^{11}
\ = \
D_Z^{01} \cdot D_Z^{11} \cdot D_Z^{21} \cdots D_Z^{r1}
\ = \
\frac{l_{21} \cdots l_{r1}}{\det(\sigma^+)}
\ = \
\frac{1}{l_{01}l_{11}} \frac{1}{\skr{m}^+}.
\]
Note that any two of the $D_Z^{i}$ are linearly
equivalent to each other and $D_Z^\pm$ is linearly
equivalent to $\mp \sum E_Z^i$ with
$E_Z^i := \sum d_{ij}D_Z^{ij}$. This enables us
to represent any intersection between $D_X^{ij}$
and $D_X^\pm$ as a linear combination of
intersection numbers of pairwise distinct toric
prime divisors and hence to proceed as above.
\end{summary}

Now we list the intersection numbers of the invariant
curves of a given projective $\KK^*$-surface $X = X(A,P)$;
compare Summary~\ref{sum:torsurfintnos}
for the case of a toric surface and see also~\cite[Sec.~5.4.2]{ArDeHaLa}.

\begin{summary}
\label{sum:intersections}
Let $X=X(A,P)$ be a projective $\KK^*$-surface
with $P$ slope-ordered. For $i=0,\dots,r$, set
\[
\begin{array}{lcl}
\skr{m}_{i0}
& := &
\begin{cases}
0, & \text{if there is } D_X^+ \subseteq X,
\\
-\frac{1}{\skr{m}^+}, & \text{if there is } x^+ \in X,
\end{cases}
\\[16pt]
\skr{m}_{ij}
& := &
\frac{1}{m_{ij}-m_{ij+1}},
\quad
j=1, \dots, n_i-1,
\\[7pt]
\skr{m}_{in_i}
& := &
\begin{cases}
0, & \text{if there is } D_X^- \subseteq X,
\\
\frac{1}{\skr{m}^-}, & \text{if there is } x^- \in X.
\end{cases}
\end{array}                
\]
Then the non-zero intersection numbers of
(possible) curves $D_X^\pm$ and the curves
$D_X^{ij}$ in $X$ are given by
\[
D_X^+ \cdot D_X^{i1} \ = \ \frac{1}{l_{i1}},
\qquad
D_X^{ij} \cdot D_X^{ij+1} \ = \  \frac{1}{l_{ij}l_{ij+1}} \skr{m}_{ij},
\qquad
D_X^{in_i} \cdot D_X^- \ = \ \frac{1}{l_{in_i}}.
\]
Moreover, in case there is an elliptic fixed
point $x^+ \in X$ or $x^- \in X$, for $i \ne k$
we have the following intersections 
\[
\begin{array}{cclcl}
\displaystyle
D_X^{i1} \cdot D_X^{k1}
& = & 
\displaystyle
- \frac{1}{l_{i1}l_{k1}} \bigl(\skr{m}_{i0}+\skr{m}_{in_i}\bigr),
& & 
\displaystyle
n_in_k = 1,
\\[12pt]
\displaystyle
D_X^{i1} \cdot D_X^{k1}
& = & 
\displaystyle
- \frac{1}{l_{i1}l_{k1}} \skr{m}_{i0},
& & 
n_in_k > 1,
\displaystyle
\\[12pt]
\displaystyle
D_X^{in_i} \cdot D_X^{kn_k}
& = & 
\displaystyle
- \frac{1}{l_{in_i}l_{kn_k}} \skr{m}_{in_i},
& & 
\displaystyle
n_in_k > 1.
\end{array}
\]
Finally, the self intersection numbers of the
(possible) curves $D_X^\pm$ and the curves~$D_X^{ij}$
in $X$ are given by
\[
D_X^+ \cdot D_X^+ \ = \ - \skr{m}^+,
\qquad
D_X^- \cdot D_X^- \ = \ \skr{m}^-,
\qquad
D_X^{ij} \cdot D_X^{ij} \ = \ -\frac{1}{l_{ij}^2}\bigl(\skr{m}_{ij-1} + \skr{m}_{ij}\bigr).
\]

\end{summary}

As the $D_X^{ij}$ and $D_X^\pm$
generate the divisor class group of $X = X(A,P)$,
the above computations determine the entire
intersection theory of $X$. In particular, we 
can directly compute intersections of the
anticanonical divisor with the relevant curves
and, subsequently, the anticanonical self
intersection.

\begin{proposition}
\label{prop:ak-intersections}  
Consider a projective $\KK^*$-surface $X=X(A,P)$
with slope-ordered $P$. For $i=0,\dots,r$, set
\[
\begin{array}{lcl}
\skr{l}_{i0}
& := &
\begin{cases}
\infty, & \text{if there is } D_X^+ \subseteq X,
\\
-\skr{l}^+, & \text{if there is } x^+ \in X,
\end{cases}
\\[16pt]
\skr{l}_{ij}
& := &
\frac{1}{l_{ij}}-\frac{1}{l_{ij+1}},
\quad
j=1, \dots, n_i-1,
\\[7pt]
\skr{l}_{in_i}
& := &
\begin{cases}
-\infty, & \text{if there is } D_X^- \subseteq X,
\\
\skr{l}^-, & \text{if there is } x^- \in X.
\end{cases}
\end{array}                
\]
Setting $\infty \cdot 0 = 1$ and
$-\infty \cdot 0 = -1$, we can write the
intersections of an anticanonical divisor
$-\mathcal{K}_X$ with the $D_X^{ij}$ and
$D_X^\pm$ as
\[
\begin{array}{lcl}
- \mathcal{K}_X \cdot D_X^{+}
& = &
\skr{l}^+ - \skr{m}^+,
\\[7pt]
- \mathcal{K}_X \cdot D_X^{ij}
& = &
\frac{1}{l_{ij}}\left(\skr{l}_{ij-1}\skr{m}_{ij-1} -\skr{l}_{ij}\skr{m}_{ij}  \right),
\\[7pt]
- \mathcal{K}_X \cdot D_X^{-}
& = &
\skr{l}^- + \skr{m}^-.
\end{array}      
\]

\end{proposition}

\begin{proof}
This is an explicit computation. Take the
anticanonical divisor $\mathcal{K}_X^0$ from
Summary \ref{sum:cl-data} and then apply Summary
\ref{sum:intersections}.
\end{proof}

\begin{proposition}
\label{prop:kstaraccselfie}
Consider a projective $\KK^*$-surface $X=X(A,P)$
with slope-ordered $P$. To any arm
$\mathscr{A}_i$, where $i=0,\dots,r$, we attach
the number
\[
\alpha_i
\ := \ 
\sum_{j=1}^{n_i-1}
\frac{\lambda_{ij}}{\Delta_{ij}},
\qquad
\Delta_{ij}
\ := \
l_{ij+1}d_{ij} - l_{ij}d_{ij+1},
\qquad
\lambda_{ij}
\ := \
2 - \frac{l_{ij+1}}{l_{ij}} - \frac{l_{ij}}{l_{ij+1}},
\]
where $\alpha_i = 0$ if $n_i = 1$.
Then, according to the possible constellations of
source and sink, the anticanonical self
intersection number of $X$ is given by 
\[
\begin{array}{llcl}
\text{\rm (e-e)}\quad
&  
\displaystyle
\mathcal{K}_X^2
& = &
\displaystyle
\frac{(\skr{l}^+)^2}{\skr{m}^+} + \alpha_0 + \dots + \alpha_r - \frac{(\skr{l}^-)^2}{\skr{m}^-},\\[9pt]
\text{\rm (e-p)}\quad
&  
\displaystyle
\mathcal{K}_X^2
& = &
\displaystyle
\frac{(\skr{l}^+)^2}{\skr{m}^+} + \alpha_0 + \dots + \alpha_r + 2\skr{l}^- +\skr{m}^-,
\\[9pt]
\text{\rm (p-e)}\quad
&  
\displaystyle
\mathcal{K}_X^2
& = &
\displaystyle
2\skr{l}^+-\skr{m}^+ + \alpha_0 + \dots + \alpha_r - \frac{(\skr{l}^-)^2}{\skr{m}^-},
\\[9pt]
\text{\rm (p-p)}\quad
&  
\displaystyle
\mathcal{K}_X^2
& = &
\displaystyle
2\skr{l}^+-\skr{m}^+ + \alpha_0 + \dots + \alpha_r  + 2\skr{l}^- +\skr{m}^-.
\end{array}
\]
\end{proposition}

\begin{proof}
Take again the anticanonical divisor
$\mathcal{K}_X^0$ from Summary \ref{sum:cl-data}
and use Proposition \ref{prop:ak-intersections}.
\end{proof}

\begin{remark}
In the case $r=1$ with two elliptic fixed points,
Proposition \ref{prop:kstaraccselfie} exactly
reproduces Remark \ref{rem:toricaccselfie}. Note
that on every complete toric surface, we find a
$\KK^*$-action with two elliptic fixed points.
\end{remark}


\section{Singularities}

We take a close look at the singularities
of our $\KK^*$-surfaces $X = X(A,P)$.
After recalling basics and the canonical
resolution, we describe in
Propositions~\ref{prop:kstar-gorind-1}
and~\ref{prop:kstar-gorind-2} 
the local Gorenstein indices.
Moreover, in Proposition~\ref{prop:logtermchar}
and Corollary~\ref{cor:logtermchar}, we
characterize log terminality.
This leads to an explicit description of
all rational log terminal surface singularities,
provided by Propositions~\ref{prop:affineiota}
and~\ref{prop:affine-non-toric-ell}.

We begin with locating smooth points and (mild)
singularities of $X = X(A,P)$.
A point $x \in X$ is called \emph{quasismooth}
if $x = p(\hat x)$ holds with a smooth point
$\hat x \in \hat X$. Moreover, a point
$x \in X$ is called \emph{factorial} if its local
divisor class group is trivial. The latter holds
if and only if the local ring $\mathcal{O}_{X,x}$
admits unique factorization.

\begin{summary}
\label{sum:singul-basics}
Consider a $\KK^*$-surface $X = X(A,P) \subseteq Z$
as provided by Construction \ref{constr:kstarsurf},
where $P$ is slope-ordererd. First note the following.
\begin{itemize}
\item
As $X$ is normal, its singularities are
necessarily fixed points.
\item
The surface $X$ inherits $\QQ$-factoriality from
$Z$; see Proposition \ref{prop:kstarsurfisQfact}.
\end{itemize}
A parabolic fixed point is always quasismooth
and it is even smooth if it is not contained in
an arm of $X$.
Moreover, as in~\cite[Prop. 3.4.4.6]{ArDeHaLa},
we see that each
\[
x_i^+
\ \in \
D_X^+ \cap \mathscr{A}_i,
\qquad\qquad
x_i^-
\ \in \
D_X^- \cap \mathscr{A}_i
\]
admits a neighborhood
$x_i^\pm \in U_i^\pm \subseteq X$ isomorphic
to a neighborhood of the fixed point
of the toric surface with generator matrix
\[
P_i^+
\ = \
\left[
\begin{array}{cc}
0 & l_{i1}
\\
1 & d_{i1} 
\end{array}
\right],
\qquad\qquad
P_i^-
\ = \
\left[
\begin{array}{cc}
l_{in_i} & 0
\\
d_{in_i} & -1  
\end{array}
\right].
\]
In particular, $x_i^+$ ($x_i^-$) is smooth if and
only if $l_{i1} = 1$ ($l_{in_i} = 1$).
Next we discuss the hyperbolic fixed points.
These are precisely the intersection points
\[
x_{ij} \ \in \ D_X^{ij} \cap D_X^{ij+1}.
\]
Each $x_{ij}$ is quasismooth and allows a
neighborhood $x_{ij} \in U_{ij} \subseteq X$
isomorphic to a neighborhood of the
fixed point of the toric surface with generator
matrix
\[
P_{ij}
\ = \
\left[
\begin{array}{cc}
l_{ij} & l_{ij+1}
\\
d_{ij} & d_{ij+1} 
\end{array}
\right].
\]
In particular $x_{ij}$ is smooth if and only if
$\det(P_{ij})=-1$. For the possible elliptic
fixed points, we have the affine
$\KK^*$-invariant open neighborhoods
\[
x^+
\in
U^+
=
\{x \in X; \ x_0 = x^+\}
\subseteq
X,
\quad
x^-
\in
U^-
=
\{x \in X; \ x_\infty = x^-\}
\subseteq
X.
\]
As $\KK^*$-surfaces they are given by
$U^+ = X(A,P^+)$ and $U^- = X(A,P^-)$ with the
defining matrices 
\[
P^+
\ = \
[v_{01}, \dots, v_{r1}],
\qquad\qquad
P^-
\ = \
[v_{0n_0}, \dots, v_{rn_r}].
\]
The local class group of an elliptic fixed point
$x^\pm$ equals the divisor class group of~$U^\pm$.
Moreover, $x^+$ ($x^-$) is
\begin{itemize}
\item 
quasismooth if and only if $l_{i1} = 1$
($l_{in_i} = 1$) for at least $r-1$ of
$i = 0, \dots, r$,
\item 
factorial if and only if
$\det(P^+) = l^+\skr{m}^+ = 1$
($\det(P^-) = l^-\skr{m}^- = -1$),   
\item
smooth if and only if it is factorial and
quasismooth.
\end{itemize}
Note that $x^\pm \in X$ is factorial if and only
if the Cox ring of the affine $\KK^*$-surface~$U^\pm$
coincides with its total coordinate ring.
\end{summary}

We present the resolution of
singularities of a $\KK^*$-surface $X=X(A,P)$ in
terms of the defining matrix $P$. Note that the
procedure yields in particular the star-shaped
resolution graphs observed in
\cite{OrWa1, OrWa2, Pi}.

\begin{summary}
\label{sum:surfsingres}
See \cite[Constr. 5.4.3.2]{ArDeHaLa}. The
\emph{canonical resolution} of singularities
$X'' \to X$ of $X = X(A,P)$ is obtained by the
following two-step procedure.
\begin{enumerate}
\item
Enlarge $P$ to a matrix $P'$ by adding the
columns $v^+$ and $v^-$ if not already present. 
Then we have the surface $X' := X(A,P')$ and a
canonical morphism $X' \to X$.
\item 
Let $P''$ be the defining matrix having the
primitive generators of the regular subdivision
$\Sigma''$ of $\Sigma'$ as its columns. Then
$X'' := X(A,P'')$ is smooth and there is a
canonical morphism $X'' \to X'$.
\end{enumerate}
Both fans $\Sigma'$ and $\Sigma''$ have the
\emph{tropical variety} of $X \subseteq Z$
as their support. With the canonical basis
vectors $e_1,\dots,e_{r+1} \in \ZZ^{r+1}$ and
$e_0 := -e_1 - \dots - e_r$, it is given by
\[
\trop(X)
\ = \
\lambda_0 \cup \dots \cup \lambda_r
\ \subseteq \
\QQ^{r+1},
\qquad
\lambda_i \ := \ \cone(e_i,\pm e_{r+1}).
\]
Contracting all $(-1)$-curves inside the smooth
locus that lie over singularities of $X$ gives
$X'' \to \tilde X \to X$, where
$X(A,\tilde P) = \tilde X \to X$ is the
\emph{minimal resolution} of $X$.
\end{summary}

\begin{example}
\label{ex:runningexampleA3}
Consider again the $\KK^*$-surface $X = X(A,P)$
from Example \ref{ex:runningexampleA1}, given by
the defining data
\[
A
\ = \
\left[
\begin{array}{rrr}
1 & 0 & -1
\\
0 & 1 & -1
\end{array}    
\right],
\qquad\qquad
P
\ = \
\left[
\begin{array}{ccccc}
-3 & -5 & 2 & 0 & 0 
\\
-3 & -5 & 0 & 2 & 0 
\\
-4 & -8 & 1 & 1 & 1 
\end{array}    
\right].
\]
The four hyperbolic fixed points and the elliptic
fixed point are singular. The two resolution
steps from Summary \ref{sum:surfsingres}
schematically look as follows.
\def\fanrunningex{
\begin{tikzpicture}[scale=0.6]
\sffamily

\coordinate(cplus) at (0,2);
\coordinate(c01) at (-2,-0);
\coordinate(c02) at (-2,-1.5);
\coordinate(c11) at (1,-1.8);
\coordinate(c21) at (2,-.75);

\node[centered] at (0,-3) {$\Sigma$};

\path[fill, color=gray!40] (ooo) -- (c11) -- (c21) -- (ooo);
\path[fill, color=gray!30] (ooo) -- (cplus) -- (c21) -- (ooo);
\path[fill, color=gray!10] (ooo) -- (cplus) -- (c01) -- (c02) -- (ooo);
\path[fill, color=gray!30] (ooo) -- (c02) -- (c11) -- (ooo);

\draw[thick, color=black] (ooo) to (c01);
\draw[thick, color=black] (ooo) to (c02);
\draw[thick, color=black] (ooo) to (c11);
\draw[thick, color=black] (ooo) to (c21);
\draw[thick, color=black] (ooo) to (cplus);

\draw[thick, dotted, color=black] (cplus) to (c01);
\draw[thick, dotted, color=black] (cplus) to (c21);
\draw[thick, dotted, color=black] (c01) to (c02);
\draw[thick, dotted, color=black] (c02) to (c11);
\draw[thick, dotted, color=black] (c11) to (c21);

\path[fill, opacity=.9, color=gray!20] (ooo) -- (cplus) -- (c11) -- (ooo);
\draw[thick, dotted, color=black] (cplus) to (c11);
\end{tikzpicture}
}

\def\weakttropresrunningex{
\begin{tikzpicture}[scale=0.6]
\sffamily

\coordinate(cplus) at (0,2);
\coordinate(c01) at (-2,-0);
\coordinate(c02) at (-2,-1.5);
\coordinate(c11) at (1,-1.8);
\coordinate(c21) at (2,-.75);
\coordinate(cminus) at (0,-2);

\node[centered] at (0,-3) {$\Sigma'$};

\path[fill, color=gray!30] (ooo) -- (cplus) -- (c21) -- (cminus) -- (ooo);
\path[fill, color=gray!10] (ooo) -- (cplus) -- (c01) -- (c02) -- (cminus) -- (ooo);

\draw[thick, color=black] (ooo) to (c01);
\draw[thick, color=black] (ooo) to (c02);
\draw[thick, color=black] (ooo) to (c21);
\draw[thick, color=black] (ooo) to (cplus);
\draw[thick, color=black] (ooo) to (cminus);

\draw[thick, dotted, color=black] (cplus) to (c01);
\draw[thick, dotted, color=black] (cplus) to (c21);
\draw[thick, dotted, color=black] (c21) to (cminus);
\draw[thick, dotted, color=black] (c01) to (c02);

\path[fill, opacity=.9, color=gray!20] (ooo) -- (cplus) -- (c11) -- (cminus) -- (ooo);
\draw[thick, color=black] (ooo) to (c11);
\draw[thick, dotted, color=black] (cplus) to (c11);
\draw[thick, dotted, color=black] (c02) to (cminus);
\draw[thick, dotted, color=black] (c11) to (cminus);
\end{tikzpicture}
}

\def\resolutionrunningex{
\begin{tikzpicture}[scale=0.6]
\sffamily

\coordinate(cplus) at (0,2);
\coordinate(cplusa) at (-1,1.5);
\coordinate(cplusb) at (1.3,0.3);
\coordinate(cplusc) at (1.5,1);
\coordinate(c01) at (-2,0);
\coordinate(c01a) at (-2.15,-1);
\coordinate(c02) at (-2,-1.5);
\coordinate(c02a) at (-1.2,-1.9);
\coordinate(c02b) at (-0.5,-2);
\coordinate(c11) at (1,-1.8);
\coordinate(c11a) at (0.6,-1.9);
\coordinate(c21) at (2,-.75);
\coordinate(c21a) at (1.7,-1.4);
\coordinate(cminus) at (0,-2);

\node[centered] at (0,-3) {$\Sigma''$};

\path[fill, color=gray!30] (ooo) -- (cplus) -- (cplusc) -- (c21) -- (c21a) -- (cminus) -- (ooo);
\path[fill, color=gray!10] (ooo) -- (cplus) -- (cplusa) -- (c01) -- (c01a) -- (c02) -- (c02a) -- (c02b) -- (cminus) -- (ooo);

\draw[thick, color=black] (ooo) to (c01);
\draw[thick, color=black] (ooo) to (c01a);
\draw[thick, color=black] (ooo) to (c02);
\draw[thick, color=black] (ooo) to (c02a);
\draw[thick, color=black] (ooo) to (c02b);
\draw[thick, color=black] (ooo) to (c21);
\draw[thick, color=black] (ooo) to (c21a);
\draw[thick, color=black] (ooo) to (cplusc);
\draw[thick, color=black] (ooo) to (cplus);
\draw[thick, color=black] (ooo) to (cplusa);
\draw[thick, color=black] (ooo) to (cminus);

\draw[thick, dotted, color=black] (cplus) to (cplusc) to (c21) to (c21a) to (cminus);
\draw[thick, dotted, color=black] (cplus) to (cplusa) to (c01) to (c01a) to (c02) to (c02a) to (c02b) to (cminus);

\path[fill, opacity=.9, color=gray!20] (ooo) -- (cplus) -- (cplusb) -- (c11) -- (c11a) -- (cminus) -- (ooo);
\draw[thick, color=black] (ooo) to (cplusb);
\draw[thick, color=black] (ooo) to (c11);
\draw[thick, color=black] (ooo) to (c11a);
\draw[thick, dotted, color=black] (cplus) to (cplusb) to (c11) to (c11a) to (cminus);
\end{tikzpicture}
}

\begin{center}

\fanrunningex
\qquad\qquad
\weakttropresrunningex
\qquad\qquad
\resolutionrunningex
\end{center}

\noindent Explicitly, the defining matrix
$\tilde P$ of the minimal resolution
$\tilde X = X(A, \tilde P)$ is given by
\[
\tilde P
\ = \ 
\left[
\begin{array}{rrrrrrrrrrrrrr}
-1 & -3 & -2 & -5 & -3 & -1 & 1 & 2 & 1 & 0 & 0 & 0 & 0 & 0\\
-1 & -3 & -2 & -5 & -3 & -1 & 0 & 0 & 0 & 1 & 2 & 1 & 0 & 0\\
-1 & -4 & -3 & -8 & -5 & -2 & 1 & 1 & 0 & 1 & 1 & 0 & 1 & -1
\end{array}
\right].
\]
\end{example}

\begin{remark}
Consider $X = X(A,P)$ and the associated surface
$X'$ from Step (i) of Summary
\ref{sum:surfsingres}. Then we have
\begin{alignat*}{2}
&\skr{m}^+
\ = \
-D_{X'}^+ \cdot D_{X'}^+,
\qquad\qquad
&&\skr{l}^+
\ = \ 
D_{X'}^+\cdot D_{X'}^+-\mathcal{K}_{X'}^0\cdot D_{X'}^+,
\\
&\skr{m}^-
\ = \
D_{X'}^- \cdot D_{X'}^-,
\qquad\qquad
&&\skr{l}^-
\ = \ 
D_{X'}^{-}\cdot D_{X'}^{-}-\mathcal{K}_{X'}^0\cdot D_{X'}^{-}.
\end{alignat*}
\end{remark}

\begin{construction}
\label{rem:weaktropgeom}
Consider a rational projective $\KK^*$-surface
$X$, the quotient
$\pi \colon X \dasharrow \PP_1$, its domain of
definition $U \subseteq X$ and the closure of the
graph
\[
X' \ := \ \overline{\Gamma}_\pi \ \subseteq \ X \times \PP_1,
\qquad
\Gamma_\pi \ = \ \{(x,\pi(x)); \ x \in U\} \ \subseteq \ X \times \PP_1.
\]
Then $X'$ comes with a $\KK^*$-action given by
$t \ast (x,z) = (t \cdot x, z)$ and the
projection yields an equivariant birational
morphism
\[
X' \ \to \ X,
\qquad
(x,\pi(x)) \ \mapsto \ x.
\]
Furthermore, for every equivariant morphism
$\varphi \colon X_1 \to X_2$ of $\KK^*$-surfaces
with rational quotients
$\pi_i \colon X_i \dasharrow \PP_1$, we have an
induced equivariant morphism
\[
\varphi' \colon X_1' \ \to \ X_2',
\qquad
(x,\pi_1(x)) \ \mapsto \ (\varphi(x),\pi_2(x)).
\]
\end{construction}

\begin{proposition}
For any $X = X(A,P)$, the morphism
$X' \to X$ from Construction
\ref{rem:weaktropgeom} equals the one presented
in Summary \ref{sum:surfsingres} (i).
\end{proposition}

\begin{corollary}
\label{cor:mandzeta}
Let the $\KK^*$-surfaces $X_1$ and $X_2$
arise from defining data $(A_1, P_1)$ and
$(A_2, P_2)$. If there is an equivariant
isomorphism $X_1 \cong X_2$, then we have
\[
\skr{m}_1^+ = \skr{m}_2^+,
\qquad
\skr{l}_1^+ = \skr{l}_2^+,
\qquad
\skr{m}_1^- = \skr{m}_2^-,
\qquad
\skr{l}_1^- = \skr{l}_2^-.
\]  
In particular, for every rational projective
$\KK^*$-surface $X$, we can choose any
isomorphism $X \cong X(A,P)$ and obtain well
defined numbers
\[
\skr{m}_X^+ := \skr{m}^+,
\qquad
\skr{l}_X^+ := \skr{l}^+,
\qquad
\skr{m}_X^- := \skr{m}^-,
\qquad
\skr{l}_X^- := \skr{l}^-.
\]
\end{corollary}

We discuss Gorenstein indices.
Recall that a normal variety $X$ is
\emph{$\QQ$-Gorenstein} if some non-zero multiple
of a canonical divisor $\mathcal{K}_X$ is
Cartier.
In this case, the
\emph{local Gorenstein index} of a point
$x \in X$ is the smallest non-zero integer
$\iota_x$ such that $\iota_x \mathcal{K}_X$
is principal on some neighborhood of $x$.

\begin{proposition}
\label{prop:kstar-gorind-1}
Consider a $\KK^*$-surface $X=X(A,P)$ with
slope-ordered $P$ and the (possible) parabolic
and hyperbolic fixed points of $X$.
\begin{enumerate}
\item
The local Gorenstein index of a parabolic fixed
point $x_i^+ \in D_X^+ \cap \mathscr{A}_i$ is
given by
\[
\iota_{i}^{+} \ := \ \iota(x_i^+)
\ = \
\frac{l_{i1}}{\gcd(d_{i1}-1, l_{i1})}.
\]
\item
The local Gorenstein index of a parabolic fixed
point $x_i^- \in D_X^- \cap \mathscr{A}_i$ is
given by 
\[
\iota_{i}^{-} \ := \ \iota(x_i^-)
\ = \
\frac{l_{in_i}}{\gcd(d_{in_i}+1, l_{in_i})}.
\]
\item
The local Gorenstein index of a hyperbolic fixed
point $x_{ij} \in D_X^{ij} \cap D_X^{ij+1}$ is
given by
\[
\iota_{ij} \ := \ \iota(x_{ij})
\ = \
\frac{l_{ij+1}d_{ij} - l_{ij}d_{ij+1}}{\gcd(l_{ij}-l_{ij+1}, d_{ij}-d_{ij+1})}.
\]
\end{enumerate}
\end{proposition}

\begin{proof}
By Summary~\ref{sum:singul-basics}, 
a neighborhood of $x_i^+$, $x_i^-$, $x_{ij}$ is
isomorphic to a neighborhood of the toric fixed
point of the affine toric surface with generator
matrix
\[
P_i^+
\ = \
\left[
\begin{array}{cc}
0 & l_{i1}
\\
1 & d_{i1} 
\end{array}
\right],
\qquad
P_i^-
\ = \
\left[
\begin{array}{cc}
l_{in_i} & 0
\\
d_{in_i} & -1  
\end{array}
\right],
\qquad
P_{ij}
\ = \
\left[
\begin{array}{cc}
l_{ij} & l_{ij+1}
\\
d_{ij} & d_{ij+1} 
\end{array}
\right],
\]
respectively. Consequently, Remark \ref{rem:toric-surf-gorind}
provides the desired formulae for the local Gorenstein indices.
\end{proof}

\begin{proposition}
\label{prop:kstar-gorind-2}
Consider a $\KK^*$-surface $X = X(A,P)$ with
slope-ordered $P$. We define linear forms
$u^\pm \in \QQ^{r+1}$ by
\[
u^+
\ := \
\frac{1}{l^+\skr{m}^+}
\left(u_1^+, \dots, u_r^+,l^+ \skr{l}^+ \right),
\qquad
u^-
\ := \
\frac{1}{l^-\skr{m}^-}
\left(u_1^-, \dots, u_r^-,l^- \skr{l}^- \right),
\]
where $u_i^\pm \in \ZZ$ is given by
\begin{alignat*}{2}
&u_i^+
\ = \
(r-1)d_{i1}\frac{l^+}{l_{i1}}
+
\sum_{j \ne i} (d_{j1}-d_{i1})\frac{l^+}{l_{j1}l_{i1}},
\qquad
&&i = 1, \dots, r,
\\
&u_i^-
\ = \
(r-1)d_{in_i}\frac{l^-}{l_{in_i}}
+
\sum_{j \ne i} (d_{jn_j}-d_{in_i})\frac{l^-}{l_{jn_j}l_{in_i}} ,
\qquad
&&i = 1, \dots, r.
\end{alignat*}
Then the linear forms $u^\pm \in \QQ^{r+1}$ are
uniquely determined by satisfying the properties
\begin{alignat*}{2}
&\bangle{u^+,v_{01}} = 1-(r-1)l_{01},
\qquad
&&\bangle{u^+,v_{i1}} = 1, \ i = 1, \dots, r,
\\
&\bangle{u^-,v_{0n_0}} =  1-(r-1)l_{0n_0},
\qquad
&&\bangle{u^-,v_{in_i}} =  1, \ i = 1, \dots, r.
\end{alignat*}
If there is an elliptic fixed point
$x^\pm \in X$, then the following holds for the
local Gorenstein index.
\begin{enumerate}
\item
$\iota(x^+)$ is the unique positive integer with
$\iota(x^+)u^+ \in \ZZ^{r+1}$ being primitive and
it is explicitly given by
\[
\iota^{+}
\ := \
\iota(x^{+})
\ = \
\frac{l^+\skr{m}^+}{\gcd(u_1^+,\dots, u_r^+, l^+\skr{l}^+)}.
\]
\item
$\iota(x^-)$ is the unique positive integer with
$\iota(x^-)u^- \in \ZZ^{r+1}$ being primitive and
it is explicitly given by
\[
\iota^{-}
\ := \
\iota(x^{-})
\ = \
 - \frac{l^-\skr{m}^-}{\gcd(u_1^-,\dots, u_r^-, l^- \skr{l}^-)}.
\]
\end{enumerate}
\end{proposition}

\begin{proof}
The characterizing properties of the linear forms
$u^\pm$ are a result of direct computation. Near
$x^\pm$, the characterizing properties of $u^\pm$
yield
\[
l^+\skr{m}^+ \mathcal{K}_X^0
\ = \
\div(\chi^{l^+\skr{m}^+u^+}),
\qquad\qquad
l^-\skr{m}^- \mathcal{K}_X^0
\ = \
\div(\chi^{l^-\skr{m}^-u^-}).
\]
Note that $l^\pm \skr{m}^\pm u^\pm$ is an integral
vector and thus $l^\pm \skr{m}^\pm$ is a multiple
of the local Gorenstein index $\iota(x^\pm)$.
The remaining assertions follow.
\end{proof}

\begin{corollary}
\label{cor:gorind}
For any $\KK^*$-surface $X = X(A,P)$,
the Gorenstein index $\iota_X$ equals
the least common multiple of 
the numbers $\iota_i^\pm, \iota_{ij}, \iota^\pm$
associated with its (possible)
parabolic, hyperbolic and elliptic fixed
points.
\end{corollary}

We characterize log terminality of
rational $\KK^*$-surfaces. Recall that for any
variety $X$ with a $\QQ$-Cartier canonical
divisor $\mathcal{K}_X$ one considers a
resolution of singularities
$\pi \colon X \to X'$ and the associated
\emph{ramification formula}
\[
\mathcal{K}_X'
\ = \
\pi^* \mathcal{K}_X + \sum a(E) E.
\]
Here, $E$ runs through the exceptional prime
divisors and the $a(E) \in \QQ$ are the
\emph{discrepancies} of $\pi \colon X' \to X$.
The singularities of $X$ are called
\emph{log terminal} if $a(E)>-1$ for each $E$.

\begin{proposition}
\label{prop:logtermchar}
Consider a $\KK^*$-surface $X = X(A,P)$ and
its canonical resolution $\pi \colon X'' \to X$.
\begin{enumerate}
\item
If there is an elliptic fixed point $x^+ \in X$,
then the discrepancy of $D_{X''}^+$ is given by
\[
\qquad
a(D_{X''}^+) \ = \ \frac{\skr{l}^+}{\skr{m}^+}-1.
\]
\item
If there is an elliptic fixed point $x^- \in X$,
then the discrepancy of $D_{X''}^-$ is given by  
\[
\qquad
a(D_{X''}^-) \ = \ -\frac{\skr{l}^-}{\skr{m}^-}-1.
\]
\item
If $\skr{l}^\pm > 0$ holds for $x^\pm \in X$,
then every exceptional divisor
$D \subseteq \pi^{-1}(x^\pm)$ has discrepancy
strictly greater than $-1$.
\item
Every exceptional divisor
$D \subseteq X'' \setminus \pi^{-1}(x^\pm)$ has
discrepancy strictly greater than $-1$.
\end{enumerate}
In particular, all points $x \in X$ distinct from
$x^\pm$ are log terminal and $x^\pm \in X$ is log
terminal if and only if $\skr{l}^\pm > 0$.
\end{proposition}

\begin{proof}
For (i), we compute the discrepancy on the
affine open subset $U^+ \subseteq X$ containing
all orbits that have $x^+ \in X$ in their
closure. That means we are in the case
$P=[v_{01}, \dots, v_{r1}]$. Let
$u \in \QQ^{r+1}$ represent $\mathcal{K}_{X}^0$.
Then
\[
\bangle{u,v_{01}} \ = \ -1 + (r-1)l_{01},
\qquad
\bangle{u,v_{i1}} \ = \ -1,  \quad i = 1, \dots, r.
\]
Additionally, with the Gorenstein index
$\iota = \iota_X$, we have $\pi^*(\iota
\mathcal{K}_{X}^0)=\div(\chi^{\iota u})$.
Plugging this into the ramification formula we
get
\[
-D_{X''}^+
\ = \
\bangle{u,v^+}D_{X''}^+ + a(D_{X''}^+)D_{X''}^+
\ = \
- \frac{\skr{l}^+}{\skr{m}^+} D_{X''}^+ + a(D_{X''}^+)D_{X''}^+.
\]
The evaluation of $u$ at $v^+$ is a direct
computation. We conclude that the discrepancy
$a(D_{X''}^+)$ is as claimed in the assertion.
The case of an elliptic fixed point $x^- \in X$
is treated analogously.

We show (iii) and (iv). First look at
the case that $D$ comes from a column $v$ of
$P''$ with $v \in \cone(v^+,v_{i1})$ for
$i=1,\dots,r$. Then $v = b v^+ + c v_{i1}$ with
positive $b,c \in \QQ$. Using the linear form $u$
from above, we compute
\[
-D
\ = \
\bangle{u,v}D + a(D)
\ = \
\bangle{u,bv^++cv_{i1}}D + a(D)
\ = \
-\left(b\frac{\skr{l}^+}{\skr{m}^+} + c \right) D + a(D)D.
\]
Since $\skr{l}^+ > 0$, we can conclude $a(D)>-1$.
If $v \in \cone(v^+, v_{01})$ the same reasoning
works with the canonical divisor
$\mathcal{K}_X^1$. Moreover, the arguments adapt
to the cases $v \in \sigma^-$ and
$v \in \tau_{ij}$.
\end{proof}

A tuple $(q_0,\dots, q_r)$ of positive integers
is called \emph{platonic} if
$q_0^{-1} + \dots + q_r^{-1} > r-1$.
If $q_0 \ge \dots \ge q_r$ holds, then platonicity is
equivalent to $q_3 = \dots = q_r = 1$ and
$(q_0,q_1,q_2)$ being one of
\[
(q_0,q_1,1),
\quad
(q_0,2,2),
\quad
(5,3,2),
\quad
(4,3,2),
\quad
(3,3,2).
\]
For a $\KK^*$-surface $X=X(A,P)$ with $P$-slope
ordered, the \emph{elliptic tuples} associated
with the (possible) elliptic fixed points are
$(l_{01}, \dots, l_{r1})$ for $x^+$
and $(l_{0n_0}, \dots, l_{rn_r})$ for~$x^-$.

\begin{corollary}
\label{cor:logtermchar}
Let $X=X(A,P)$ with $P$ slope-ordered.
Then $X$ is log terminal, if and and only if
each of its elliptic tuples is platonic. 
\end{corollary}

Log terminal surface singularities have
been intensely studied; see \cite{Bk} for a
classical reference. They are known to be
precisely the quotient singularities $\KK^2/G$,
where $G \subseteq \GL(2,\KK)$ is a finite
subgroup and any such $\KK^2/G$ comes
with the $\KK^*$-action induced by scalar
multiplication. This turn allows a
treatment in terms of defining matrices.

\begin{summary}
\label{sum:canmult}
Consider a rational affine  $\KK^*$-surface $X$
with an elliptic fixed point $x^\pm \in X$. Then
there is a rational function $f \in \KK(X)$ with
\[
\div(f) \ = \ \iota_X \mathcal{K}_{X},
\]
where $\iota_X = \iota(x^\pm)$ is the Gorenstein
index of~$X$.
The \emph{canonical multiplicity
$\zeta_X$ of~$X$} is the weight of $f$ with
respect to the $\KK^*$-action, hence
\[
f(t \mal x) \ = \ t^\zeta f(x),
\]
whenever $f$ is defined at $x$; see~\cite[Def.~4.2]{ArBrHaWr}.
For $X = X(A,P)$, using the linear form $u^\pm$ from Proposition
\ref{prop:kstar-gorind-2}, we can write 
\[
\iota_X \mathcal{K}_{X}
\ = \
\div ( \chi^{\iota_X u^\pm} ).
\]
Keeping in mind that $\KK^*$ acts on
$X \subseteq Z$ as the subgroup of the acting
torus $\TT^{r+1}$ given by
$t \mapsto (1,\dots,1,t)$, we see
\[
\zeta_X
\ = \
\iota_X \bangle{ u^\pm, e_{r+1}}
\ = \
\iota_X \frac{\skr{l}^\pm}{\skr{m}^\pm}.
\]
\end{summary}

Recall that Proposition~\ref{prop:affineiota}
traets the case of toric singularities
$\KK^2/G$.
In the non-toric case, we have the following;
see also~\cite[Ex.~4.8]{ArBrHaWr}.


\begin{proposition}
\label{prop:affine-non-toric-ell}
Let $X$ be a non-toric log terminal affine
$\KK^*$-surface with an elliptic fixed point
$x \in X$. Then $X \cong X(A,P)$ with $P$ being
one of the following.

\bigskip
\noindent
\[
\begin{array}{llllll}
&
\text{Type $D_n^{\zeta,\iota}$:} 
&
&
&
\text{Type $E_6^{\zeta,\iota}$:}
&    
\\[5pt]
&  
\left[
\begin{array}{ccc}
-l_0 & 2 & 0
\\
-l_0 & 0 & 2
\\
\frac{\pm \iota - \zeta l_0}{\zeta} & 1 & 1
\end{array}
\right],
&
\begin{array}{l}
\scriptstyle \zeta = 1,2,
\\
\scriptstyle \gcd(\iota, \zeta l_0) = \zeta.
\end{array}
&
&
\left[
\begin{array}{ccc}
-3 & 3 & 0
\cr
-3 & 0 & 2
\cr
\frac{\pm \iota-5 \zeta}{2 \zeta} & 1  & 1
\end{array}
\right],
&   
\begin{array}{l}
\scriptstyle \zeta = 1,3,
\\
\scriptstyle \gcd(\zeta \pm \iota, 6 \zeta) = 2\zeta.
\end{array}
\\[25pt]
&
\text{Type $E_7^{\iota}$:} 
&
&
&
\text{Type $E_8^{\iota}$:}
&    
\\[5pt]
&
\left[
\begin{array}{ccc}
-4 & 3 & 0
\cr
-4 & 0 & 2
\cr
\frac{\pm \iota - 10}{3} & 1 & 1
\end{array}
\right],
&
\begin{array}{l}
\scriptstyle \zeta = 1,
\\  
\scriptstyle \gcd(2 \pm \iota,12) = 3.
\end{array}
&
&
\left[
\begin{array}{ccc}
-5 & 3 & 0
\cr
-5 & 0 & 2
\cr
\frac{\pm \iota - 25}{6}& 1 & 1
\end{array}
\right],
&
\begin{array}{l}
\scriptstyle \zeta = 1,
\\
\scriptstyle \gcd(5 \pm \iota,30) = 6.
\end{array}
\end{array}
\]

\bigskip
\noindent
In all cases, $\iota$ denotes the
Gorenstein index and $\zeta$ the canonical
multiplicity of $X$. Moreover, a ``$+\iota$'' in
the defining matrix gives $x=x^+$ and a
``$-\iota$'' gives $x=x^-$. 
\end{proposition}

\begin{proof}
We may assume $X=X(A,P)$. Then, as $X$ is affine,
$n_0 = \dots = n_r = 1$ holds. Since $X$ is
non-toric and log terminal, we have $l_{i1}>1$
exactly three times. Thus, we arrive at defining
$3 \times 3$ matrices $P$ of the shape
\[
\left[
\begin{array}{rrr}
-l_0 & 2 & 0
\cr
-l_0 & 0 & 2
\cr
d_0 & 1 & 1
\end{array}
\right],
\quad
l_0 \ge 2,
\qquad\qquad
\left[
\begin{array}{rrr}
-l_0 & 3 & 0
\cr
-l_0 & 0 & 2
\cr
d_0 & 1 & 1
\end{array}
\right],
\quad
l_0 = 3,4,5
\]
by applying suitable admissible operations and
removing erasable columns. Now compute the linear
forms $u^\pm$ from Proposition
\ref{prop:kstar-gorind-2} for these matrices.
Then $\iota_X u^\pm$ being primitive, $\iota_X$
dividing $\det(P)$ and
$\zeta_X \skr{m}^\pm = \iota_X  \skr{l}^\pm$ lead
to the assertion.
\end{proof}

\begin{example}
Look again at the log terminal elliptic fixed
points of given Gorenstein index $\iota$ from
Proposition~\ref{prop:affine-non-toric-ell}.
Resolving the singularity according to
Summary~\ref{sum:surfsingres} and computing the self
intersection numbers with the aid of
Summary~\ref{sum:intersections}, we get for $\iota=1$
the classical resolution graphs of type
$D_{n},E_{6},E_{7}$ and $E_{8}$:

\def\Dn-one-graph{
\begin{tikzpicture}[scale=0.6]
\sffamily
\node at (-2,0) {$D^{1,1}_n$:};
\draw (0,1) circle (10pt);
\draw (0,-1) circle (10pt);
\node[] (ao) at (0,1) {$\scriptscriptstyle -2$};
\node[] (au) at (0,-1) {$\scriptscriptstyle -2$};
\draw (2,0) circle (10pt);
\draw (4,0) circle (10pt);
\node[] (a1) at (2,0) {$\scriptscriptstyle -2$};
\node[] (a2) at (4,0) {$\scriptscriptstyle -2$};
\draw (8,0) circle (10pt);
\node[] (a3) at (8,0) {$\scriptscriptstyle -2$};
\draw[] (ao) edge (a1);
\draw[] (au) edge (a1);
\draw[] (a1) edge (a2);
\draw[dotted] (4.5,0) to (7.5,0);
\node[right] (n) at (10,0) {$n \in \ZZ_{\ge 1}$};
\end{tikzpicture}
}

\def\Esix-one-graph{
\begin{tikzpicture}[scale=0.6]
\sffamily
\node at (-2,0) {$E^{1,1}_6$:};
\draw (0,0) circle (10pt);
\node[] (a0) at (0,0) {$\scriptscriptstyle -2$};
\draw (2,0) circle (10pt);
\node[] (a1) at (2,0) {$\scriptscriptstyle -2$};
\draw (4,0) circle (10pt);
\node[] (a2) at (4,0) {$\scriptscriptstyle -2$};
\draw (4,1.5) circle (10pt);
\node[] (a2o) at (4,1.5) {$\scriptscriptstyle -2$};
\draw (6,0) circle (10pt);
\node[] (a3) at (6,0) {$\scriptscriptstyle -2$};
\draw (8,0) circle (10pt);
\node[] (a4) at (8,0) {$\scriptscriptstyle -2$};
\draw[] (a0) edge (a1);
\draw[] (a1) edge (a2);
\draw[] (a2) edge (a3);
\draw[] (a2o) edge (a2);
\draw[] (a3) edge (a4);
\end{tikzpicture}
}

\def\Eseven-one-graph{
\begin{tikzpicture}[scale=0.6]
\sffamily
\node at (-2,0) {$E^{1,1}_7$:};
\draw (0,0) circle (10pt);
\node[] (a0) at (0,0) {$\scriptscriptstyle -2$};
\draw (2,0) circle (10pt);
\node[] (a1) at (2,0) {$\scriptscriptstyle -2$};
\draw (4,0) circle (10pt);
\node[] (a2) at (4,0) {$\scriptscriptstyle -2$};
\draw (4,1.5) circle (10pt);
\node[] (a2o) at (4,1.5) {$\scriptscriptstyle -2$};
\draw (6,0) circle (10pt);
\node[] (a3) at (6,0) {$\scriptscriptstyle -2$};
\draw (8,0) circle (10pt);
\node[] (a4) at (8,0) {$\scriptscriptstyle -2$};
\draw (10,0) circle (10pt);
\node[] (a5) at (10,0) {$\scriptscriptstyle -2$};
\draw[] (a0) edge (a1);
\draw[] (a1) edge (a2);
\draw[] (a2) edge (a3);
\draw[] (a2o) edge (a2);
\draw[] (a3) edge (a4);
\draw[] (a4) edge (a5);
\end{tikzpicture}
}

\def\Eeight-one-graph{
\begin{tikzpicture}[scale=0.6]
\sffamily
\node at (-2,0) {$E^{1,1}_8$:};
\draw (0,0) circle (10pt);
\node[] (a0) at (0,0) {$\scriptscriptstyle -2$};
\draw (2,0) circle (10pt);
\node[] (a1) at (2,0) {$\scriptscriptstyle -2$};
\draw (4,0) circle (10pt);
\node[] (a2) at (4,0) {$\scriptscriptstyle -2$};
\draw (4,1.5) circle (10pt);
\node[] (a2o) at (4,1.5) {$\scriptscriptstyle -2$};
\draw (6,0) circle (10pt);
\node[] (a3) at (6,0) {$\scriptscriptstyle -2$};
\draw (8,0) circle (10pt);
\node[] (a4) at (8,0) {$\scriptscriptstyle -2$};
\draw (10,0) circle (10pt);
\node[] (a6) at (12,0) {$\scriptscriptstyle -2$};
\draw (12,0) circle (10pt);
\node[] (a5) at (10,0) {$\scriptscriptstyle -2$};
\draw[] (a0) edge (a1);
\draw[] (a1) edge (a2);
\draw[] (a2) edge (a3);
\draw[] (a2o) edge (a2);
\draw[] (a3) edge (a4);
\draw[] (a4) edge (a5);
\draw[] (a5) edge (a6);
\end{tikzpicture}
}

\medskip

\Dn-one-graph 

\medskip

\Esix-one-graph

\medskip

\Eseven-one-graph

\medskip

\Eeight-one-graph

\medskip

\noindent For Gorenstein index $\iota = 2$, none
of the types $D$ and $E$ can occur. In Gorenstein
indices $\iota = 3,4$, the simplest examples are

\def\Dnthreegraph{
\begin{tikzpicture}[scale=0.6]
\sffamily
\node at (-12,0) {$D^{1,3}_4$:};
\node at (-7,0) {$P \ = \ 
\left[              
\begin{array}{rrr}
-2 & 2 & 0
\\
-2 & 0 & 2
\\
1 & 1 & 1
\end{array}                   
\right]
$};
\draw (0,1) circle (10pt);
\draw (0,-1) circle (10pt);
\node[] (ao) at (0,1) {$\scriptscriptstyle -2$};
\node[] (au) at (0,-1) {$\scriptscriptstyle -2$};
\draw (2,0) circle (10pt);
\draw (4,0) circle (10pt);
\node[] (a1) at (2,0) {$\scriptscriptstyle -3$};
\node[] (a2) at (4,0) {$\scriptscriptstyle -2$};
\draw[] (ao) edge (a1);
\draw[] (au) edge (a1);
\draw[] (a1) edge (a2);
\end{tikzpicture}
}

\def\Dnfourgraph{
\begin{tikzpicture}[scale=0.6]
\sffamily
\node at (-12,0) {$D^{2,4}_4$:};
\node at (-7,0) {$P \ = \ 
\left[              
\begin{array}{rrr}
-3 & 2 & 0
\\
-3 & 0 & 2
\\
-1 & 1 & 1
\end{array}                   
\right]
$};
\draw (0,1) circle (10pt);
\draw (0,-1) circle (10pt);
\node[] (ao) at (0,1) {$\scriptscriptstyle -2$};
\node[] (au) at (0,-1) {$\scriptscriptstyle -2$};
\draw (2,0) circle (10pt);
\draw (4,0) circle (10pt);
\node[] (a1) at (2,0) {$\scriptscriptstyle -2$};
\node[] (a2) at (4,0) {$\scriptscriptstyle -3$};
\draw[] (ao) edge (a1);
\draw[] (au) edge (a1);
\draw[] (a1) edge (a2);
\end{tikzpicture}
}

\Dnthreegraph

\vskip .5cm

\Dnfourgraph

\end{example}


\section{Log del Pezzo $\KK^*$-surfaces}

In a first part of this section, we discuss
geometric consequences of the del Pezzo property.
We show that every del Pezzo $\KK^*$-surface
naturally embeds into a toric Fano variety and
every non-toric del Pezzo $\KK^*$-surface admits
at most one parabolic fixed point curve
and at most one non-log terminal singularity.
As an application, we bound the number of
singularities of a log del Pezzo surface in
terms of its Picard number; see
Proposition~\ref{prop:singbounds}.
Then, in the second part, we present and discuss
the central combinatorial tool for the subsequent
classification, the anticanonical complex.

We enter the first part.
Recall that a Fano variety is a normal projective
variety admitting an ample anticanonical divisor.
Thus, the del Pezzo surfaces are precisely the
two-dimensional Fano varieties. 
The following observation provides a close
connection between del Pezzo $\KK^*$-surfaces and
toric Fano varieties.

\begin{proposition}
\label{prop:fano-completion}
Consider $X = X(A,P)$ and the toric embedding
$X \subseteq Z$ from Construction
\ref{constr:kstarsurf}. If $X$ is a del Pezzo
surface the following holds.
\begin{enumerate}
\item
The anticanonical class $w = -w_Z \in K$ of $Z$
yields a Fano toric completion
$X \subseteq Z \subseteq Z(w)$.
\item
The defining fan $\Sigma(w)$ of $Z(w)$ from (i)
is spanned by the convex hull
$\mathcal{A} := \conv(v_1, \ldots, v_r)$ over the
columns of $P$.
\end{enumerate}
\end{proposition}

\begin{proof}
By Summary \ref{sum:cl-data} the relation degree
$\mu$ lies in $\Mov(Z)=\SAmple(X)$.
The anticanonical divisor classes $-w_X$ of $X$
and $-w_Z$ of $Z$ are related via
\[
-w_X
\ = \
-w_Z - (r-1)\mu.
\]
Since $-w_X$ is ample, we conclude that $-w_Z$
lies in the interior of the moving cone
$\Mov(Z)$. Thus, Summary \ref{sum:completions}
gives the desired completion.
\end{proof}

Given $X = X(A,P)$, the Fano toric
completion $Z \subseteq Z(w)$ from Proposition
\ref{prop:fano-completion} is uniquely determined
by the Cox ring of $Z$ and hence by $P$. Although
there always exist $\QQ$-factorial toric
completions $Z \subseteq Z'$, the Fano toric
completion $Z \subseteq Z(w)$ need not be
$\QQ$-factorial in general. We present a concrete
example.

\begin{example}
Consider the projective $\KK^*$-surface
$X = X(A,P)$ and its toric embedding
$X \subseteq Z$ for the defining matrix
$$
P
\ = \
\left[
\begin{array}{rrrrr}
-1 & -1 & 1 & 1 & 0 
\\
-1 & -1 & 0 & 0 & 2
\\
0 & -1 & 1 & 0 & 1
\end{array} 
\right].    
$$
Then $X$ is a Gorenstein del Pezzo surface and in
$\Cl(X) = \ZZ^2 = \Cl(Z)$  the degrees
$w_{ij}=\deg(T_{ij})$ are located as follows.
\begin{center}
\begin{tikzpicture}[scale=0.6]
\sffamily
\coordinate(oo) at (0,0);
\coordinate(w01) at (-1,1);
\coordinate(w02) at (1,1);
\coordinate(w11) at (1,0);
\coordinate(w12) at (-1,2);
\coordinate(w21) at (0,1);
\coordinate(c01) at (-1.8,1.8);
\coordinate(c02) at (2,2);
\coordinate(c11) at (3,0);
\coordinate(c12) at (-1.3,2.6);
\coordinate(c21) at (0,3);
\path[fill, color=gray!20] (oo) -- (c11) -- (c02) -- (oo);
\path[fill, color=gray!40] (oo) -- (c02) -- (c21) -- (oo);
\path[fill, color=gray!40] (oo) -- (c21) -- (c12) -- (oo);
\path[fill, color=gray!20] (oo) -- (c12) -- (c01) -- (oo);
\path[fill, color=black] (w01) circle (1mm) 
node[left]{$\scriptscriptstyle w_{01}$};
\path[fill, color=black] (w02) circle (1mm) 
node[right]{$\scriptscriptstyle w_{02}$};
\path[fill, color=black] (w11) circle (1mm)
node[below]{$\scriptscriptstyle w_{11}$};
\path[fill, color=black] (w12) circle (1mm)
node[above right]{$\scriptscriptstyle w_{12}$};
\path[fill, color=black] (w21) circle (1mm)
node[above right]{$\scriptscriptstyle w_{21}$};
\draw[thick,color=black] (oo) to (c01);
\draw[thick,color=black] (oo) to (c02);
\draw[thick,color=black] (oo) to (c11);
\draw[thick,color=black] (oo) to (c12);
\draw[thick,color=black] (oo) to (c21);
\draw[color=black] (-3,0) to (3,0);
\draw[color=black] (0,-1) to (0,3);
\end{tikzpicture}
\end{center}
The anticanonical divisor classes of $X$ and $Z$
are $-w_X = (0,3)$ and $-w_Z = (0,5)$. With
$w := -w_z$ the maximal cones of the fan
$\Sigma(w)$ are given by
\begin{align*}
&\cone(v_{01},v_{02},v_{21}),
\\
\cone(v_{01},v_{11},v_{21}),
\
&\cone(v_{01},v_{02},v_{11},v_{12}),
\
\cone(v_{02},v_{12},v_{21}),
\\
&\cone(v_{11},v_{12},v_{21}).
\end{align*}
In particular, the Fano toric completion
$Z \subseteq Z(w)$ is not $\QQ$-factorial.
Refining the fan $\Sigma(w)$ by inserting
into $\cone(v_{01},v_{02},v_{11},v_{12})$
one of the diagonals
\[
\cone(v_{01},v_{12}), \qquad \cone(v_{11},v_{02})
\]
yields the two $\QQ$-factorial toric completions
$Z \subseteq Z'$ and $Z \subseteq Z''$ given by
the members $\cone(w_{02},w_{21})$ and
$\cone(w_{21},w_{12})$ of the secondary fan.
\end{example}

Before proceeding, we provide important inequalities
on the entries of the defining
matrix $P$ of a rational $\KK^*$-surface $X = X(A,P)$.
The first series of inequalities results from
slope-orderedness and irredundance; note that we don't
use the del Pezzo property there.

\begin{proposition}
\label{lem:rviamzeta}
Let $P$ be a defining matrix as in Construction
\ref{constr:kstarsurf}.
Assume that~$P$ is
slope-ordered, irredundant and $r \ge 2$. Then
the following holds.
\begin{enumerate}
\item
$
\lceil m_{ij} \rceil - \frac{l_{ij}-1}{l_{ij}}
\ \le \ 
m_{ij}
\ \le \
\lfloor m_{ij} \rfloor + \frac{l_{ij}-1}{l_{ij}}
$.
\item
$\lceil m_{i1} \rceil - \lfloor m_{in_i} \rfloor
\ \ge \ 1
$.
\item
$
r+1
\ \le \
(\skr{m}^+ - \skr{l}^+)  - (\skr{m}^- + \skr{l}^-) + 4
$.
\item
$\skr{m}^{+} \ge \skr{l}^{+}$ or $\skr{m}^{-} \le -\skr{l}^{-}$.
\end{enumerate}
\end{proposition}

\begin{proof}
The first assertion is obvious. For the second
one note that $m_{i1} \ge m_{in_i}$ holds due
to slope-orderedness. Furthermore
\[
\lceil m_{i1} \rceil = \lfloor m_{in_i} \rfloor
\ \implies \
m_{i1} = m_{in_i} \in \ZZ
\ \implies \
n_i=1, \  l_{i1}=1.
\]
Thus, $\lceil m_{i1} \rceil =
\lfloor m_{in_i} \rfloor$ cannot happen since $P$
is irredundant with $r \ge 2$. The third
assertion is now deduced from the first two:
\begin{eqnarray*}
r+1
& \le &
\sum_{i=0}^r \lceil m_{i1} \rceil - \lfloor m_{in_i} \rfloor
\\
& \le &
\sum_{i=0}^r m_{i1} + \frac{l_{i1}-1}{l_{i1}}
-
\sum_{i=0}^r m_{in_i} - \frac{l_{in_i}-1}{l_{in_i}}
\\
& = &
(\skr{m}^+ - \skr{l}^+)  - (\skr{m}^- + \skr{l}^-) + 4.
\end{eqnarray*}
For the fourth assertion, assume
$\skr{m}^{+}<\skr{l}^{+}$ and
$\skr{m}^{-}>-\skr{l}^{-}$. Then (iii) gives the
estimate
\[
r+1
\ \le \
(\skr{m}^+ - \skr{l}^+)  - (\skr{m}^- + \skr{l}^-) + 4
\ < \
4.
\]
So $r+1 = 3$. Consequently, using (ii), we obtain
\[
3
\ \le \ 
\sum_{i=0}^2 \lceil m_{i1} \rceil - \lfloor m_{in_i} \rfloor
\ = \
\sum_{i=0}^2 \lceil m_{i1} \rceil - \sum_{i=0}^2 \lfloor m_{in_i} \rfloor.
\]
This contradicts the subsequent two estimates,
showing that the right hand side equals at most
two.
\begin{alignat*}{3}
2
\ &> \
2 + \skr{m}^+ - \skr{l}^+
\ &&= \
\sum_{i=0}^2 m_{i1} +  \frac{l_{i1}-1}{l_{i1}}
\ &&\ge \
\sum_{i=0}^2 \lceil m_{i1} \rceil,
\\
2
\ &> \
2 - \skr{m}^- - \skr{l}^-
\ &&= \
- \sum_{i=0}^2 m_{in_i} -  \frac{l_{in_i}-1}{l_{in_i}}
\ &&\ge \
- \sum_{i=0}^2 \lfloor m_{in_i} \rfloor.
\end{alignat*}
\end{proof}

The second series of inequalities arises from
characterizing the del
Pezzo property via Kleiman's ampleness criterion,
telling us that a divisor is ample if
and only if it has positive intersection with any
effective curve.

\begin{proposition}
\label{prop:kleimanample}
A projective $\KK^*$-surface $X = X(A,P)$ is a
del Pezzo surface if and only if the following
inequalities are satisfied.
\begin{itemize}
\item $\skr{l}^+ > \skr{m}^+$ if there is
$D_X^+ \subseteq X$,
\item
$\skr{l}_{ij-1}\skr{m}_{ij-1} >
\skr{l}_{ij}\skr{m}_{ij}$,
$i = 0, \ldots, r$, $j = 1, \ldots, n_i$,
\item
$\skr{l}^- > - \skr{m}^-$ if there is
$D_X^- \subseteq X$.
\end{itemize}
\end{proposition}

\begin{proof}
By Kleiman's criterion, $X$ is del Pezzo if and
only $-\mathcal{K}_X \cdot D > 0$ for all
effective curves $D$ on $X$. The latter is true
if and only if $-\mathcal{K}_X$ has positive
intersection with all $D_X^{ij}$ and $D_X^\pm$.
Due to Proposition \ref{prop:ak-intersections},
the latter is equivalent to the inequalities of
the assertion.
\end{proof}

We are ready to prove the announced statments
on possible parabolic and elliptic fixed points
of del Pezzo $\KK^*$-surfaces.

\begin{proposition}
\label{prop:delPezzo-parabolic}
Let $X$ be a non-toric rational del Pezzo
$\KK^*$-surface. Then $X$ admits at most one
parabolic fixed point curve.
\end{proposition}

\begin{proof}
We may assume $X = X(A,P)$ with slope-ordered,
irredundant $P$ and $r \ge 2$. Suppose that there
are $D_X^+$ and $D_X^-$. Then Proposition
\ref{prop:kleimanample} yields
$\skr{l}^{+}>\skr{m}^{+}$ and
$\skr{l}^{-}>-\skr{m}^{-}$. This contradicts
Proposition \ref{lem:rviamzeta} (iv).
\end{proof}

Let us explore the impact of the del
Pezzo property on the possible singularities of a
rational projective $\KK^*$-surface. Recall that
at most elliptic fixed points can be non log
terminal and thus the number of non log terminal
singularities of any rational $\KK^*$-surface is
bounded by two.

\begin{proposition}
\label{prop:delp2onlyonenonlt}
Let $X$ be a rational del Pezzo $\KK^*$-surface.
Then $X$ can have at most one non log terminal
singularity. 
\end{proposition}

\begin{proof}
It suffices to treat the case that $X$ is
non-toric and comes with  $x^+ \in X$ and
$x^- \in X$. Moreover, we may assume $X = X(A,P)$
with slope-ordered, irredundant $P$ and
$r \ge 2$. For any $i$, Proposition
\ref{prop:kleimanample} yields 
\[
\frac{\skr{l}^+}{\skr{m}^+}
\ = \
\skr{l}_{i0} \skr{m}_{i_0}
\ > \ \ldots \ > \
\skr{l}_{in_i} \skr{m}_{in_i}
\ = \
\frac{\skr{l}^-}{\skr{m}^-}.
\]
Now, if $\skr{l}^+ > 0$ then at most $x^-$ can be
non log terminal and we are done. If
$\skr{l}^{+} \le 0$, then $\skr{m}^{-}<0$ and the
above (strict) inequalities imply
$\skr{l}^{-}>0$. Thus, at most $x^+$ can be non
log terminal.
\end{proof}

The following example gives an infinite
series of rational del Pezzo $\KK^*$-surfaces of
Gorenstein index two, each coming with a non log
terminal singularity. It also shows that the
class of del Pezzo surfaces of fixed Gorenstein
index is not finite.

\begin{example}
Consider the projective $\KK^*$-surfaces
$X_l = X(A,P_l)$ given by the defining data  
\[
A
\ = \
\left[
\begin{array}{ccc}
1 & 0 & -1
\\
0 & 1 & -1
\end{array}
\right],
\qquad
P
\ = \
\left[
\begin{array}{cccc}
-l & 2 & 0 & 0  
\\
-l & 0 & 6 & 1
\\
-\frac{l+1}{2} & 1 & 1 & 0
\end{array}
\right],
\quad
5 \le l \in 2\ZZ +1.
\]
Then $X_l$ is a del Pezzo surface of Gorenstein index~2,
has $x^{+}$ as its only singularity and is not log terminal
due to Corollary \ref{cor:logtermchar}.
\end{example}

As an application of Propositions \ref{prop:delPezzo-parabolic}
and \ref{prop:delp2onlyonenonlt}, we obtain
bounds on the number of singularities.
Recall that for any log del Pezzo surface of Picard
number one, results of Keel/McKernan \cite{KeMK}
and Belousov \cite{Bel} tell us that the number
of singularities is sharply bounded by four. In
presence of a $\KK^*$-action, we can extend this
statement to higher Picard numbers.

\begin{proposition}
\label{prop:singbounds}
Let $X$ be a del Pezzo $\KK^*$-surface of
Picard number $\rho(X)$. Then the number $s(X)$
of singularities of $X$ satisfies
\[
s(X)
\ \le \ 
\begin{cases}
\rho(X)+2, & \quad \text{if $X$ has no parabolic fixed points},
\\
2\rho(X)+2, & \quad \text{if $X$ has parabolic fixed points and is log terminal}.
\end{cases}
\]
\end{proposition}

\begin{lemma}
\label{lem:singbounds}
Consider a non-toric $X = X(A,P)$ with $P$ irredundant
having a log terminal elliptic fixed point.
Then the number $r+1$ of arms of $X$ is bounded
by
\[
r+1 \ \le \ \rho(X) + 3 - m.
\]
Moreover, according to the possible
constellations of source and sink in $X$, the
number $s(X)$ of singularities of $X$ is bounded
as follows:
\[
\text{\rm (e-e):} \enspace s(X) \le \rho(X) + 2,
\qquad
\text{\rm (e-p), (p-e):} \enspace s(X) \le r + 1 + \rho(X) \le 2 \rho(X) + 2.
\]
\end{lemma}

\begin{proof}
Since $X$ is $\QQ$-factorial, $\rho(X)$ coincides
with the rank $\Cl(X)$ and hence equals $n+m-r-1$;
see Summary \ref{sum:classgroup-coxring}.
Since $P$ is irredundant and $X$ has a log terminal
elliptic fixed point, we can infer from
Corollary \ref{cor:logtermchar} that there are at
least $r-2$ arms having length greater or equal
than $2$.
Therefore
\[
\rho(X)
\ = \
m + \sum_{i=0}^r (n_i-1)
\ \ge \
m + r + 1 - 3.
\]
This proves the first statement. To estimate the
number of singularities we use that each of them is a
fixed point. Note that in any case the number of
hyperbolic fixed points of $X = X(A,P)$ is given
by
\[
(n_{0}-1) + \dots + (n_{r}-1).
\]
Thus, in the case (e-e), we have $\rho(X)+2$
fixed points in total. For the cases (e-p) and
(p-e) recall from Summary \ref{sum:singul-basics}
that there are at most $r+1$ singular parabolic
fixed points. Moreover, we have $\rho(X)-1$
hyperbolic fixed points.
\end{proof}

\begin{proof}[Proof of Proposition~\ref{prop:singbounds}]
First note that for a
projective toric surface $X$ the number of
singularities is at most the number of its fixed
points and is therefore bounded by $\rho(X)+2$.
If $X$ is non-toric, then
Propositions~\ref{prop:delPezzo-parabolic}
and~\ref{prop:delp2onlyonenonlt} ensure that
in either of the two cases under consideration,
we have a log terminal elliptic fixed point.
Thus, Lemma~\ref{lem:singbounds} gives the assertion.
\end{proof}

\begin{remark}
For rational $\KK^*$-surfaces coming with two parabolic fixed
point curves or with a non log terminal elliptic fixed
point and a parabolic fixed point curve, we can't expect to
bound the number of singularities in terms of the Picard number.
To see an example, take $n_{0}=\dots=n_{r}=1$
and consider 
\[
l_{01} = \dots = l_{r1} = 2,
\qquad
d_{01}=-1,
\qquad
d_{11} = \dots = d_{r1} = 1.
\]
This gives matrices
$P = [v_{01},\ldots,v_{r+1},v^+,v^-]$
and
$P = [v_{01},\ldots,v_{r+1},v^-]$
and associated $X = X(A,P)$.
In the first case, we have $\rho(X)=2$
and $2r+2$ singularities, $r+1$
on each fixed point curve.
In the second case, we have 
$\rho(X)=1$ and $r+2$ singularities,
$r+1$ on the fixed point curve
and the elliptic fixed point.
\end{remark}

We introduce the main tool of our
classification, the \emph{anticanonical complex}.
It was first introduced in~\cite{BeHaHuNi} to
study terminal Fano threefolds with an action of
a two-dimensional torus; see also \cite{HaMaWr,HiWr}
for further developments.
Differently as in~\cite{BeHaHuNi}, we give an explicit
construction of the complex, adapted to the surface case.

\begin{definition}
We call the defining matrix $P$ of $X = X(A,P)$
an \emph{LDP-matrix} if $X$ is a del Pezzo
surface with at most log terminal singularities.
\end{definition}

\begin{construction}
\label{constr:antican}
Consider an LDP-matrix $P$. Define vectors
$\tilde v^+ := \skr{d}^+v^+$ and
$\tilde v^- := \skr{d}^-v^-$ in $\RR^{r+1}$ by
\[
\skr{d}^+
 := 
\frac{\skr{m}^+}{\skr{l}^+},
\qquad\qquad
\skr{d}^-
 := 
\frac{\skr{m}^-}{\skr{l}^-}.
\]
We associate  with $P$ the two-dimensional
simplicial complex $\mathcal{A}_P$ in $\RR^{r+1}$
having as its cells 
\[
\kappa_{ij}
\ := \
\conv(0,v_{ij}, v_{ij+1}),
\qquad
i = 0, \ldots, r, \ j = 1, \ldots, n_i-1
\]
and, according to the possible cases (p-e), (e-p)
and (e-e), for $i = 0, \ldots, r$ the simplices
\[
\begin{array}{lll}
\text{(p-e)} 
& \kappa_i^+ := \conv(0, v^+,v_{i1}),        
& \kappa_i^- := \conv(0,\tilde v^-,v_{in_i}),
\\[3ex]                     
\text{(e-p)} 
& \kappa_i^+ := \conv(0, \tilde v^+,v_{i1}),
& \kappa_i^- := \conv(0,v^-,v_{in_i}),
\\[3ex]
\text{(e-e)} 
& \kappa_i^+ := \conv(0, \tilde v^+,v_{i1}),
& \kappa_i^- := \conv(0,\tilde v^-,v_{in_i}).
\end{array}                     
\]
Observe that the support of the simplicial
complex $\mathcal{A}_P$ is a subset of the
tropical variety $\trop(X) \subseteq \QQ^{r+1}$.
\end{construction}

\begin{remark}
\label{rem:generalmandzeta}
Corollary \ref{cor:mandzeta} ensures that for any
log del Pezzo $\KK^*$-surface $X=X(A,P)$ we can
set in accordance with Construction
\ref{constr:antican}:
\[
\skr{d}_X^+ \ := \ \frac{m_X^+}{\skr{l}_X^+},
\qquad\qquad
\skr{d}_X^- \ := \ \frac{m_X^-}{\skr{l}_X^-}.
\]
Moreover, using the properties of the surface
$X'$ from Remark \ref{rem:weaktropgeom}, we have
the following descriptions in terms of
intersection numbers.
\[
\skr{d}_X^+
\ = \
- \frac{D_{X'}^+ \cdot D_{X'}^+}{D_{X'}^+\cdot (D_{X'}^+-\mathcal{K}_{X'}^0)},
\qquad 
\skr{d}_X^-
\ = \
- \frac{D_{X'}^- \cdot D_{X'}^-}{D_{X'}^-\cdot (D_{X'}^-+\mathcal{K}_{X'}^0)}.
\]
\end{remark}

\begin{remark}
Consider a del Pezzo surface $X = X(A,P)$ with
$r=1$. Then $X$ is toric, thus log terminal
and the support of $\mathcal{A}_P$
equals the LDP-polygon of $X$.
\end{remark}

\begin{definition}
Let $P$ be an LDP-matrix and $X = X(A,P)$ the
corresponding $\mathbb{K}^{\ast}$-surface.
Consider the simplicial complex $\mathcal{A}_P$
associated with $P$.
\begin{enumerate}
\item
The \emph{interior} of $\mathcal{A}_P$
is the relative interior $\mathcal{A}_P^\circ$
of its support with respect to $\trop(X)$. 
\item
The \emph{outer vertices} of $\mathcal{A}_P$ are
the vertices of the complex $\mathcal{A}_P$ apart
from the origin.
\item
Given $k \in \ZZ_{\ge 1}$, we say that the
complex $\mathcal{A}_P$ is
\emph{almost $k$-hollow} if we have
$\mathcal{A}_P^\circ \cap k\ZZ^{r+1} = \{0\}$.
\end{enumerate} 
\end{definition}

\def\accrunningex{
\begin{tikzpicture}[scale=0.6]
\sffamily

\coordinate(cplus) at (0,2);
\coordinate(c01) at (-2,-0);
\coordinate(c02) at (-2,-1.5);
\coordinate(c11) at (1,-1.8);
\coordinate(c21) at (2,-.75);
\coordinate(cminus) at (0,-2);

\node[centered] at (0,-3) {$\mathcal{A}_P$};

\path[fill, color=gray!30] (ooo) -- (cplus) -- (c21) -- (cminus) -- (ooo);
\path[fill, color=gray!10] (ooo) -- (cplus) -- (c01) -- (c02) -- (cminus) -- (ooo);

\draw[thick, color=black] (ooo) to (c01);
\draw[thick, color=black] (ooo) to (c02);
\draw[thick, color=black] (ooo) to (c21);
\draw[thick, color=black] (ooo) to (cplus);
\draw[thick, color=black] (ooo) to (cminus);

\draw[thick, color=black] (cplus) to (c01);
\draw[thick, color=black] (cplus) to (c21);
\draw[thick, color=black] (c21) to (cminus);
\draw[thick, color=black] (c01) to (c02);

\path[fill, opacity=.9, color=gray!20] (ooo) -- (cplus) -- (c11) -- (cminus) -- (ooo);
\draw[thick, color=black] (ooo) to (c11);
\draw[thick, color=black] (cplus) to (c11);
\draw[thick, color=black] (c02) to (cminus);
\draw[thick, color=black] (c11) to (cminus);
\end{tikzpicture}
}

\begin{example}
\label{ex:runningexampleA5}
Consider again the $\KK^*$-surface $X = X(A,P)$
from Example \ref{ex:runningexampleA3}. The
defining matrix $P$ and the complex
$\mathcal{A}_P$ are given as
\[
\qquad\qquad
P
\ = \
\left[
\begin{array}{ccccc}
-3 & -5 & 2 & 0 & 0 
\\
-3 & -5 & 0 & 2 & 0 
\\
-4 & -8 & 1 & 1 & 1 
\end{array}    
\right]
\qquad\qquad
\vcenter{
\accrunningex
}
\]
The outer vertices of $\mathcal{A}_P$ are the
columns $v_{01},v_{02},v_{11},v_{21},v^+$ of $P$
together with the vector $\tilde v^{-} = (0,0,3)$.
\end{example}

We gather properties of $\mathcal{A}_P$
and in particular obtain that it coincides
with the anticanonical complex from~\cite{BeHaHuNi}.
Moreover, we will see that the
complex~$\mathcal{A}_P$ of a $\KK^*$-surface
$X(A,P)$ naturally generalizes the LDP-polygon of
a toric del Pezzo surface arising from a fan.

\begin{theorem}
\label{prop:ak-properties}
Consider an LDP-matrix $P$, its associated
simplicial complex $\mathcal{A}_P$ and the
projective $\KK^*$-surface $X = X(A,P)$.
\begin{enumerate}
\item
According to the possible constellations of
source and sink, the outer vertices of the
complex $\mathcal{A}_P$ are given by 

\medskip
\hspace*{1cm}{\rm (e-e)}\hspace*{.3cm}
$v_{ij}$ for $i=0,\dots,r,j=1,\dots,n_{i}$
and $\tilde v^+$, $\tilde v^-$,

\smallskip
\hspace*{1cm}{\rm  (e-p)}\hspace*{.3cm}
$v_{ij}$ for $i=0,\dots,r,j=1,\dots,n_{i}$ and
$\tilde v^+$, $v^-$,

\smallskip
\hspace*{1cm}{\rm (p-e)}\hspace*{.3cm}
$v_{ij}$ for $i = 0, \ldots r$, $j = 1, \ldots, n_i$
and $v^+$, $\tilde v^-$.

\medskip
\item
The simplicial complex $\mathcal{A}_P$ equals the
anticanonical complex of the log del Pezzo
$\KK^*$-surface $X = X(A,P)$.
\item
For all $i=0,\ldots,r$, the intersection
$\mathcal{A}_P \cap \lambda_i$ with the $i$-th
arm $\lambda_i \subseteq \trop(X)$ is a convex
polygon. Regarding possible outer vertices
$\tilde v^\pm$ we have
\[
\qquad
\tilde v^+ \ \not\in \ \conv(0,v_{01}, \ldots, v_{r1}),
\qquad 
\tilde v^- \ \not\in \ \conv(0,v_{0n_0}, \ldots, v_{rn_r}).
\]
\item
The discrepancy along an exceptional divisor
$E_\varrho \subseteq X''$ of the canonical
resolution $X'' \to X$ is given by
\[
a(E_\varrho) 
\ = \ 
\frac{\Vert v_\varrho \Vert}{\Vert \tilde v_\varrho  \Vert} - 1,
\]
where $v_\varrho \in \varrho$ is the primitive
vector and $\tilde v_\varrho \in \varrho$ is the
intersection point of $\varrho$ and the boundary
$\partial \mathcal{A}_P$.  
\item
The surface $X$ has at most $1/k$-log
canonical singularities if and only if the
anticanonical complex $\mathcal{A}_X$ is almost
$k$-hollow.
\end{enumerate}
\end{theorem}

\begin{proof}
Assertions (i), (ii) and (iv), (v) are covered by
\cite[Thm. 1.4, Prop. 2.3, Prop. 3.7 and Cor.
4.10]{BeHaHuNi}. We show (iii). The intersection
points of the boundary of $\mathcal{A}_P$ and the
facet $\conv(v_{01}, \ldots, v_{r1})$ with the
ray through $v^+$ are given by
\[
\tilde v^+ \ = \ \frac{\skr{m}^+}{\skr{l}^+} v^+,
\qquad 
\frac{\skr{m}^+}{\skr{l}^+ + r-1} v^+.
\]
Since $X$ is log terminal we have
$\skr{l}^{+}>0$. Together with $\skr{m}^{+}>0$,
this gives the assertion in case of the existence
of an elliptic fixed point $x^+$. The case of an
elliptic fixed point $x^-$ is analogous.
\end{proof}

As an application, we can bound the number
$r+1$ of rows of a defining matrix $P$ for
$1/k$-log canonical $\KK^*$-surfaces $X(A,P)$.
As we will see later in Example~\ref{ex:maxfam},
the presented bound is even sharp.
In more geomeric terms, the statement is the following.

\begin{proposition}
\label{prop:armbounds}
For any non-toric rational $1/k$-log
canonical del Pezzo $\KK^*$-surface $X$, the
number $r+1$ of critical values of
$\pi \colon X \dasharrow \PP_1$ is bounded as
follows.
\begin{enumerate}
\item  
If $X$ has two elliptic fixed points, then
$r+1 \le 4k$ holds.
\item
If $X$ has a parabolic fixed point curve, then
$r+1 \le 2k+1$ holds.
\end{enumerate}
\end{proposition}

\begin{proof}
We may assume $X = X(A,P)$ with slope-ordered and
irredundant $P$ and $r \ge 2$. According to
Theorem \ref{prop:ak-properties}, we have
$\skr{d}^+ \le k$ and $\skr{d}^- \ge -k$. This
gives
\[
\skr{m}^+ \ \le \ k\skr{l}^+,
\qquad\qquad
\skr{m}^- \ \ge \ -k\skr{l}^-.
\]
Assume that $X$ has two elliptic fixed points.
Combining the above estimates with Lemma
\ref{lem:rviamzeta} (iii) and using
$\skr{l}^{+}+\skr{l}^{-} \le 4$ as well as
$k \ge 1$ leads to 
\[
r+1
\ \le \ 
(\skr{m}^+ - \skr{l}^+)  - (\skr{m}^- + \skr{l}^-) + 4
\ \le \ 
(k-1)(\skr{l}^++\skr{l}^-) + 4
\ \le \
4k.
\]
Assume that $X$ has a parabolic fixed point
curve, say $D_X^+$. Then
$-\mathcal{K}_X \cdot D_X^+ > 0$ implies
$\skr{m}^{+}<\skr{l}^{+}$. Similarly as before we
conclude
\[
r+1
\ \le \ 
(\skr{m}^+ - \skr{l}^+)  - (\skr{m}^- + \skr{l}^-) + 4
\ < \ 
(k-1)\skr{l}^- + 4
\ \le \
2k + 2.
\]
\end{proof}


\section{Contractions and combinatorial minimality}

We discuss contractions and combinatorial minimality
of $\KK^*$-surfaces.
In particular, Proposition \ref{prop:kstarcontractdecomp}
provides an explicit description of contractions in
terms of defining matrices,
Proposition \ref{prop:dplusdminuscontract} tells about
the effect of a contraction on the anticanonical complex
and Proposition \ref{prop:combminprops} presents
geometric properties of combinatorially minimal 
$\KK^*$-surfaces.

\begin{definition}
By a \emph{contraction} we mean a proper
morphism $\psi \colon X \to Y$
of normal varieties such that
$\psi \colon \psi^{-1}(V) \to V$
is an isomorphism for some open subset
$V \subseteq Y$ with complement of
codimension at least two in $Y$.
\end{definition}



Note that our contractions are often
referred to as ``birational contractions''. 
We gather basic properties of
contractions. Recall that for any proper morphism
$\psi \colon X \to Y$ of normal varieties, we
have the push forward homomorphisms
\[
\psi_* \colon \WDiv(X) \ \to \ \WDiv(Y),
\qquad\qquad
\psi_* \colon \Cl(X) \ \to \ \Cl(Y),
\]
defined by sending a prime divisor
$D \subseteq X$ to $\psi(D) \subseteq Y$
if $\psi(D)$ is a prime divisor in $Y$
and to zero else.

\begin{proposition}
\label{prop:contractprops}
Let $\psi \colon X \to Y$ be a
contraction and consider
the associated push forward
homomorphisms on the Weil divisors
and the divisor classes.
\begin{enumerate}
\item
If $\mathcal{K}_X$ is a canonical divisor on $X$,
then its push forward $\psi_{\ast}(\mathcal{K}_X)$ is
a canonical divisor on $Y$.
\item
The homomorphisms
$\psi_{\ast} \colon \WDiv(X) \to \WDiv(Y)$
and 
$\psi_{\ast} \colon \Cl(X) \to \Cl(Y)$ 
are both surjective.
\item
For the cones of effective and
movable divisor classes in
$\Cl_{\mathbb{Q}}(X)$ and $\Cl_{\mathbb{Q}}(Y)$
we have 
\[
\psi_{\ast}(\Eff(X)) = \Eff(Y),
\qquad
\psi_{\ast}(\Mov(X)) = \Mov(Y).
\]
\end{enumerate} 
\end{proposition}

\begin{proof}
Take any open subset $V \subseteq Y$ with
complement of codimension at least two
in $Y$ such that $\psi \colon \psi^{-1}(V) \to V$
is an isomorphism.
Then canonical (principal, effective, movable) divisors $D$
on $X$ restrict to canonical (principal, effective, movable)
divisors on $\psi^{-1}(V) \cong V$ and thus yield
canonical (principal, effective, movable) divisors $\psi_*(D)$
on $Y$.
\end{proof}

If a contraction $\psi \colon X \to Y$ maps a
prime divisor $E \subseteq X$ onto a subset of codimension
at least two in $Y$, then we say $\psi$ \emph{contracts}
$E$ and call $E$ an \emph{exceptional divisor} of $\psi$.
Moreover, a prime divisor on a normal variety $X$
is called \emph{contractible} if it gets contracted
by some contraction $X \to Y$.

\begin{remark}
For any contraction $\psi \colon X \to Y$,
there is a finite (possibly empty) collection
$E_1, \ldots, E_q \subseteq X$ of exceptional
divisors.
\end{remark}

\begin{remark}
\label{rem:contractprops}
Consider normal complete varieties $X$, $Y$ and 
let $\psi \colon X \to Y$ be a contraction with
the exceptional divisors $E_1, \ldots, E_q$.
\begin{enumerate}
\item
If $X$ comes with a morphical action of a connected algebraic
group $G$, then $E_1, \ldots, E_q \subseteq X$
are invariant and $Y$ admits a morphical $G$-action making
$\psi \colon X \to Y$ equivariant, see \cite[Prop.~I.1]{Bl}.
\item
Assume that $X$ has finitely generated Cox ring $\mathcal{R}(X)$
and let $f_i \in \mathcal{R}(X)$ represent the canonical section of 
$E_i$. Then we have an isomorpism
\[
\mathcal{R}(X) / \bangle{1 - f_i; \ i = 1, \ldots, q}
 \ \cong \ 
\mathcal{R}(Y)
\]
induced by sending homogeneous elements $f \in \mathcal{R}(X)$
of degree $[D]$ to homogeneous elements $\psi_* f \in \mathcal{R}(Y)$ of
degree $[\psi_*D]$, see \cite[Prop.~4.1.3.1]{ArDeHaLa}.
\end{enumerate}
\end{remark}

We focus on the surface case. 
We begin with some general observations
and then specialize to the rational
projective $\KK^*$-surfaces.

\begin{remark}
Let $\psi \colon X \to Y$ be a contraction
of surfaces.
Then the exceptional divisors
$E_1,\ldots, E_q \subseteq X$
map to points $y_1,\ldots,y_q \in Y$
and $X \setminus (E_1 \cup \ldots \cup E_q)$
maps isomorphically onto
$Y \setminus \{y_1,\ldots,y_q\}$.
\end{remark}

\begin{proposition}
\label{prop:delpcontractstodelp}
Let $\psi \colon X \to Y$ be a contraction of
surfaces,
where $X$ has finitely generated Cox ring.
Then $Y$ has finitely generated Cox ring and
we have 
\[
\psi_*(\SAmple(X)) = \SAmple(Y),
\qquad\qquad
\psi_*(\Ample(X)) = \Ample(Y)
\]
for the cones of semiample and ample
divisor classes.
Moreover, if $X$ is a del Pezzo surface,
then $Y$ is a del Pezzo surface.
\end{proposition}

\begin{proof}
According to Remark~\ref{rem:contractprops}~(ii),
also $Y$ has a finitely generated Cox ring.
By~\cite[Thm.~4.3.3.5]{ArDeHaLa}, the
movable and semiample cones coincide in $\Cl_\QQ(X)$
and as well in $\Cl_\QQ(Y)$.
Thus, the first displayed equation follows from
Proposition~\ref{prop:contractprops}~(iii).
Moreover, the respective ample cones are the relative
interiors of the semiample cones, see
\cite[Prop.~3.3.2.9]{ArDeHaLa}. Hence, the second
displayed equation follows from the first one and
the fact that any linear map sends the interior of a
cone onto the interior of the image cone. For the
supplement, we use Proposition~\ref{prop:contractprops}~(i)
to see that $Y$ has an ample anticanonical divisor.
\end{proof}

\begin{corollary}
\label{cor:eqivarcontract}
Any contraction of a toric del Pezzo surface is
a toric del Pezzo surface and any contraction of a
del Pezzo $\KK^*$-surface is a del Pezzo
$\KK^*$-surface.
\end{corollary}

\begin{proof}
Remark~\ref{rem:contractprops} yields that
the contracted surface inherits the desired
torus action and Proposition~\ref{prop:delpcontractstodelp}
ensures that it is del Pezzo.
\end{proof}

\begin{proposition}
\label{prop:exccurveskstar}
For a contraction $X \to Y$ of $\KK^*$-surfaces,
every exceptional divisor $E \subseteq X$
is either a parabolic fixed point curve or
it is an orbit closure containing a hyperbolic
fixed point.
\end{proposition}

\begin{proof}
Otherwise, being invariant, $E$ is an orbit closure
containing a point from the source and a point from
the sink. Hence, source and sink of $Y$ would
intersect in the image point of $E$, which is
impossible.
\end{proof}


Now we study contractions of
$\KK^*$-surfaces in terms of defining data.
Observe that the special case of defining
matrices $P$ with $r=1$ provides a full
treatment of contractions of toric surfaces,
see also Remark \ref{rem:toric-XAP}.

\begin{definition}
Let $P$ be a defining matrix as in Construction
\ref{constr:kstarsurf}. We call a column of $P$
\emph{contractible} if it lies in the cone generated
by the remaining ones.
\end{definition}

\begin{remark}
\label{characterizationweightexc}
Consider a slope-ordered defining matrix $P$.
A column $v$ is contractible if and only
if one of the following conditions is satisfied:
\begin{enumerate}
\item
$n_{i} \ge 2, v=v_{i1}$ and
$\skr{m}^{+}-m_{i1}+m_{i2}>0$,
\item
$v=v_{ij}$ and $1<j<n_{i}$,
\item
$n_{i} \ge 2, v=v_{in_{i}}$ and
$\skr{m}^{-}-m_{in_{i}}+m_{in_{i}-1}<0$,
\item
$v=v^{+}$ and $\skr{m}^{+}>0$,
\item
$v=v^{-}$ and $\skr{m}^{-}>0$.
\end{enumerate}
\end{remark}

\begin{construction}
\label{constr:contract}
Consider a projective $\KK^*$-surface
$X_1 = X(A,P_1)$ and assume that $P_1$ has a
contractible column $v$. Then, erasing $v$ from
$P_1$ yields a defining matrix $P_2$ of a
projective $\KK^*$-surface $X_2 = X(A,P_2)$.
Moreover, we obtain a commutative diagram
\[
\xymatrix{
X_1
\ar[r]
\ar[d]_{\psi_v}
&
Z_1
\ar[d]^{\psi_v}
\\
X_2
\ar[r]
&
Z_2  
}
\]
involving the $\KK^*$-surfaces $X_i$
and their ambient toric varieties $Z_i$.
The downward maps contract the prime divisors
$D_{X_1} \subseteq X_1$ and $D_{Z_1} \subseteq Z_1$
corresponding to the contracted column $v$ of $P_1$.
In particular, the downward maps are non-trivial
contractions. 
\end{construction}

\begin{remark}
\label{rem:obviouscontractible}
Consider a $\KK^*$-surface $X=X(A,P)$, where
the matrix $P$ is slope-ordered.  
\begin{enumerate}
\item
Assume that there is a curve $D_X^+ \subseteq X$. 
Then for every $i = 0, \ldots, r$ with $n_i \ge 2$,
each column $v_{ij}$ with $j = 1, \ldots, n_{i-1}$
is contractible.
\item
Assume that there is a curve $D_X^- \subseteq X$. 
Then for every $i = 0, \ldots, r$ with $n_i \ge 2$,
each column $v_{ij}$ with $j = 2, \ldots, n_i$
is contractible.
\item
Assume that $X$ has two elliptic fixed points.
Then for every $i = 0, \ldots, r$ with $n_i \ge 3$,
each column $v_{ij}$ with $j = 2, \ldots, n_{i-1}$
is contractible.
\end{enumerate}    
In particular, we obtain a contraction $X \to X'$
onto a $\KK^*$-surface $X'$ given by defining data
$(A,P')$ such that
\begin{enumerate}
\item[(iv)]
in the case that $X$ has of two elliptic fixed points, 
we have $n_i' \le 2$ for $i=0, \ldots, r$,
\item[(v)]
in the case that $X$ has a parabolic fixed point curve,
we have $n_i' = 1$ for $i=0, \ldots, r$.
\end{enumerate}   
\end{remark}

\begin{proposition}
\label{prop:contractcolchar}
Let $X = X(A,P)$ be projective, $v$ a column of $P$
and $D \subseteq X$ the corresponding prime divisor.
Then the following statements are equivalent.
\begin{enumerate}
\item
The column $v$ is contractible.
\item
The curve $D \subseteq X$ is contractible.
\item  
We have $D^2 < 0$.
\item
The divisor $D$ is not movable.
\end{enumerate}
\end{proposition}

\begin{proof}
If (i) holds, then we contract $D$ by means of
Construction \ref{constr:contract}.
The implications from (ii) to (iii)
and from (iii) to (iv) are standard surface
geometry.
If $D$ is not movable, then, in $\Cl_\QQ(X)$,
the ray through $[D]$ intersects the
cone generated by the remaining $w_{ij}$ and $w^\pm$
in the origin.
By~\cite[Lemma~2.2.3.2]{ArDeHaLa}, the column $v$
lies in the interior of the cone
generated by the remaining columns of $P$
and thus is contractible.
\end{proof}

\begin{proposition}
\label{prop:kstarcontractdecomp}
Any contraction $X \to Y$ of rational
projective $\KK^*$-surfaces decomposes as 
$X \cong X_0 \to \ldots \to X_q \cong Y$ 
with $X_i \to X_{i+1}$ as
in Construction~\ref{constr:contract}.
\end{proposition}

\begin{proof}
We may assume $X = X_0$ with $X_0$ arising from
the defining data $(A,P_0)$.
As observed in Proposition~\ref{prop:exccurveskstar},
the exceptional divisors
$E_1, \ldots, E_q \subseteq X$
are taken from the $D_X^\pm$ and the $D_X^{ij}$ that
contain a hyperbolic fixed point.
In particular, $E_1$ corresponds to a column of $P_0$
and Proposition~\ref{prop:contractcolchar}
provides $\psi \colon X_0 \to X_1$, as in
Construction~\ref{constr:contract}, contracting $E_1$.
Now observe that none of
$\psi_*(E_2), \ldots, \psi_*(E_q)$
are movable and hence by Proposition~\ref{prop:contractcolchar},
they are all contractible.
Iterating this consideration, we arrive at a
sequence $X_0 \to \ldots \to X_q$,
contracting $E_1, \ldots, E_q$ stepwise.
The remaining task is to show that $X \to Y$ factors
via an isomorphism through $X \to X_q$.
By construction, we obtain such a factorization
apart from the respective image points of
$E_1, \ldots, E_q$ in $X_q$ and $Y$.
Being an isomorphism up to codimension two,
$X_q \dasharrow Y$ lifts to the total
coordinate spaces and descends again to an
isomorphism of the surfaces $X_q$ and $Y$.
\end{proof}

We show that via contractions one does
not leave the class of log del Pezzo $\KK^*$-surfaces.
Moreover, we study their effect on the invariants
$\skr{d}^\pm$ defined in Construction \ref{constr:antican}
and Remark \ref{rem:generalmandzeta}.

\begin{proposition}
\label{prop:contr2dpm}
Let $X$ be a log del Pezzo $\KK^*$-surface
and $X \to Y$ a contraction of surfaces.
Then $Y$ is a log del Pezzo $\KK^*$-surface
and, according to the constellations of source
and sink, we have:
\begin{enumerate}
\item
For $x^+ \in X$ and $y^+ \in Y$,
we have $\skr{d}_X^+ \ge \skr{d}_Y^+$.
\item
For $D_X^+ \subseteq X$ and $y^+ \in Y$,
we have  $1 > \skr{d}_X^+ \ge \skr{d}_Y^+$.
\item
For $x^- \in X$ and $y^- \in Y$,
we have $\skr{d}_Y^- \ge \skr{d}_X^-$.
\item
For $D_X^- \subseteq X$ and $y^- \in Y$,
we have $\skr{d}_Y^- \ge \skr{d}_X^- > -1$.
\end{enumerate}
\end{proposition}

\begin{proof}
Remark \ref{rem:contractprops} and Proposition
\ref{prop:delpcontractstodelp} tell us that $Y$ is a
del Pezzo $\KK^*$-surface such that $X \to Y$ is
equivariant. By Proposition \ref{prop:kstarcontractdecomp}
it suffices to prove the assertion for $X$ and $Y$ arising
from defining data $(A,P_X)$ and $(A,P_Y)$ with $P_X$,~$P_Y$
slope-ordered and $P_Y$ obtained by erasing a
column $v$ from $P_X$. The task is to show that elliptic
fixed points of $Y$ are at most log terminal singularities.

Consider the case that we have elliptic fixed
points $x^+ \in X$ and $y^+ \in Y$. Then only for
$v = v_{i1}$ there is something to show. As $v$ is
erasable, we have $n_i \ge 2$. The tuples of exponents
associated with $x^+$ and $y^+$ differ only in the $i$-th
place and are given as
\[
(l_{01}, \ldots, l_{i1}, \ldots, l_{r1}),
\qquad\qquad
(l_{01}, \ldots, l_{i2}, \ldots, l_{r1}).
\]
For the first one, we have $\ell_{X}^{+} > 0$ by log
terminality. For the second one, we have to show this
property. For $l_{i2} \le l_{i1}$ this is obvious. So let
$l_{i2} > l_{i1}$. First observe 
\[
\skr{m}_Y^+ \ = \ \skr{m}_X^+ + m_{i2} - m_{i1},
\qquad\qquad
\ell_{Y}^{+} \ = \ \ell_{X}^{+} + \frac{1}{l_{i2}} -  \frac{1}{l_{i1}}.
\]
Due to slope-orderedness of $P_X$ and $m_Y^+ > 0$,
we obtain $0 < m_{i1} - m_{i2} < m_X^+$.
Since $X$ is del Pezzo, we have the positive
intersection number 
\[
0
\ < \
-\mathcal{K}_{X}^0 \cdot D_{X}^{i1}
\ = \
\frac{(m_{i1}-m_{i2})\ell_{X}^{+} - \left(\frac{1}{l_{i1}} - \frac{1}{l_{i2}}\right)m_X^+}{l_{i1}(m_{i1}-m_{i2})},
\]
where we use Summary~\ref{sum:intersections} for the
computation. Now $\ell_{Y}^{+} > 0$ is an immediate consequence
of the estimates
\[
\frac{1}{l_{i1}}-\frac{1}{l_{i2}} 
\ < \ 
\left( \frac{1}{l_{i1}}-\frac{1}{l_{i2}} \right)
\frac{m_X^{+}}{m_{i1}-m_{i2}}
\ < \
\ell_X^+
\ = \
\ell_Y^+ + \frac{1}{l_{i1}}-\frac{1}{l_{i2}}.
\]
Knowing that $Y$ is a log del Pezzo surface,
the number $d_Y^+$ is defined and we compare
it with $d_X^+$.
Consider the difference
\[
d_X^+ - d_Y^+
\ = \
\frac{m_X^+}{\ell_{X}^{+}}
- 
\frac{m_Y^+}{\ell_{Y}^{+}}
\ = \ 
\frac{\ell_Y^+m_X^+ - \ell_X^+m_Y^+}{\ell_X^+\ell_Y^+}.
\]
For $d_X^+ \ge d_Y^+$ this fraction has to be non-negative.
The denominator is obviously positive and for the
enumerator we compute
\[
\ell_Y^+m_X^+ - \ell_X^+m_Y^+
\ = \
(m_{i1} -m_{i2}) \ell_X^+
-
\left(\frac{1}{l_{i1}} -  \frac{1}{l_{i2}}\right) m_X^+ 
\]
which is also positive, as seen in the above computation
of $-\mathcal{K}_{X}^0 \cdot D_{X}^{i1}$.
Thus, we verified $d_X^+ \ge d_Y^+$.
The case of elliptic fixed points $x^- \in X$
and $y^- \in Y$ is transformed into the present one
via swapping the action and needs no extra treatment.

Now assume that we have a parabolic fixed point
curve $D_X^+ \subseteq X$ and an elliptic fixed point
$y^+ \in Y$. Then $v = v^+$ holds and we obtain
\[
0 \ < \ - (D_X^+)^2 \ = \ m_X^+,
\qquad
0 \ < \ -\mathcal{K}_X^0 \cdot D_X^+ \ = \ -m_X^+ + \ell_X^+,
\]
as $D_X^+$ is contracted and $X$ del Pezzo. Thus,
$\ell_Y^+ = \ell_X^+ > m_X^+ > 0$ and $d_Y^+ = d_X^+ < 1$.
Again, for the case $D_X^- \subseteq X$ and $y^- \in Y$, we
just swap the action.
\end{proof}

\begin{proposition}
\label{prop:dplusdminuscontract}
Consider the defining matrices $P_1$ and $P_2$ of log
del Pezzo $\KK^*$-surfaces $X_1$ and $X_2$, where
$P_2$ arises from $P_1$ by removing a column.
\begin{enumerate}
\item
We have $\mathcal{A}_{P_2} \subseteq \mathcal{A}_{P_1}$
for the associated anticanonical complexes.
\item
If $X_1$ is $1/k$-log canonical, then $X_2$ is
$1/k$-log canonical.
\end{enumerate}
\end{proposition}

\begin{proof}
The first assertion is a consequence of Proposition
\ref{prop:contr2dpm} and the description of the
anticanonical complex provided by Proposition
\ref{prop:ak-properties} (i) and (ii). The second
assertion follows from the first one and Proposition
\ref{prop:ak-properties} (v).
\end{proof}

The notion of combinatorial minimality was
introduced in~\cite{Ha}. The following version is
adapted to the  setting of rational projective
$\KK^*$-surfaces.

\begin{definition}
We call a normal, complete surface $X$
\emph{combinatorially minimal} if every contraction
$X \to Y$ is an isomorphism.
\end{definition}

\begin{example}
Up to isomorphy, the combinatorially minimal toric
surfaces arise from fans with a generator matrix of
the form $[v_1,v_2,v_3]$ or $[v_1,v_2,-v_1,-v_2]$.
\end{example}

\begin{remark}
\label{rem:combinchar}
Consider a projective $\KK^*$-surface $X = X(A,P)$.
Then, according to Proposition~\ref{prop:contractcolchar},
the following statements are equivalent.
\begin{enumerate}
\item
The surface $X$ is combinatorially minimal.
\item
The matrix $P$ has no contractible column.
\item
We have $\Eff(X) = \Mov(X)$.  
\item
Each of the $D_X^{ij}$, $D_X^+$, $D_X^-$ has
non-negative self intersection.
\end{enumerate}

\end{remark}


\begin{remark}
\label{rem:kstarcontract}
For every rational, projective $\KK^*$-surface $X$,
there is a sequence $X \cong X_0 \to \ldots \to X_q$
with contractions as in Construction
\ref{constr:contract} such that $X_q$ is
combinatorially minimal. The length of such a
sequence is bounded by the Picard number, i.e. we
have $q < \rho(X)$.
\end{remark}

\begin{proposition}
\label{prop:combminprops}
Let $X = X(A,P)$ be non-toric, projective, and
combinatorially minimal. Then the following
statements hold.
\begin{enumerate}
\item
We have $1 \le \rho(X) \le 2$ for the Picard
number $\rho(X)$.
\item
If $X$ has two elliptic fixed points,
then $1 \le n_i \le 2$ for $i=0, \ldots, r$
and $n_i=2$ happens at most twice.
\item
If $X$ has a parabolic fixed point curve,
then $n_i = 1$ for $i=0, \ldots, r$.
\end{enumerate}
\end{proposition}

\begin{proof}
Due to Remark \ref{rem:obviouscontractible} combinatorial
minimality implies $n_i \le 2$ for $i=0, \ldots, r$ in
case of two elliptic fixed points and $n_i = 1$ for
$i=0, \ldots, r$ if $X$ admits a parabolic fixed point
curve. In particular, the third assertion holds.

We prove (i). Remark \ref{rem:combinchar}
gives $Q(\gamma) = \Eff(X) = \Mov(X)$. Thus, on
each extremal ray of the effective cone, we find at
least two of the generator degrees
\[
w_{ij} \ = \ \deg(T_{ij}),
\qquad
w_{k} \ = \ \deg(S_{k}).
\]
Assume $\rho(X) \ge 3$. Then $\dim(\Eff(X)) \ge 3$
and we find an extremal ray
$\varrho \preccurlyeq \Eff(X)$ hosting two
generator degrees $w_{i_1j_1}$ and $w_{i_2j_2}$.
We claim $i_1 \ne i_2$. Otherwise $n_i \le 2$ for
$i := i_1 = i_2$ implies that $T_{ij_1}$ and
$T_{ij_2}$ are the variables of a monomial $h$ 
of a defining relation $g_\iota$. Consequently,
\[
\mu
\ = \
\deg(g_\iota)
\ = \
\deg(h)
\ = \
\deg\left(T_{ij_1}^{l_{ij_1}}T_{ij_2}^{l_{ij_2}}\right)
\ = \
l_{ij_1}w_{ij_1} + l_{ij_2}w_{ij_2}
\ \in \
\varrho
\]
with the exponents $l_{ij_1}$ and $l_{ij_2}$ of $h$.
But then all generator degrees except the $w_k$
are located on $\varrho$. This is a contradiction
to $m \le 2$ and $\Eff(X)$ being of dimension at
least three. Thus, $i_1 \ne i_2$ and we find
monomials
\[
T_{i_1j_1}^{l_{i_1j_1}}T_{i_3j_3}^{l_{i_3j_3}},
\qquad
T_{i_2j_2}^{l_{i_2j_2}}T_{i_4j_4}^{l_{i_4j_4}}
\]
in the defining relations. Observe that
$w_{i_1j_1}$, $w_{i_3j_3}$ as well as
$w_{i_2j_2}$, $w_{i_4j_4}$ generate two-dimensional
cones, both containing $\mu \in \Mov(X)^\circ$ in
their relative interior. Hence the same holds for
$\eta = \cone(w_{i_1j_1},w_{i_4j_4})$.
Moreover, the point $z \in \bar X$ with 
\[
z_{i_1j_1} \ = \ z_{i_4j_4} \ = \ 1
\]
and all other coordinates equal to zero belongs to
$\hat X$, see \cite[Constr. 3.3.1.3]{ArDeHaLa}.
This contradicts Remark \ref{rem:coxcoord}.
Thus, we verified the first assertion.

For the second assertion, note that we
have $\dim(X) + \rho(X) \le 4$ and $m = 0$. Thus,
the claim follows from 
\[
n
\ = \
n_0 + \ldots + n_r
\ = \
\dim(X) + \rho(X) + r-1
\ \le \
r + 3.
\]
\end{proof}


\section{The combinatorially minimal case}
\label{sec:combmin}

In this section, we classify the non-toric
combinatorially minimal $1/k$-log canonical del
Pezzo $\KK^*$-surfaces. The classification process runs
entirely in terms of the defining matrices $P$ from
Construction \ref{constr:kstarsurf}. Theorem
\ref{thm:surfclass1} provides bounds for the entries of
the defining matrices for arbitrary $k$ and in Theorem
\ref{thm:combminclass} the concrete classification for
$k=1,2,3$ is presented.

We denote by $c(k)$ the maximum volume of
almost $k$-hollow lattice simplices. Due to
Corollary \ref{cor:globalbound} we always have 
\[
c(k) \ \le \ \pi R(k)^2 \ = \  2\pi k^4 (2k^2+2k+1).
\]

\goodbreak

\begin{theorem}
\label{thm:surfclass1}
Every non-toric combinatorially minimal
$1/k$-log canonical del Pezzo
$\KK^*$-surface is isomorphic to an $X(A,P)$ with
$P$ from the following list.

\[
\begin{array}{llll}
{\rm (i)}
&
{\scriptsize 
\left[
\begin{array}{cccc}
-l_{01} & l_{11} & 0 & 0 
\\
-l_{01} & 0 & l_{21} & 0 
\\
d_{01} & d_{11} & d_{21} & 1
\end{array}
\right],
}
&
{\scriptsize
\begin{array}{l}
\scriptstyle{2 \le l_{01} \le \max(2k^{2},5)},
\\
\scriptstyle{2 \le l_{11} \le \max(4k^{2},5)},
\\
\scriptstyle{2 \le l_{21} \le 5},
\end{array}
}
&
{\scriptsize
\begin{array}{l}
\scriptstyle{-k-2l_{01}+1 \le d_{01} \le -1},
\\
\scriptstyle{1 \le d_{11} \le l_{11}-1},
\\
\scriptstyle{1 \le d_{21} \le l_{21}-1}.
\end{array}
}
\\[2em]
{\rm (ii)}
&
{\scriptsize   \setlength{\arraycolsep}{4.5pt}
\left[
\begin{array}{cccc} 
-l_{01} & -l_{02} & l_{11} & 0
\\
-l_{01} & -l_{02} & 0 & l_{21}
\\
d_{01} & d_{02} & d_{11} & d_{21}
\end{array}
\right],
}
&
{\scriptsize
\begin{array}{l}
\scriptstyle{1 \le l_{01} \le 2k^2},
\\
\scriptstyle{1 \le l_{02} \le 2k^2},
\\
\scriptstyle{2 \le l_{11} \le \max(6,4k^2,c(k))},
\\
\scriptstyle{2 \le  l_{21} \le \max(6,4k^2,c(k))},  
\end{array}
}
&
{\scriptsize
\begin{array}{l}
\scriptstyle{-2l_{01}+1 \le d_{01} \le k-1},
\\
\scriptstyle{-2l_{02}-k+1  \le d_{02} \le -1},
\\
\scriptstyle{1 \le d_{11} \le l_{11}-1},
\\
\scriptstyle{1 \le d_{21} \le l_{21}}.
\end{array}
}
\\[2em]
{\rm (iii)}
&
{\scriptsize   \setlength{\arraycolsep}{4.5pt}
\left[
\begin{array}{ccccc} 
-1 & -1 & l_{11} & 0 & 0
\\
-1 & -1 & 0 & l_{21} & 0
\\
-1 & -1 & 0 & 0 & 2
\\
d_{01} & d_{02} & d_{11} & d_{21} & 1
\end{array}
\right],
}
&
{\scriptsize
\begin{array}{l}
\scriptstyle{2 \le l_{11} \le \max(2k^{2},5)},
\\
\scriptstyle{2 \le l_{21} \le 3},
\\
\end{array}
}
&  
{\scriptsize
\begin{array}{l}
-\scriptstyle{2 \le d_{01} < \frac{k}{2}},
\\
\scriptstyle{-\frac{k}{2}-3 < d_{02} \le -1},
\\
\scriptstyle{1 \le d_{11} \le l_{11}-1},
\\
\scriptstyle{1 \le d_{21} \le l_{21}-1}.
\end{array}
}
\\[2em]
{\rm (iv)}
&
{\scriptsize   \setlength{\arraycolsep}{3.5pt}
\left[
\begin{array}{ccccc}
-l_{01} & -l_{02} & l_{11} & l_{12} & 0
\\
-l_{01} & -l_{02} & 0 & 0 & l_{21}
\\
d_{01} & d_{02} & d_{11} & d_{12} & d_{21}
\end{array}
\right],
}
&
{\scriptsize
\begin{array}{l}
\scriptstyle{1 \le l_{01} \le k^3+3k^2},
\\
\scriptstyle{1 \le l_{02} \le \ 6k^{2}},
\\
\scriptstyle{1 \le l_{11} \le 6k^{2}},
\\
\scriptstyle{1 \le l_{12} \le \ 6k^{2}},
\\
\scriptstyle{2 \le l_{21} \le \ 6k^{2}},
\end{array}
}
&
{\scriptsize
\begin{array}{l}
\scriptstyle{-2(k+3) k^{2} \le d_{01} \le k^{4}},
\\
\scriptstyle{-25k^{5}-15k^{2} \le d_{02} \le 5k^{6}+15k^{3}},
\\
\scriptstyle{0 \le d_{11} < l_{11}},
\\
-\scriptstyle{5k^{6}-3k \le d_{12} \le  125k^{7}+75k^{4}},
\\
\scriptstyle{1 \le d_{21} \le l_{21}-1}.
\end{array}
}
\\[2em]
{\rm (v)}
&
{\scriptsize  \setlength{\arraycolsep}{2.7pt}
\left[
\begin{array}{cccccc}
-1 & -l_{02} & l_{11} & l_{12} & 0 & 0
\\
-1 & -l_{02} & 0 & 0 & l_{21} & 0
\\
-1 & -l_{02} & 0 & 0 & 0 & l_{31}
\\
d_{01} & d_{02} & d_{11} & d_{12}  & d_{21} & d_{31}
\end{array}
\right],
}
&
{\scriptsize
\begin{array}{l}
\scriptstyle{1 \le l_{02} \le 5},
\\
\scriptstyle{1 \le l_{11} \le  2k^{2}},
\\
\scriptstyle{1 \le l_{12} \le \max(2k^{2},5)},
\\
\scriptstyle{2 \le l_{21} \le \max(2k^{2},5)},
\\
\scriptstyle{2 \le l_{31} \le l_{21}},
\end{array}
}
&
{\scriptsize
\begin{array}{l}
\scriptstyle{-2 \le d_{01} < k},
\\
\scriptstyle{-k-3 \le d_{02} \le k-2},
\\
\scriptstyle{0 \le d_{11} \le l_{11}-1},
\\
\scriptstyle{-20k^{3} \le d_{12} \le l_{12}-1},
\\
\scriptstyle{1 \le d_{21} \le l_{21}-1},
\\
\scriptstyle{1 \le d_{31} \le l_{31}-1}.
\end{array}
}
\end{array}  
\]

\medskip

\end{theorem}

In the proof, subsequent Lemma helps us to
bound the entries $l_{ij}$ of the defining matrices~$P$.
The polygon $\mathcal{B}$ considered there will
be part of an arm of the corresponding anticanonical
complex $\mathcal{A}_{P}$ and $a/b$ a
lower bound for $\mathrm{min}(\skr{d}^{+},-\skr{d}^{-})$.

\begin{lemma}
\label{lem:mink2}
Fix $k \in \ZZ_{\ge 1}$. Given $a,b,l,d \in \mathbb{Z}_{\ge 1}$,
consider the polygon $\mathcal{B}$ with the vertices
$(0, \pm \frac{a}{b})$ and $(l,d)$ in $\mathbb{R}^{2}$. 
If $\mathcal{B}$ is almost $k$-hollow, then $l \le \frac{2b}{a}k^2$.
\end{lemma}

\begin{proof}
The set $\mathcal{C} := k^{-1}(-\mathcal{B} \cup \mathcal{B})$
is centrally symmetric and convex having the origin as its
only interior lattice point. The volume of $\mathcal{C}$
equals $2alb^{-1}k^{-2}$ and is due to Minkowski's Theorem
bounded by $4$.
\end{proof}

\begin{proof}[Proof of Theorem \ref{thm:surfclass1}]
Let $X$ be a non-toric combinatorially minimal
$1/k$-log canonical del Pezzo $\KK^*$-surface.
We may assume $X = X(A,P)$, where $P$ is adjusted, see
Definition \ref{adjustedP}. By Proposition
\ref{prop:combminprops}, combinatorial minimality
forces the format $(n_0,\ldots,n_r;m)$ to be one of
\[
(1,\ldots,1;1),
\qquad
(2,1,\ldots,1;0),
\qquad
(2,2,1,\ldots,1;0).
\]
Moreover, $X$ is log terminal and thus Summary
\ref{cor:logtermchar} leaves the following six
possibilities for $(n_0,\ldots,n_r;m)$:
\[
(1,1,1;1),\
(2,1,1;0),\
(2,1,1,1;0),\
(2,2,1;0),\
(2,2,1,1;0),\
(2,2,1,1,1;0).
\]

We exemplarily treat the format $(2,1,1;0)$.
In this situation, the defining matrix~$P$ is of the shape
\[
P
\ = \
\left[
\begin{array}{cccc}
-l_{01} & -l_{02} & l_{11} & 0
\\
-l_{01} & -l_{02} & 0 & l_{21}
\\
d_{01} & d_{02} & d_{11} & d_{21}
\end{array}
\right].
\]
As $P$ is adjusted, we have $1 \le d_{i1} < l_{i1}$
for $i = 1,2$. Consider the vertices
$\tilde v^{+} = \skr{d}^{+}v^{+}$ and
$\tilde v^{-} = \skr{d}^{-}v^{-}$ of the
anticanonical complex $\mathcal{A}_{P}$, where
\[
\skr{d}^{+}
\ = \
\frac{m_{01}+m_{11}+m_{21}}{l_{01}^{-1}+l_{11}^{-1}+l_{21}^{-1}-1},
\qquad
\skr{d}^{-}
\ = \
\frac{m_{02}+m_{11}+m_{21}}{l_{02}^{-1}+l_{11}^{-1}+l_{21}^{-1}-1}.
\]
Again, according to Remark \ref{cor:logtermchar}, log
terminality of $X$ ensures that the denominators are
both positive. Using $l_{11},l_{21} \ge 2$, we obtain
\[
0 \ < \ l_{01}^{-1}+l_{11}^{-1}+l_{21}^{-1}-1 \ \le \ l_{01}^{-1},
\qquad
0 \ < \ l_{02}^{-1}+l_{11}^{-1}+l_{21}^{-1}-1 \ \le \ l_{02}^{-1}.
\]
Observe that the expressions $\skr{d}^{+}$, $m_{01}$ and
$\skr{d}^{-}$, $m_{02}$ are strictly increasing in $d_{01}$
and $d_{02}$, respectively. Moreover, for any
$\alpha \in \mathbb{Q}$, we have
\[
\skr{d}^{+} = \alpha
\ \iff \ 
m_{01}
=
\left(l_{01}^{-1}+l_{11}^{-1}+l_{21}^{-1}-1\right) \alpha -m_{11}-m_{21},
\]
\[
\skr{d}^{-} = \alpha
\ \iff \ 
m_{02}
=
\left(l_{02}^{-1}+l_{11}^{-1}+l_{21}^{-1}-1\right) \alpha -m_{11}-m_{21}.
\]
Since $0 \in \mathcal{A}_P^{\circ}$ and $X$ is
$1/k$-log canonical, we have $0 < \skr{d}^{+} \le k$
and $-k \le \skr{d}^{-} < 0$, see Proposition
\ref{prop:ak-properties}. For $\alpha = 0, k, \pm 1$, the
above considerations yield
\[
\begin{array}{r}
-2 l_{01} < d_{01} < k,
\\[1ex]
-k -2l_{02} < d_{02} < 0,
\end{array}
\qquad
\begin{array}{l}
\skr{d}^{+} <  1 \ \Rightarrow \ d_{01} < 1,
\\[1ex]
\skr{d}^{-} > -1 \ \Rightarrow \ d_{02} > -1-2l_{02}.
\end{array}
\]
We bound $l_{11}$ and $l_{21}$ in terms of 
$\alpha = \min(\skr{d}^{+},-\skr{d}^{-})$.
For $i = 1,2$, consider the polygons
\[
\mathcal{B}_i
\ := \
\mathrm{conv}
\left(
\left(
0,\skr{d}^{+})
\right),
\left(
l_{i1},d_{i1}
\right),
\left(
0,\skr{d}^{-}
\right)
\right)
\ \subseteq \
\RR^2.
\]
Since $\mathcal{A}_P$ is almost $k$-hollow, both
$\mathcal{B}_i$ are $k$-hollow. Thus, Lemma
\ref{lem:mink2} applies and we obtain 
\[
l_{i1}
\ \le \ 
\frac{2}{\alpha}k^2.
\]

One now proceeds by going through all the possible
constellations of the platonic triples
$(l_{01},l_{11},l_{21})$ and $(l_{02},l_{11},l_{21})$.
We restrict ourselves to the most lively case of
$l_{01}=l_{02}=1$. In this setting, there are no a
priori constraints on $l_{11}$, $l_{21}$ and our task
is to find suitable bounds. First we specify
\[
\skr{d}^{+} = \frac{d_{01}+m_{11}+m_{21}}{l_{11}^{-1}+l_{21}^{-1}},
\quad
\skr{d}^{-} = \frac{d_{02}+m_{11}+m_{21}}{l_{11}^{-1}+l_{21}^{-1}},
\quad
\skr{d}^{+}-\skr{d}^{-} = \frac{d_{01}-d_{02}}{l_{11}^{-1}+l_{21}^{-1}}.
\]
In the case $\skr{d}^{+} \ge 1/2$ and $\skr{d}^{-} \le -1/2$,
we have $l_{i1} \le 4 k^2$ for $ i=1,2$ as just observed.
Moreover, with $\beta := \skr{d}^{+}- \skr{d}^{-}$, we
obtain 
\[
1
\ \le \
d_{01}-d_{02}
\ = \
\beta
\left(\frac{1}{l_{11}}+\frac{1}{l_{21}} \right).
\]
Then $\beta \le 3/2$, implies $l_{11},l_{21} \le 6$ and
thus we are left with $\beta > 3/2$. Suppose
$\skr{d}^{+} < 1/2$. Then $d_{01} = -1$ and
$\skr{d}^{+} > 0$ forces $m_{11}+m_{21} > 1$. Moreover,
\[
\left(\frac{1}{l_{11}}+\frac{1}{l_{21}} \right)
\ \ge \
\frac{1}{\beta} 
\ > \
\frac{1}{k+\frac{1}{2}}
\]
due to $\beta = \skr{d}^{+}-\skr{d}^{-} <  k + \frac{1}{2}$.
We conclude $\min(l_{11},l_{21})  \le 2k$. We may assume
$l_{11} \le 2k$. Then we can estimate 
\[
\frac{1}{2k} \  \le \ m_{11} \ \le \  \frac{2k-1}{2k},
\qquad\qquad
\frac{1}{2k} \  < \ m_{21} \ \le \ 1.
\]
The idea is to bound $l_{21}$ via the volume of a
suitable almost $k$-hollow lattice simplex. Consider the
prolongated second arm of the anticanonical complex:
\begin{center}

\begin{tikzpicture}[scale=.5]
\coordinate (o) at (0,0);
\coordinate (v+) at (0,1);
\coordinate (x) at (-3,-1.5);
\coordinate (ld) at (3,2.25);
\coordinate (v-) at (0,-.75);
\coordinate (d-) at (0,-1.8);
\coordinate (ah+) at (4.5,0);
\coordinate (ah-) at (-3,0);
\coordinate (av+) at (0,3.5);
\coordinate (av-) at (0,-2);
\fill[gray!50,opacity=0.3] (x)--(v-)--(ld)--(x);
\draw[thin] (ah+)--(ah-);
\draw[thin] (av+)--(av-);
\draw[] (d-)--(ld);
\draw[dotted] (x)--(v-)--(ld)--(x);
\draw[] (ld)--(x);
\path[fill, color=black] (v-) circle (.4ex);
\path[fill, color=black] (o) circle (.4ex); 
\path[fill, color=black] (ld) circle (.4ex) node[right]{$\scriptscriptstyle (l_{21}, d_{21})$};
\path[fill, color=black] (x) circle (.4ex)  node[below]{$\scriptscriptstyle (-l_{11},d_{11}-l_{11})$};
\node at (.6,.4) {$\scriptstyle \mathcal{C}$};
\end{tikzpicture}

\end{center}
One directly checks that $(-l_{11},d_{11}-l_{11})$
lies on the bounding line $L$ through $(0,\skr{d}^{+})$
and $(l_{21},d_{21})$.
Thus, we can indeed define the polygon $\mathcal{C}$
as indicated above by
\[
\mathcal{C}
\ := \
\mathrm{conv}
\left(
\left(
0,-1
\right),
\left(
l_{21}, d_{21}
\right),
\left(
-l_{11},d_{11}-l_{11}
\right)
\right).
\]
We check that  $\mathcal{C}$ is almost $k$-hollow.
As a subset of  $\mathcal{A}_P$, the r.h.s. part
of $\mathcal{C}^\circ$ contains no $k$-fold lattice
points except $(0,0)$. Concerning the l.h.s. part,
we need
\[
(-k,0) \ \not\in  \ \mathcal{C},
\qquad\qquad
(-k,-k) \ \not\in  \ \mathcal{C}.
\]
Indeed, due to $\skr{d}^{+} < 1$ and $l_{11} \le 2k$,
this gives $\mathcal{C}^\circ \cap k \ZZ^2 = \{(0,0)\}$.
The slopes of $L$ and the bounding line $G$ through
$(-l_{11},d_{11}-l_{11})$ and $(0,-1)$ satisfy
\[
m_L \ > \ m_{11}\ \ge \ \frac{1}{2k},
\qquad\qquad
m_G \ = \ \frac{l_{11}-d_{11}-1}{l_{11}} \ \le \ 1 - \frac{2}{k}.
\]
Consequently, $(-k,0)$ lies above $L$ and $(-k,-k)$ lies
below $G$. Altogether, $\mathcal{C}$ is almost $k$-hollow
and we have $l_{21} \le 2 \vol( \mathcal{C})$. Thus the
case $\skr{d}^{+} < 1/2$ is settled. If $\skr{d}^{-} > -1/2$,
then we swap source and sink by multiplying the last row by
$-1$. After re-adjusting, we are again in the case
$\skr{d}^{+} < 1/2$.
\end{proof}

For any given $k$, Theorem \ref{thm:surfclass1} provides a
finite list of matrices that comprises in particular
defining matrices~$P$ for all the non-toric combinatorially
minimal $1/k$-log canonical del Pezzo $\KK^*$-surfaces.
In order to obtain a classification, we perform
computationally the following steps:
\begin{itemize}
\item
remove all defining matrices~$P$ from the list
that don't yield a $1/k$-log canonical del
Pezzo $\KK^*$-surface,
\item
remove all redundancies such that any
combinatorially minimal $1/k$-log
canonical del Pezzo $\KK^*$-surface is listed
precisely once up to isomorphy.
\end{itemize}
For the first step one uses the del Pezzo criterion
from Proposition~\ref{prop:kleimanample} and
the characterization of $1/k$-log canonicity
via the anticanonical complex from
Theorem~\ref{prop:ak-properties} allow efficient
implementation and in the second step one works with
the normal form from Definition~\ref{normalformdef}.
Altogether we arrive at the following results.

\begin{theorem}
\label{thm:combminclass}
We obtain the following statements on non-toric
combinatorially minimal $\varepsilon$-log canonical del Pezzo
$\KK^*$-surfaces.
\begin{itemize}
\item[$\varepsilon=1$:]
There are exactly $13$ sporadic and $2$ one-parameter families
of non-toric combinatorially minimal canonical del Pezzo
$\KK^*$-surfaces. 
\item[$\varepsilon=\frac{1}{2}$:]
There are exactly $62$ sporadic and $5$ one-parameter families
of non-toric combinatorially minimal $1/2$-log canonical
del Pezzo $\KK^*$-surfaces. 
\item[$\varepsilon=\frac{1}{3}$:]
There are exactly $318$ sporadic and $14$ one-parameter families
of non-toric combinatorially minimal $1/3$-log canonical
del Pezzo $\KK^*$-surfaces. 
\end{itemize}
\end{theorem}


\section{Classifying $1/k$-log canonical del Pezzo $\KK^*$-surfaces}
\label{sec:classall}

We are ready for the classification of non-toric
$1/k$-log canonical del Pezzo $\KK^*$-surfaces.
Theorem~\ref{thm:kstardelpezzo} shows basic data
of the classsification for $k=1,2,3$. The full
list of defining data will be made available
in~\cite{HaHaSp}. The idea is to build up the
defining $P$-matrices from the almost $k$-hollow
polygons classified in Theorem~\ref{thm:polygonclass}
on the one hand and the
defining data of the non-toric combinatorially
minimal $1/k$-log canonical del Pezzo
$\KK^*$-surfaces classified in
Theorem~\ref{thm:combminclass} on the other.

\begin{construction}
Fix $k \in \ZZ_{\ge 1}$.
With any almost $k$-hollow LDP-polygon
$\mathcal{A} \subseteq \RR^2$, we associate a
defining matrix $P_{\mathcal{A}}$ with $r=1$,
having the vertices of $\mathcal{A}$ as its columns:
\[
P_{\mathcal{A}} \ := \ [v_{01}, \ldots, v_{0n_0},v_{11}, \ldots, v_{1n_1},v^\pm],
\]
where $v_{0j}=(-l_{0j},d_{0j})$ and
$v_{1j}=(l_{1j},d_{1j})$ and, if present,
$v^\pm = (0, \pm 1)$. Suitable numbering of the
columns ensures that $P_\mathcal{A}$ is
slope-ordered. Conversely, setting
\[
\mathcal{A}_{\mathcal{P}}
\ := \
\conv(v; \ v \text{ column of } P)
\]
for every slope-ordered defining matrix $P$ with
$r=1$ of a toric $1/k$-log canonical del
Pezzo $\KK^*$-surface gives an almost $k$-hollow
LDP-polygon.
\end{construction}

\begin{definition}
\label{def:ldp-sets}
Fix $k \in \ZZ_{\ge 1}$ and consider the lattice vectors
$v^{+} = (0,1)$ and $v^{-} = (0,-1)$.
\begin{enumerate}
\item
We denote by $\mathfrak{A}_k$ the set of all almost $k$-hollow
LDP-polygons $\mathcal{A}$ such that 
$\conv(\mathcal{A} \cup \{v^{+}\})$
or $\conv(\mathcal{A} \cup \{v^{-}\})$
is almost $k$-hollow.
\item
We write $\mathfrak{P}_k = \{P_{\mathcal{A}}; \ \mathcal{A} \in \mathfrak{A}_k\}$
for the set of defining matrices associated with the polygons
from $\mathfrak{A}_k$.
\end{enumerate}

\end{definition}

\begin{remark}

The set  $\mathfrak{P}_k$ is invariant under admissible operations
of types (i), (ii), (iii), (iv) and (v).

\end{remark}

\begin{proposition}

Let $X_1 \to X_2$ be a contraction of non-toric
$1/k$-log canonical del Pezzo $\KK^*$-surfaces as in
Construction \ref{constr:contract}, where $X_1$ is non-toric
and $X_2$ is toric and combinatorially minimal. Then the
defining matrix $P_{2}$ stems from $\mathfrak{P}_{k}$.
\end{proposition}

\begin{proof}
Let $(A_i,P_i)$ be the defining data of $X_i$.
Then $P_1$ has three arms, at most six columns.
We may assume that a column of the third
arm of $P_1$ is contracted.
Then the arms of $P_2$ are the first two
of $P_1$ and Proposition \ref{lem:rviamzeta} (iv)
guarantees that $P_2$ belongs to $\mathfrak{P}_k$.
\end{proof}

\begin{construction}
\label{constr:startingclass}
Let $\mathcal{B}$ be an almost $k$-hollow LDP-polygon
with $q \le 5$ vertices. 
We obtain a finite set
$\mathfrak{A}(\mathcal{B}) \subseteq \mathfrak{A}_k$
of almost $k$-hollow LDP-polygons with at most $4$
vertices via the following procedure.
\begin{itemize}
\item
For each primitive lattice point $v \in \mathcal{B}$
fix unimodular matrices $M_{v}^{+}$, $M_{v}^{-}$ such
that 
\[
M_{v}^{+} \cdot v \ = \ v^{+},
\qquad
\mathrm{det}(M_{v}^{+}) \ = \ 1, 
\]
\[
M_{v}^{-} \cdot v \ = \ v^{-},
\qquad
\mathrm{det}(M_{v}^{-}) \ = \ -1.
\]
\item
Start with $\mathfrak{A}(\mathcal{B}) := \emptyset$.
With each primitive vector $v \in \mathcal{B}$
perform the following steps.
\begin{itemize}
\item
If $4 \le q \le 5$ and $v$ is a vertex of $\mathcal{B}$
but $-v$ is not, then consider the convex hull $\mathcal{B}_v$
over all vertices of $\mathcal{B}$ except $v$.
If $\mathcal{B}_v$ is an LDP-polygon, then 
add $M_{v}^+ \cdot \mathcal{B}_v$ and
$M_{v}^- \cdot \mathcal{B}_v$ to $\mathfrak{A}(\mathcal{B})$.
\item
If $3 \le q \le 4$ and neither $v$ nor $-v$
is  a vertex of $\mathcal{B}$,
then set $\mathcal{B}_v := \mathcal{B}$
and add
$M_{v}^+ \cdot \mathcal{B}_v$ and
$M_{v}^- \cdot \mathcal{B}_v$ to $\mathfrak{A}(\mathcal{B})$.
\end{itemize}
\end{itemize}
Moreover, given any (finite) set $\mathfrak{B}$
of almost $k$-hollow LDP-polygons with at most $5$ vertices,
we obtain a (finite) set of almost $k$-hollow LDP-polygons
with at most $4$ vertices by setting
\[
\mathfrak{A}(\mathfrak{B})
\ := \
\bigcup_{\mathcal{B} \in \mathfrak{B}} \mathfrak{A}(\mathcal{B}).
\]
\end{construction}

\begin{proposition}
\label{prop:toricrep}
Let the set $\mathfrak{B}$ represent all
almost $k$-hollow LDP-polygons with at
most $5$ vertices up to unimodular equivalence.
Then every matrix $P \in \mathfrak{P}_k$ with 
at most four columns arises via admissible
operations of types (i), (ii) and (iv)
from some matrix $P_{\mathcal{C}}$ with
$\mathcal{C} \in \mathfrak{A}(\mathfrak{B})$.
\end{proposition}

\begin{proof}
Given $P \in \mathfrak{P}_k$, consider the associated
LDP-polygon $\mathcal{A}_P \in \mathfrak{A}_k$.
By definition of $\mathfrak{A}_k$, one of the
following holds.
\[
v^{+} \ \in \ \mathcal{B}^+ \ := \ \conv(\mathcal{A}_P \cup \{v^{+}\}),
\qquad 
v^{-} \ \in \ \mathcal{B}^- \ := \ \conv(\mathcal{A}_P \cup \{v^{-}\}).
\]
In either case, by the choice of $\mathfrak{B}$, we find
a unimodular matrix $U$ with
$U \cdot \mathcal{B}^\pm \in \mathfrak{B}$.
Set $v := U \cdot v^\pm$.
Then, with $M_v^\pm$ and $\mathcal{B}^\pm_v$ from 
Construction \ref{constr:startingclass},
we have 
\[
\mathcal{C}^+ \ := \ M_v^{+} \cdot \mathcal{B}^+_v \ \in \ \mathfrak{A}(\mathfrak{B}),
\qquad\qquad
\mathcal{C}^- \ := \ M_v^{-} \cdot \mathcal{B}^-_v \ \in \ \mathfrak{A}(\mathfrak{B}).
\]
The matrix $P_{\mathcal{C}^\pm}$ associated with
$\mathcal{C}^\pm$ equals
$M_v^\pm \cdot U \cdot P$.
As $M_v^\pm \cdot U$ fixes $v^\pm$ up to sign,
$P$ arises from $P_{\mathcal{C}^\pm}$ via 
admissible operations of types (i), (ii) and (iv).
\end{proof}

The main result of \cite{Al} implies
that for given $\varepsilon > 0$ there are up
to deformation only finitely many
$\varepsilon$-log canonical surfaces.
The following provides an
effective version of this statement for the case
of non-toric $1/k$-log canonical del
Pezzo $\KK^*$-surfaces.

\begin{theorem}
\label{thm:kstar-effective-bounds}
Consider a non-toric $1/k$-log canonical del
Pezzo $\KK^*$-surface $X(A,P)$, where $P$ is irredundant,
slope-ordered and adapted to the source. Fix $\alpha > 0$
such that $d_Y^+ > \alpha$ and $d_Y^- < - \alpha$ for any
combinatorially minimal $1/k$-log canonical del
Pezzo $\KK^*$-surface $Y$ and any toric surface $Y$
arising from Proposition \ref{prop:toricrep}. Set
$\ell := 2 \alpha^{-1}k^2$.
\begin{enumerate}
\item
The number $r+1$ of arms of $P$ is bounded by $4k$
in the case of two elliptic fixed points and by
$2k+2$ in the case of only one elliptic fixed point.
\item
We have $n_i \le 2 \ell + 1$ for $i = 0, \ldots, r$.
Moreover,  the entries of $P$ are bounded by
$l_{ij} \le \ell$ and $-l_{ij}(4k+r-1) < d_{ij} < l_{ij}$.
\end{enumerate}  
\end{theorem}

\begin{proof}
Assertion (i) is clear by Proposition \ref{prop:armbounds}.
We show (ii). Proposition \ref{prop:dplusdminuscontract}
yields $\skr{d}^{+} > \alpha$ and $\skr{d}^{-} < -\alpha$
for any defining matrix $P$ of a $1/k$-log
canonical del Pezzo $\KK$-surface. Thus, Lemma
\ref{lem:mink2} provides the desired common bound for
all the $l_{ij}$. Now, by convexity of the $i$-th arm of
$\mathcal{A}_P$, there can be at most $2 \ell +1$ columns
in the $i$-th arm of $P$. That means that we have
$n_i \le 2 \ell +1$.

Finally, we have to bound the numbers $d_{ij}$.
From Theorem \ref{prop:ak-properties}, we infer
$\skr{d}^{+} \le k$ and $\skr{d}^{-} \ge -k$. As $P$ is
adapted to the source, we have $0 \le m_{i1} < 1$ for
$i = 1, \ldots, r$. Remark \ref{rem:obviouszetaprops}
says $\ell^+ \le 2$ and $\ell^- \le 2$. Altogether,
this yields 
\[
m_{01} \ \le k \ \ell^+ - \sum_{i=1}^r m_{i1} \ \le \ 2k,
\qquad
m_{0n_0} \ \ge \ - k \ell^- - \sum_{i=1}^r m_{in_i} \ \ge \ -2k -r.
\]
Thus, besides $d_{i1}, \ldots, d_{ir}$ also $d_{01}$
and $d_{0n_0}$ are bounded as claimed. Due to
slope-orderedness, the remaining $d_{ij}$ are bounded
once we have bounded $m_{in_i}$ from below for
$i = 1,\ldots, r$. This is seen as follows.
\[
m_{in_i}
\ \le \
- \ell^- - m_{0n_0} - \sum_{q \ne i,0} m_{qn_q}
\ \le \
-2k -2k -r +1
\ = \
-4k - r + 1.
\]
\end{proof}

\begin{algorithm}
\label{algo:classall}
Let $k \in \ZZ_{\ge 1}$.
The input is a finite set $\mathfrak{B}$ representing
up to unimodular equivalence all almost $k$-hollow
lattice polygons and a finite set $\mathfrak{M}$
containing up to isomorphism all defining matrices
$P$ of non-toric combinatorially minimal $1/k$-log
canonical del Pezzo $\KK^*$-surfaces.
\begin{itemize}
\item
Compute the set $\mathfrak{A}(\mathfrak{B})$
of defining matrices $P$ with $r=1$ from
Construction~\ref{constr:startingclass}.
\item
Compute the set $\mathfrak{S}_0$ of all irredundant
defining matrices $P$ with $r = 2$ and entries bounded
according to Theorem \ref{thm:kstar-effective-bounds}
that arise from $\mathfrak{A}(\mathfrak{B})$ via a
redundant extension followed by a proper one.
\item
Compute the union 
$\mathfrak{S}_1 := \mathfrak{S}_0 \cup \mathfrak{M}$,
bring the matrices of $\mathfrak{S}_1$ into normal form and,
this way, remove all doubled entries $\mathfrak{S}_1$.
\item
Compute the set $\mathfrak{S}$ of all irredundant
defining matrices $P$ with data bounded according
to Theorem \ref{thm:kstar-effective-bounds}
that arise from $\mathfrak{S}_1$ via a series of
redundant and proper extensions.
\item
Bring the matrices of $\mathfrak{S}$ into normal form and,
in this way, remove all entries of $\mathfrak{S}$ that
appear multiple times.
\end{itemize}
The output set $\mathfrak{S}$ contains precisely one
member for each equivalence class of defining matrices
$P$ of non-toric $1/k$-log canonical del Pezzo
$\KK^*$-surfaces.
\end{algorithm}

The classification algorithm turns out to be feasible
on a personal computer for the cases $k = 1,2$ and $3$.
As a result, we obtain the following.

\begin{theorem}
\label{thm:kstardelpezzo1}
The non-zero numbers of $d$-dimensional families of non-toric
$1/k$-log canonical
del Pezzo $\KK^*$-surfaces for $k = 1,2,3$ are given
as follows:
\begin{center}
\begin{tabular}{c|c|c|c|c|c|c|c|c|c|c}
\scriptsize
$d$
&
\scriptsize
$0$
&
\scriptsize
$1$
&
\scriptsize
$2$
&
\scriptsize
$3$
&
\scriptsize
$4$
&
\scriptsize
$5$
& 
\scriptsize
$6$
&
\scriptsize
$7$
&
\scriptsize
$8$
&
\scriptsize
$9$
\\
\hline
\scriptsize
$k=1$
&
\scriptsize
$30$
&
\scriptsize
$4$
&
\scriptsize
---
&
\scriptsize
---
&
---
&
---
&
---
&
---
&
---
&
---
\\
\hline
\scriptsize
$k=2$
&
\scriptsize
$998$
&
\scriptsize
$184$
&
\scriptsize
$40$
&
\scriptsize
$12$
&
\scriptsize
$2$
&
\scriptsize
$1$
&
\scriptsize
---
&
\scriptsize
---
&
---
&
---
\\
\hline
\scriptsize
$k=3$
&
\scriptsize
$65022$
&
\scriptsize
$12402$
&
\scriptsize
$3190$
&
\scriptsize
$917$
&
\scriptsize
$254$
&
\scriptsize
$64$
&
\scriptsize
$14$
&
\scriptsize
$6$
&
\scriptsize
$2$
&
\scriptsize
$1$
\end{tabular} 
\end{center}
\end{theorem}

Any log del Pezzo $\mathbb{K}^{\ast}$-surface of
Gorenstein index $\iota$ is $1/\iota$-log
canonical.
Thus, Theorem~\ref{thm:kstardelpezzo1} comprises
in particular the log del Pezzo $\mathbb{K}^{\ast}$-surfaces
of Gorenstein indices $\iota = 1,2,3$.
Picking out those, we obtain the following.

\begin{corollary}
\label{thm:kstardelpezzogorind}
The non-zero numbers of $d$-dimensional families of log
del Pezzo $\KK^*$-surfaces of Gorenstein index
$\iota = 1,2,3$ are given as follows:
\begin{center}
\begin{tabular}{c|c|c|c|c|c|c|c|c|c|c}
\scriptsize
$d$
&
\scriptsize
$0$
&
\scriptsize
$1$
&
\scriptsize
$2$
&
\scriptsize
$3$
&
\scriptsize
$4$
&
\scriptsize
$5$
& 
\scriptsize
$6$
&
\scriptsize
$7$
&
\scriptsize
$8$
&
\scriptsize
$9$
\\
\hline
\scriptsize
$\iota=1$
&
\scriptsize
$30$
&
\scriptsize
$4$
&
\scriptsize
---
&
\scriptsize
---
&
---
&
---
&
---
&
---
&
---
&
---
\\
\hline
\scriptsize
$\iota=2$
&
\scriptsize
$53$
&
\scriptsize
$17$
&
\scriptsize
$7$
&
\scriptsize
$3$
&
\scriptsize
$1$
&
\scriptsize
$1$
&
\scriptsize
---
&
\scriptsize
---
&
---
&
---
\\
\hline
\scriptsize
$\iota=3$
&
\scriptsize
$268$
&
\scriptsize
$123$
&
\scriptsize
$67$
&
\scriptsize
$18$
&
\scriptsize
$36$
&
\scriptsize
$18$
&
\scriptsize
$10$
&
\scriptsize
$5$
&
\scriptsize
$3$
&
\scriptsize
$1$
\end{tabular} 
\end{center}
\end{corollary}

\begin{example}
\label{ex:maxfam}
The families of maximal dimension listed
in Theorem~\ref{thm:kstardelpezzo1} are the
cases $k=1,2,3$ of the following series,
given by the defining matrices
\[
P_{k} \ = \ 
\begin{bmatrix}
-1 & -1 & 1 & 1 & \dots & 0 & 0\\
-1 & -1 & 0 & 0 & \ddots & \vdots & \vdots \\
\vdots & \vdots & \vdots & \vdots & & 0 & 0\\
-1 & -1 & 0 & 0 &  & 1 & 1 \\
2k & 2k-1 & 0 & -1 & \dots & 0 & -1
\end{bmatrix}
\ \in \ \mathrm{Mat}(4k,8k;\mathbb{Z}).
\]
According to Proposition~\ref{prop:kleimanample},
each of the associated $\mathbb{K}^*$-surfaces
$X_{k} = X(A,P_k)$ is del Pezzo.
Furthermore we can directly compute
\[
\skr{d}^{+}
=
\frac{\skr{m}^{+}}{\skr{l}^{+}}
=
\frac{2k}{4k-(4k-2)}
=
k,
\qquad
\skr{d}^{-}
=
\frac{\skr{m}^{-}}{\skr{l}^{-}}
=
\frac{2k-1-(4k-1)}{4k-(4k-2)}
=
-k.
\]
Thus, Theorem~\ref{prop:ak-properties}~(v) tells us that
$X_{k}$ is $1/k$-log canonical.
Moreover, $X_k$ is of Gorenstein index $k$; use
Proposition~\ref{prop:kstar-gorind-2}
and Corollary~\ref{cor:gorind}.
Finally, we obtain
\[
r
= 
4k-1,
\qquad
d_k
 = 
r-2
 = 
4k-3,
\qquad
\rho(X_k)
 = 
m+n-r-1
 = 
4k,
\]
where $d_k$ denotes the dimension of the family given by $P_k$;
see Remark~\ref{rem:families} and Proposition~\ref{prop:kstarsurfisQfact}.
In particular, the bound of Proposition~\ref{prop:armbounds}
is sharp.
\end{example}


Making use of the database~\cite{HaHaSp}, we can refine
the table from Corollary~\ref{introcor2} by providing
the numbers of families of $\varepsilon$-log canonical
del Pezzo surfaces of given Picard number $\varrho$ at
the jumping places. We first look at the toric and the
non-toric cases separately and then at both together.

\begin{corollary}
We obtain the following table for the numbers
$\nu_{\mathrm{toric}}(\alpha,\rho)$
of families of $\alpha^{-1}$-log canonical toric
del Pezzo surfaces of Picard number~$\rho$,
where~$\rho$ runs horizontally and
$\alpha$ vertically:
\begin{center}
\begin{longtable}{c||c|c|c|c|c|c|c|c|c|c}
\scriptsize  $\alpha$ vs. $\rho$ & \scriptsize $1$ & \scriptsize $2$ & \scriptsize $3$ & \scriptsize $4$ & \scriptsize $5$ & \scriptsize $6$ & \scriptsize $7$ & \scriptsize $8$ & \scriptsize $9$ & \scriptsize $10$
  \\\hline\hline \endhead
\scriptsize $ 1 $ & \scriptsize $ 5 $ & \scriptsize $ 7 $ & \scriptsize $ 3 $ & \scriptsize $ 1 $ & --- & --- & --- & --- & --- & ---
  \\\hline
\scriptsize $ 3/2 $ & \scriptsize $ 9 $ & \scriptsize $ 20 $ & \scriptsize $ 12 $ & \scriptsize $ 3 $ & --- & --- & --- & --- & --- & ---
  \\\hline
\scriptsize $ 5/3 $ & \scriptsize $ 13 $ & \scriptsize $ 40 $ & \scriptsize $ 23 $ & \scriptsize $ 3 $ & --- & --- & --- & --- & --- & ---
  \\\hline
\scriptsize $ 7/4 $ & \scriptsize $ 18 $ & \scriptsize $ 51 $ & \scriptsize $ 24 $ & \scriptsize $ 3 $ & --- & --- & --- & --- & --- & ---
  \\\hline
\scriptsize $ 9/5 $ & \scriptsize $ 20 $ & \scriptsize $ 53 $ & \scriptsize $ 24 $ & \scriptsize $ 3 $ & --- & --- & --- & --- & --- & ---
  \\\hline
\scriptsize $ 11/6 $ & \scriptsize $ 21 $ & \scriptsize $ 53 $ & \scriptsize $ 24 $ & \scriptsize $ 3 $ & --- & --- & --- & --- & --- & ---
  \\\hline
\scriptsize $ 2 $ & \scriptsize $ 42 $ & \scriptsize $ 181 $ & \scriptsize $ 202 $ & \scriptsize $ 74 $ & \scriptsize $ 5 $ & \scriptsize $ 1 $ & --- & --- & --- & ---
  \\\hline
\scriptsize $ 11/5 $ & \scriptsize $ 48 $ & \scriptsize $ 200 $ & \scriptsize $ 209 $ & \scriptsize $ 74 $ & \scriptsize $ 5 $ & \scriptsize $ 1 $ & --- & --- & --- & ---
  \\\hline
\scriptsize $ 7/3 $ & \scriptsize $ 68 $ & \scriptsize $ 357 $ & \scriptsize $ 479 $ & \scriptsize $ 205 $ & \scriptsize $ 16 $ & \scriptsize $ 1 $ & --- & --- & --- & ---
  \\\hline
\scriptsize $ 17/7 $ & \scriptsize $ 70 $ & \scriptsize $ 360 $ & \scriptsize $ 479 $ & \scriptsize $ 205 $ & \scriptsize $ 16 $ & \scriptsize $ 1 $ & --- & --- & --- & ---
  \\\hline
\scriptsize $ 5/2 $ & \scriptsize $ 94 $ & \scriptsize $ 630 $ & \scriptsize $ 1228 $ & \scriptsize $ 1039 $ & \scriptsize $ 449 $ & \scriptsize $ 158 $ & \scriptsize $ 37 $ & \scriptsize $ 11 $ & \scriptsize $ 1 $ & \scriptsize $ 1 $
  \\\hline
\scriptsize $ 23/9 $ & \scriptsize $ 95 $ & \scriptsize $ 632 $ & \scriptsize $ 1229 $ & \scriptsize $ 1039 $ & \scriptsize $ 449 $ & \scriptsize $ 158 $ & \scriptsize $ 37 $ & \scriptsize $ 11 $ & \scriptsize $ 1 $ & \scriptsize $ 1 $
  \\\hline
\scriptsize $ 13/5 $ & \scriptsize $ 115 $ & \scriptsize $ 804 $ & \scriptsize $ 1590 $ & \scriptsize $ 1334 $ & \scriptsize $ 549 $ & \scriptsize $ 175 $ & \scriptsize $ 38 $ & \scriptsize $ 11 $ & \scriptsize $ 1 $ & \scriptsize $ 1 $
  \\\hline
\scriptsize $ 8/3 $ & \scriptsize $ 120 $ & \scriptsize $ 834 $ & \scriptsize $ 1627 $ & \scriptsize $ 1350 $ & \scriptsize $ 552 $ & \scriptsize $ 175 $ & \scriptsize $ 38 $ & \scriptsize $ 11 $ & \scriptsize $ 1 $ & \scriptsize $ 1 $
  \\\hline
\scriptsize $ 19/7 $ & \scriptsize $ 134 $ & \scriptsize $ 913 $ & \scriptsize $ 1751 $ & \scriptsize $ 1418 $ & \scriptsize $ 568 $ & \scriptsize $ 176 $ & \scriptsize $ 38 $ & \scriptsize $ 11 $ & \scriptsize $ 1 $ & \scriptsize $ 1 $
  \\\hline
\scriptsize $ 11/4 $ & \scriptsize $ 154 $ & \scriptsize $ 1159 $ & \scriptsize $ 2532 $ & \scriptsize $ 2382 $ & \scriptsize $ 1129 $ & \scriptsize $ 352 $ & \scriptsize $ 68 $ & \scriptsize $ 14 $ & \scriptsize $ 1 $ & \scriptsize $ 1 $
  \\\hline
\scriptsize $ 25/9 $ & \scriptsize $ 160 $ & \scriptsize $ 1195 $ & \scriptsize $ 2583 $ & \scriptsize $ 2408 $ & \scriptsize $ 1134 $ & \scriptsize $ 352 $ & \scriptsize $ 68 $ & \scriptsize $ 14 $ & \scriptsize $ 1 $ & \scriptsize $ 1 $
  \\\hline
\scriptsize $ 31/11 $ & \scriptsize $ 163 $ & \scriptsize $ 1211 $ & \scriptsize $ 2601 $ & \scriptsize $ 2415 $ & \scriptsize $ 1135 $ & \scriptsize $ 352 $ & \scriptsize $ 68 $ & \scriptsize $ 14 $ & \scriptsize $ 1 $ & \scriptsize $ 1 $
  \\\hline
\scriptsize $ 17/6 $ & \scriptsize $ 174 $ & \scriptsize $ 1316 $ & \scriptsize $ 2886 $ & \scriptsize $ 2717 $ & \scriptsize $ 1269 $ & \scriptsize $ 377 $ & \scriptsize $ 69 $ & \scriptsize $ 14 $ & \scriptsize $ 1 $ & \scriptsize $ 1 $
  \\\hline
\scriptsize $ 37/13 $ & \scriptsize $ 176 $ & \scriptsize $ 1322 $ & \scriptsize $ 2890 $ & \scriptsize $ 2718 $ & \scriptsize $ 1269 $ & \scriptsize $ 377 $ & \scriptsize $ 69 $ & \scriptsize $ 14 $ & \scriptsize $ 1 $ & \scriptsize $ 1 $
  \\\hline
\scriptsize $ 43/15 $ & \scriptsize $ 177 $ & \scriptsize $ 1323 $ & \scriptsize $ 2890 $ & \scriptsize $ 2718 $ & \scriptsize $ 1269 $ & \scriptsize $ 377 $ & \scriptsize $ 69 $ & \scriptsize $ 14 $ & \scriptsize $ 1 $ & \scriptsize $ 1 $
  \\\hline
\scriptsize $ 23/8 $ & \scriptsize $ 186 $ & \scriptsize $ 1375 $ & \scriptsize $ 2991 $ & \scriptsize $ 2806 $ & \scriptsize $ 1302 $ & \scriptsize $ 382 $ & \scriptsize $ 69 $ & \scriptsize $ 14 $ & \scriptsize $ 1 $ & \scriptsize $ 1 $
  \\\hline
\scriptsize $ 29/10 $ & \scriptsize $ 189 $ & \scriptsize $ 1395 $ & \scriptsize $ 3029 $ & \scriptsize $ 2833 $ & \scriptsize $ 1310 $ & \scriptsize $ 383 $ & \scriptsize $ 69 $ & \scriptsize $ 14 $ & \scriptsize $ 1 $ & \scriptsize $ 1 $
  \\\hline
\scriptsize $ 35/12 $ & \scriptsize $ 191 $ & \scriptsize $ 1404 $ & \scriptsize $ 3040 $ & \scriptsize $ 2838 $ & \scriptsize $ 1311 $ & \scriptsize $ 383 $ & \scriptsize $ 69 $ & \scriptsize $ 14 $ & \scriptsize $ 1 $ & \scriptsize $ 1 $
  \\\hline
\scriptsize $ 41/14 $ & \scriptsize $ 192 $ & \scriptsize $ 1406 $ & \scriptsize $ 3041 $ & \scriptsize $ 2838 $ & \scriptsize $ 1311 $ & \scriptsize $ 383 $ & \scriptsize $ 69 $ & \scriptsize $ 14 $ & \scriptsize $ 1 $ & \scriptsize $ 1 $
  \\\hline
\scriptsize $ 3 $ & \scriptsize $ 355 $ & \scriptsize $ 3983 $ & \scriptsize $ 13454 $ & \scriptsize $ 17791 $ & \scriptsize $ 9653 $ & \scriptsize $ 2465 $ & \scriptsize $ 292 $ & \scriptsize $ 37 $ & \scriptsize $ 1 $ & \scriptsize $ 1 $
\end{longtable}
\end{center}
\end{corollary}

\begin{corollary}
We obtain the following table for the numbers
$\nu_{\mathrm{nontoric}}(\alpha,\rho)$
of families of $\alpha^{-1}$-log canonical non-toric
del Pezzo $\KK^*$-surfaces of Picard number~$\rho$,
where~$\rho$ runs horizontally and
$\alpha$ vertically:
\begin{center}
\begin{longtable}{c||c|c|c|c|c|c|c|c|c|c|c|c}
\scriptsize  $\alpha$ vs. $\rho $ & \scriptsize $1$ & \scriptsize $2$ & \scriptsize $3$ & \scriptsize $4$ & \scriptsize $5$ & \scriptsize $6$ & \scriptsize $7$ & \scriptsize $8$ & \scriptsize $9$ & \scriptsize $10$ & \scriptsize $11$ & \scriptsize $ 12 $
  \\\hline\hline \endhead
\scriptsize $ 1 $ & \scriptsize $ 13 $ & \scriptsize $ 12 $ & \scriptsize $ 7 $ & \scriptsize $ 2 $ --- & --- & --- & --- & --- & --- & --- & ---
  \\\hline
\scriptsize $ 3/2 $ & \scriptsize $ 19 $ & \scriptsize $ 18 $ & \scriptsize $ 18 $ & \scriptsize $ 13 $ & \scriptsize $ 4 $ & \scriptsize $ 1 $ & --- & --- & --- & --- & --- & ---
  \\\hline
\scriptsize $ 5/3 $ & \scriptsize $ 21 $ & \scriptsize $ 30 $ & \scriptsize $ 56 $ & \scriptsize $ 30 $ & \scriptsize $ 6 $ & \scriptsize $ 1 $ & --- & --- & --- & --- & --- & ---
  \\\hline
\scriptsize $ 7/4 $ & \scriptsize $ 23 $ & \scriptsize $ 51 $ & \scriptsize $ 81 $ & \scriptsize $ 32 $ & \scriptsize $ 6 $ & \scriptsize $ 1 $ & --- & --- & --- & --- & --- & ---
  \\\hline
\scriptsize $ 9/5 $ & \scriptsize $ 25 $ & \scriptsize $ 70 $ & \scriptsize $ 86 $ & \scriptsize $ 32 $ & \scriptsize $ 6 $ & \scriptsize $ 1 $ & --- & --- & --- & --- & --- & ---
  \\\hline
\scriptsize $ 11/6 $ & \scriptsize $ 28 $ & \scriptsize $ 79 $ & \scriptsize $ 86 $ & \scriptsize $ 32 $ & \scriptsize $ 6 $ & \scriptsize $ 1 $ & --- & --- & --- & --- & --- & ---
  \\\hline
\scriptsize $ 13/7 $ & \scriptsize $ 31 $ & \scriptsize $ 80 $ & \scriptsize $ 86 $ & \scriptsize $ 32 $ & \scriptsize $ 6 $ & \scriptsize $ 1 $ & --- & --- & --- & --- & --- & ---
  \\\hline
\scriptsize $ 15/8 $ & \scriptsize $ 32 $ & \scriptsize $ 80 $ & \scriptsize $ 86 $ & \scriptsize $ 32 $ & \scriptsize $ 6 $ & \scriptsize $ 1 $ & --- & --- & --- & --- & --- & ---
  \\\hline
\scriptsize $ 2 $ & \scriptsize $ 61 $ & \scriptsize $ 253 $ & \scriptsize $ 471 $ & \scriptsize $ 318 $ & \scriptsize $ 107 $ & \scriptsize $ 22 $ & \scriptsize $ 4 $ & \scriptsize $ 1 $ & --- & --- & --- & ---
  \\\hline
\scriptsize $ 11/5 $ & \scriptsize $ 67 $ & \scriptsize $ 307 $ & \scriptsize $ 524 $ & \scriptsize $ 327 $ & \scriptsize $ 107 $ & \scriptsize $ 22 $ & \scriptsize $ 4 $ & \scriptsize $ 1 $ & --- & --- & --- & ---
  \\\hline
\scriptsize $ 7/3 $ & \scriptsize $ 80 $ & \scriptsize $ 426 $ & \scriptsize $ 950 $ & \scriptsize $ 829 $ & \scriptsize $ 302 $ & \scriptsize $ 55 $ & \scriptsize $ 8 $ & \scriptsize $ 1 $ & --- & --- & --- & ---
  \\\hline
\scriptsize $ 17/7 $ & \scriptsize $ 84 $ & \scriptsize $ 448 $ & \scriptsize $ 954 $ & \scriptsize $ 829 $ & \scriptsize $ 302 $ & \scriptsize $ 55 $ & \scriptsize $ 8 $ & \scriptsize $ 1 $ & --- & --- & --- & ---
  \\\hline
\scriptsize $ 5/2 $ & \scriptsize $ 104 $ & \scriptsize $ 639 $ & \scriptsize $ 1671 $ & \scriptsize $ 2179 $ & \scriptsize $ 1602 $ & \scriptsize $ 736 $ & \scriptsize $ 239 $ & \scriptsize $ 58 $ & \scriptsize $ 11 $ & \scriptsize $ 2 $ & --- & ---
  \\\hline
\scriptsize $ 23/9 $ & \scriptsize $ 106 $ & \scriptsize $ 645 $ & \scriptsize $ 1672 $ & \scriptsize $ 2179 $ & \scriptsize $ 1602 $ & \scriptsize $ 736 $ & \scriptsize $ 239 $ & \scriptsize $ 58 $ & \scriptsize $ 11 $ & \scriptsize $ 2 $ & --- & ---
  \\\hline
\scriptsize $ 13/5 $ & \scriptsize $ 115 $ & \scriptsize $ 817 $ & \scriptsize $ 2295 $ & \scriptsize $ 2906 $ & \scriptsize $ 1978 $ & \scriptsize $ 849 $ & \scriptsize $ 259 $ & \scriptsize $ 60 $ & \scriptsize $ 11 $ & \scriptsize $ 2 $ & --- & ---
  \\\hline
\scriptsize $ 8/3 $ & \scriptsize $ 118 $ & \scriptsize $ 858 $ & \scriptsize $ 2371 $ & \scriptsize $ 2954 $ & \scriptsize $ 1992 $ & \scriptsize $ 851 $ & \scriptsize $ 259 $ & \scriptsize $ 60 $ & \scriptsize $ 11 $ & \scriptsize $ 2 $ & --- & ---
  \\\hline
\scriptsize $ 19/7 $ & \scriptsize $ 128 $ & \scriptsize $ 991 $ & \scriptsize $ 2645 $ & \scriptsize $ 3147 $ & \scriptsize $ 2055 $ & \scriptsize $ 863 $ & \scriptsize $ 260 $ & \scriptsize $ 60 $ & \scriptsize $ 11 $ & \scriptsize $ 2 $ & --- & ---
  \\\hline
\scriptsize $ 11/4 $ & \scriptsize $ 133 $ & \scriptsize $ 1124 $ & \scriptsize $ 3514 $ & \scriptsize $ 4876 $ & \scriptsize $ 3562 $ & \scriptsize $ 1538 $ & \scriptsize $ 447 $ & \scriptsize $ 93 $ & \scriptsize $ 14 $ & \scriptsize $ 2 $ & --- & ---
  \\\hline
\scriptsize $ 25/9 $ & \scriptsize $ 136 $ & \scriptsize $ 1185 $ & \scriptsize $ 3654 $ & \scriptsize $ 4952 $ & \scriptsize $ 3577 $ & \scriptsize $ 1540 $ & \scriptsize $ 447 $ & \scriptsize $ 93 $ & \scriptsize $ 14 $ & \scriptsize $ 2 $ & --- & ---
  \\\hline
\scriptsize $ 14/5 $ & \scriptsize $ 137 $ & \scriptsize $ 1186 $ & \scriptsize $ 3654 $ & \scriptsize $ 4952 $ & \scriptsize $ 3577 $ & \scriptsize $ 1540 $ & \scriptsize $ 447 $ & \scriptsize $ 93 $ & \scriptsize $ 14 $ & \scriptsize $ 2 $ & --- & ---
  \\\hline
\scriptsize $ 31/11 $ & \scriptsize $ 142 $ & \scriptsize $ 1225 $ & \scriptsize $ 3709 $ & \scriptsize $ 4969 $ & \scriptsize $ 3577 $ & \scriptsize $ 1540 $ & \scriptsize $ 447 $ & \scriptsize $ 93 $ & \scriptsize $ 14 $ & \scriptsize $ 2 $ & --- & ---
  \\\hline
\scriptsize $ 17/6 $ & \scriptsize $ 146 $ & \scriptsize $ 1320 $ & \scriptsize $ 4159 $ & \scriptsize $ 5653 $ & \scriptsize $ 4007 $ & \scriptsize $ 1670 $ & \scriptsize $ 468 $ & \scriptsize $ 94 $ & \scriptsize $ 14 $ & \scriptsize $ 2 $ & --- & ---
  \\\hline
\scriptsize $ 37/13 $ & \scriptsize $ 150 $ & \scriptsize $ 1338 $ & \scriptsize $ 4172 $ & \scriptsize $ 5655 $ & \scriptsize $ 4007 $ & \scriptsize $ 1670 $ & \scriptsize $ 468 $ & \scriptsize $ 94 $ & \scriptsize $ 14 $ & \scriptsize $ 2 $ & --- & ---
  \\\hline
\scriptsize $ 43/15 $ & \scriptsize $ 151 $ & \scriptsize $ 1342 $ & \scriptsize $ 4174 $ & \scriptsize $ 5655 $ & \scriptsize $ 4007 $ & \scriptsize $ 1670 $ & \scriptsize $ 468 $ & \scriptsize $ 94 $ & \scriptsize $ 14 $ & \scriptsize $ 2 $ & --- & ---
  \\\hline
\scriptsize $ 23/8 $ & \scriptsize $ 157 $ & \scriptsize $ 1421 $ & \scriptsize $ 4414 $ & \scriptsize $ 5898 $ & \scriptsize $ 4107 $ & \scriptsize $ 1688 $ & \scriptsize $ 470 $ & \scriptsize $ 94 $ & \scriptsize $ 14 $ & \scriptsize $ 2 $ & --- & ---
  \\\hline
\scriptsize $ 49/17 $ & \scriptsize $ 158 $ & \scriptsize $ 1421 $ & \scriptsize $ 4414 $ & \scriptsize $ 5898 $ & \scriptsize $ 4107 $ & \scriptsize $ 1688 $ & \scriptsize $ 470 $ & \scriptsize $ 94 $ & \scriptsize $ 14 $ & \scriptsize $ 2 $ & --- & ---
  \\\hline
\scriptsize $ 29/10 $ & \scriptsize $ 160 $ & \scriptsize $ 1470 $ & \scriptsize $ 4524 $ & \scriptsize $ 5972 $ & \scriptsize $ 4125 $ & \scriptsize $ 1688 $ & \scriptsize $ 470 $ & \scriptsize $ 94 $ & \scriptsize $ 14 $ & \scriptsize $ 2 $ & --- & ---
  \\\hline
\scriptsize $ 35/12 $ & \scriptsize $ 164 $ & \scriptsize $ 1496 $ & \scriptsize $ 4559 $ & \scriptsize $ 5986 $ & \scriptsize $ 4127 $ & \scriptsize $ 1688 $ & \scriptsize $ 470 $ & \scriptsize $ 94 $ & \scriptsize $ 14 $ & \scriptsize $ 2 $ & --- & ---
  \\\hline
\scriptsize $ 41/14 $ & \scriptsize $ 166 $ & \scriptsize $ 1507 $ & \scriptsize $ 4568 $ & \scriptsize $ 5988 $ & \scriptsize $ 4127 $ & \scriptsize $ 1688 $ & \scriptsize $ 470 $ & \scriptsize $ 94 $ & \scriptsize $ 14 $ & \scriptsize $ 2 $ & --- & ---
  \\\hline
\scriptsize $ 47/16 $ & \scriptsize $ 168 $ & \scriptsize $ 1511 $ & \scriptsize $ 4569 $ & \scriptsize $ 5988 $ & \scriptsize $ 4127 $ & \scriptsize $ 1688 $ & \scriptsize $ 470 $ & \scriptsize $ 94 $ & \scriptsize $ 14 $ & \scriptsize $ 2 $ & --- & ---
  \\\hline
\scriptsize $ 53/18 $ & \scriptsize $ 169 $ & \scriptsize $ 1512 $ & \scriptsize $ 4569 $ & \scriptsize $ 5988 $ & \scriptsize $ 4127 $ & \scriptsize $ 1688 $ & \scriptsize $ 470 $ & \scriptsize $ 94 $ & \scriptsize $ 14 $ & \scriptsize $ 2 $ & --- & ---
  \\\hline
\scriptsize $ 3 $ & \scriptsize $ 318 $ & \scriptsize $ 3632 $ & \scriptsize $ 15174 $ & \scriptsize $ 26654 $ & \scriptsize $ 22648 $ & \scriptsize $ 10024 $ & \scriptsize $ 2731 $ & \scriptsize $ 577 $ & \scriptsize $ 100 $ & \scriptsize $ 11 $ & \scriptsize $ 2 $ & \scriptsize $ 1 $
\end{longtable}
\end{center}
\end{corollary}

\begin{corollary}
We obtain the following table for the numbers
$\nu(\alpha,\rho)$
of families of $\alpha^{-1}$-log canonical 
del Pezzo surfaces with torus action
of Picard number~$\rho$,
where~$\rho$ runs horizontally and
$\alpha$ vertically:
\begin{center}
\begin{longtable}{c||c|c|c|c|c|c|c|c|c|c|c|c}
\scriptsize  $\alpha$ vs. $\rho $ & \scriptsize $1$ & \scriptsize $2$ & \scriptsize $3$ & \scriptsize $4$ & \scriptsize $5$ & \scriptsize $6$ & \scriptsize $7$ & \scriptsize $8$ & \scriptsize $9$ & \scriptsize $10$ & \scriptsize $11$ & \scriptsize $12$ 
  \\\hline\hline  \endhead
\scriptsize $ 1 $ & \scriptsize $ 18 $ & \scriptsize $ 19 $ & \scriptsize $ 10 $ & \scriptsize $ 3 $ & --- & --- & --- & --- & --- & --- & --- & --- 
  \\\hline
\scriptsize $ 3/2 $ & \scriptsize $ 28 $ & \scriptsize $ 38 $ & \scriptsize $ 30 $ & \scriptsize $ 16 $ & \scriptsize $ 4 $ & \scriptsize $ 1 $ & --- & --- & --- & --- & --- & --- 
  \\\hline
\scriptsize $ 5/3 $ & \scriptsize $ 34 $ & \scriptsize $ 70 $ & \scriptsize $ 79 $ & \scriptsize $ 33 $ & \scriptsize $ 6 $ & \scriptsize $ 1 $ & --- & --- & --- & --- & --- & --- 
  \\\hline
\scriptsize $ 7/4 $ & \scriptsize $ 41 $ & \scriptsize $ 102 $ & \scriptsize $ 105 $ & \scriptsize $ 35 $ & \scriptsize $ 6 $ & \scriptsize $ 1 $ & --- & --- & --- & --- & --- & --- 
  \\\hline
\scriptsize $ 9/5 $ & \scriptsize $ 45 $ & \scriptsize $ 123 $ & \scriptsize $ 110 $ & \scriptsize $ 35 $ & \scriptsize $ 6 $ & \scriptsize $ 1 $ & --- & --- & --- & --- & --- & --- 
  \\\hline
\scriptsize $ 11/6 $ & \scriptsize $ 49 $ & \scriptsize $ 132 $ & \scriptsize $ 110 $ & \scriptsize $ 35 $ & \scriptsize $ 6 $ & \scriptsize $ 1 $ & --- & --- & --- & --- & --- & --- 
  \\\hline
\scriptsize $ 13/7 $ & \scriptsize $ 52 $ & \scriptsize $ 133 $ & \scriptsize $ 110 $ & \scriptsize $ 35 $ & \scriptsize $ 6 $ & \scriptsize $ 1 $ & --- & --- & --- & --- & --- & --- 
  \\\hline
\scriptsize $ 15/8 $ & \scriptsize $ 53 $ & \scriptsize $ 133 $ & \scriptsize $ 110 $ & \scriptsize $ 35 $ & \scriptsize $ 6 $ & \scriptsize $ 1 $ & --- & --- & --- & --- & --- & --- 
  \\\hline
\scriptsize $ 2 $ & \scriptsize $ 103 $ & \scriptsize $ 434 $ & \scriptsize $ 673 $ & \scriptsize $ 392 $ & \scriptsize $ 112 $ & \scriptsize $ 23 $ & \scriptsize $ 4 $ & \scriptsize $ 1 $ & --- & --- & --- & --- 
  \\\hline
\scriptsize $ 11/5 $ & \scriptsize $ 115 $ & \scriptsize $ 507 $ & \scriptsize $ 733 $ & \scriptsize $ 401 $ & \scriptsize $ 112 $ & \scriptsize $ 23 $ & \scriptsize $ 4 $ & \scriptsize $ 1 $ & --- & --- & --- & --- 
  \\\hline
\scriptsize $ 7/3 $ & \scriptsize $ 148 $ & \scriptsize $ 783 $ & \scriptsize $ 1429 $ & \scriptsize $ 1034 $ & \scriptsize $ 318 $ & \scriptsize $ 56 $ & \scriptsize $ 8 $ & \scriptsize $ 1 $ & --- & --- & --- & --- 
  \\\hline
\scriptsize $ 17/7 $ & \scriptsize $ 154 $ & \scriptsize $ 808 $ & \scriptsize $ 1433 $ & \scriptsize $ 1034 $ & \scriptsize $ 318 $ & \scriptsize $ 56 $ & \scriptsize $ 8 $ & \scriptsize $ 1 $ & --- & --- & --- & --- 
  \\\hline
\scriptsize $ 5/2 $ & \scriptsize $ 198 $ & \scriptsize $ 1269 $ & \scriptsize $ 2899 $ & \scriptsize $ 3218 $ & \scriptsize $ 2051 $ & \scriptsize $ 894 $ & \scriptsize $ 276 $ & \scriptsize $ 69 $ & \scriptsize $ 12 $ & \scriptsize $ 3 $ & --- & --- 
  \\\hline
\scriptsize $ 23/9 $ & \scriptsize $ 201 $ & \scriptsize $ 1277 $ & \scriptsize $ 2901 $ & \scriptsize $ 3218 $ & \scriptsize $ 2051 $ & \scriptsize $ 894 $ & \scriptsize $ 276 $ & \scriptsize $ 69 $ & \scriptsize $ 12 $ & \scriptsize $ 3 $ & --- & --- 
  \\\hline
\scriptsize $ 13/5 $ & \scriptsize $ 230 $ & \scriptsize $ 1621 $ & \scriptsize $ 3885 $ & \scriptsize $ 4240 $ & \scriptsize $ 2527 $ & \scriptsize $ 1024 $ & \scriptsize $ 297 $ & \scriptsize $ 71 $ & \scriptsize $ 12 $ & \scriptsize $ 3 $ & --- & --- 
  \\\hline
\scriptsize $ 8/3 $ & \scriptsize $ 238 $ & \scriptsize $ 1692 $ & \scriptsize $ 3998 $ & \scriptsize $ 4304 $ & \scriptsize $ 2544 $ & \scriptsize $ 1026 $ & \scriptsize $ 297 $ & \scriptsize $ 71 $ & \scriptsize $ 12 $ & \scriptsize $ 3 $ & --- & --- 
  \\\hline
\scriptsize $ 19/7 $ & \scriptsize $ 262 $ & \scriptsize $ 1904 $ & \scriptsize $ 4396 $ & \scriptsize $ 4565 $ & \scriptsize $ 2623 $ & \scriptsize $ 1039 $ & \scriptsize $ 298 $ & \scriptsize $ 71 $ & \scriptsize $ 12 $ & \scriptsize $ 3 $ & --- & --- 
  \\\hline
\scriptsize $ 11/4 $ & \scriptsize $ 287 $ & \scriptsize $ 2283 $ & \scriptsize $ 6046 $ & \scriptsize $ 7258 $ & \scriptsize $ 4691 $ & \scriptsize $ 1890 $ & \scriptsize $ 515 $ & \scriptsize $ 107 $ & \scriptsize $ 15 $ & \scriptsize $ 3 $ & --- & --- 
  \\\hline
\scriptsize $ 25/9 $ & \scriptsize $ 296 $ & \scriptsize $ 2380 $ & \scriptsize $ 6237 $ & \scriptsize $ 7360 $ & \scriptsize $ 4711 $ & \scriptsize $ 1892 $ & \scriptsize $ 515 $ & \scriptsize $ 107 $ & \scriptsize $ 15 $ & \scriptsize $ 3 $ & --- & --- 
  \\\hline
\scriptsize $ 14/5 $ & \scriptsize $ 297 $ & \scriptsize $ 2381 $ & \scriptsize $ 6237 $ & \scriptsize $ 7360 $ & \scriptsize $ 4711 $ & \scriptsize $ 1892 $ & \scriptsize $ 515 $ & \scriptsize $ 107 $ & \scriptsize $ 15 $ & \scriptsize $ 3 $ & --- & --- 
  \\\hline
\scriptsize $ 31/11 $ & \scriptsize $ 305 $ & \scriptsize $ 2436 $ & \scriptsize $ 6310 $ & \scriptsize $ 7384 $ & \scriptsize $ 4712 $ & \scriptsize $ 1892 $ & \scriptsize $ 515 $ & \scriptsize $ 107 $ & \scriptsize $ 15 $ & \scriptsize $ 3 $ & --- & --- 
  \\\hline
\scriptsize $ 17/6 $ & \scriptsize $ 320 $ & \scriptsize $ 2636 $ & \scriptsize $ 7045 $ & \scriptsize $ 8370 $ & \scriptsize $ 5276 $ & \scriptsize $ 2047 $ & \scriptsize $ 537 $ & \scriptsize $ 108 $ & \scriptsize $ 15 $ & \scriptsize $ 3 $ & --- & --- 
  \\\hline
\scriptsize $ 37/13 $ & \scriptsize $ 326 $ & \scriptsize $ 2660 $ & \scriptsize $ 7062 $ & \scriptsize $ 8373 $ & \scriptsize $ 5276 $ & \scriptsize $ 2047 $ & \scriptsize $ 537 $ & \scriptsize $ 108 $ & \scriptsize $ 15 $ & \scriptsize $ 3 $ & --- & --- 
  \\\hline
\scriptsize $ 43/15 $ & \scriptsize $ 328 $ & \scriptsize $ 2665 $ & \scriptsize $ 7064 $ & \scriptsize $ 8373 $ & \scriptsize $ 5276 $ & \scriptsize $ 2047 $ & \scriptsize $ 537 $ & \scriptsize $ 108 $ & \scriptsize $ 15 $ & \scriptsize $ 3 $ & --- & --- 
  \\\hline
\scriptsize $ 23/8 $ & \scriptsize $ 343 $ & \scriptsize $ 2796 $ & \scriptsize $ 7405 $ & \scriptsize $ 8704 $ & \scriptsize $ 5409 $ & \scriptsize $ 2070 $ & \scriptsize $ 539 $ & \scriptsize $ 108 $ & \scriptsize $ 15 $ & \scriptsize $ 3 $ & --- & --- 
  \\\hline
\scriptsize $ 49/17 $ & \scriptsize $ 344 $ & \scriptsize $ 2796 $ & \scriptsize $ 7405 $ & \scriptsize $ 8704 $ & \scriptsize $ 5409 $ & \scriptsize $ 2070 $ & \scriptsize $ 539 $ & \scriptsize $ 108 $ & \scriptsize $ 15 $ & \scriptsize $ 3 $ & --- & --- 
  \\\hline
\scriptsize $ 29/10 $ & \scriptsize $ 349 $ & \scriptsize $ 2865 $ & \scriptsize $ 7553 $ & \scriptsize $ 8805 $ & \scriptsize $ 5435 $ & \scriptsize $ 2071 $ & \scriptsize $ 539 $ & \scriptsize $ 108 $ & \scriptsize $ 15 $ & \scriptsize $ 3 $ & --- & --- 
  \\\hline
\scriptsize $ 35/12 $ & \scriptsize $ 355 $ & \scriptsize $ 2900 $ & \scriptsize $ 7599 $ & \scriptsize $ 8824 $ & \scriptsize $ 5438 $ & \scriptsize $ 2071 $ & \scriptsize $ 539 $ & \scriptsize $ 108 $ & \scriptsize $ 15 $ & \scriptsize $ 3 $ & --- & --- 
  \\\hline
\scriptsize $ 41/14 $ & \scriptsize $ 358 $ & \scriptsize $ 2913 $ & \scriptsize $ 7609 $ & \scriptsize $ 8826 $ & \scriptsize $ 5438 $ & \scriptsize $ 2071 $ & \scriptsize $ 539 $ & \scriptsize $ 108 $ & \scriptsize $ 15 $ & \scriptsize $ 3 $ & --- & --- 
  \\\hline
\scriptsize $ 47/16 $ & \scriptsize $ 360 $ & \scriptsize $ 2917 $ & \scriptsize $ 7610 $ & \scriptsize $ 8826 $ & \scriptsize $ 5438 $ & \scriptsize $ 2071 $ & \scriptsize $ 539 $ & \scriptsize $ 108 $ & \scriptsize $ 15 $ & \scriptsize $ 3 $ & --- & --- 
  \\\hline
\scriptsize $ 53/18 $ & \scriptsize $ 361 $ & \scriptsize $ 2918 $ & \scriptsize $ 7610 $ & \scriptsize $ 8826 $ & \scriptsize $ 5438 $ & \scriptsize $ 2071 $ & \scriptsize $ 539 $ & \scriptsize $ 108 $ & \scriptsize $ 15 $ & \scriptsize $ 3 $ & --- & --- 
  \\\hline 
\scriptsize $ 3 $ & \scriptsize $ 673 $ & \scriptsize $ 7615 $ & \scriptsize $ 28628 $ & \scriptsize $ 44445 $ & \scriptsize $ 32301 $ & \scriptsize $ 12489 $ & \scriptsize $ 3023 $ & \scriptsize $ 614 $ & \scriptsize $ 101 $ & \scriptsize $ 12 $ & \scriptsize $ 2 $ & \scriptsize $ 1 $ 
\end{longtable}
\end{center}
\end{corollary}


\end{document}